\documentclass[11pt,a4paper]{article}
\usepackage{preambule}

\title{Scaling limit of a one-dimensional polymer in a repulsive i.i.d.\ environment.}

\author{Nicolas Bouchot}

\AtEndDocument{\bigskip{\footnotesize%
  \textsc{LPSM, Sorbonne Universit\'e, UMR 8001, Campus Pierre et Marie Curie, Bo\^ite courrier 158, 4 Place Jussieu, 75252 Paris Cedex 05, France.} \par  
  \textit{E-mail address}: \texttt{nicolas.bouchot@sorbonne-universite.fr} }}

\begin{document}

\pagestyle{plain}

\maketitle

\begin{abstract}
    \noindent The purpose of this paper is to study a one-dimensional polymer penalized by its range and placed in a random environment $\omega$. The law of the simple symmetric random walk up to time $n$ is modified by the exponential of the sum of $\beta \omega_z - h$ sitting on its range, with~$h$ and $\beta$ positive parameters. It is known that, at first order, the polymer folds itself to a segment of optimal size $c_h n^{1/3}$ with $c_h = \pi^{2/3} h^{-1/3}$. Here we study how disorder influences finer quantities. If the random variables $\omega_z$ are i.i.d.\ with a finite second moment, we prove that the left-most point of the range is located near $-u_* n^{1/3}$, where $u_* \in [0,c_h]$ is a constant that only depends on the disorder. This contrasts with the homogeneous model (\textit{i.e.}\ when $\beta=0$), where the left-most point has a random location between $-c_h n^{1/3}$ and $0$.
    With an additional moment assumption, we are able to show that the left-most point of the range is at distance $\mathcal U n^{2/9}$ from $-u_* n^{1/3}$ and the right-most point at distance $\mathcal V n^{2/9}$ from $(c_h-u_*) n^{1/3}$. Here again, $\mathcal{U}$ and $\mathcal{V}$ are constants that depend only on $\omega$.


\end{abstract}

\section{Introduction}

We study a simple symmetric random walk $(S_k)_{k \geq 0}$ on $\ZZ$, starting from $0$, with law $\mathbf{P}$. Let $\omega = (\omega_z)_{z \in \ZZ}$ be a collection of i.i.d. random variables with law $\PP$, independent from the random walk $S$, which we will call \emph{environment} or \emph{field}. We also assume that $\mathbb E[\omega_0]=0$ and $\mathbb{E}[\omega_0^2]=1$.
For $h > 0$, $\beta > 0$ and a given realization of the field $\omega$, we define the following Gibbs transformation of $\mathbf{P}$, called the (quenched) \emph{polymer measure}:
\[ \dd \mathbf{P}_{n,h}^{\omega,\beta}(S) \defeq \frac{1}{Z_{n,h}^{\omega,\beta}} \exp \Big( \sum_{z \in \mathcal{R}_n} \big(\beta \omega_z - h\big) \Big) \dd \mathbf{P}(S), \]
where $\mathcal{R}_n = \mathcal{R}_n(S) \defeq \big\{ S_0, \dots, S_n \big\}$ is the range of the random walk up to time $n$, and
\[ Z_{n,h}^{\omega,\beta} \defeq \mathbf{E}\Big[ \exp \Big( \sum_{z \in \mathcal{R}_n} \big(\beta \omega_z - h\big) \Big) \Big] = \mathbf{E}\Big[ \exp \Big( \beta \sum_{z \in \mathcal{R}_n} \omega_z - h |\mathcal{R}_n| \Big) \Big] \]
is the partition function, such that $\mathbf{P}_{n,h}^{\omega,\beta}$ is a (random) probability measure on the space of trajectories of length $n$. In other words, the polymer measure $\mathbf{P}_{n,h}^{\omega,\beta}$ penalizes trajectories by their range and rewards visits to sites where the field $\omega$ takes greater values.

In this setting, the disorder term $\sum_{z \in \mathcal{R}_n} \omega_z$ is typically of order $|\mathcal{R}_n|^{1/2}$: one can prove that\footnote{In the rest of the paper we shall use the standard Landau notation:
	as $x\to a$, we write $g(x)\sim f(x)$ if $\lim_{x\to a} \frac{g(x)}{f(x)} =1$, $g(x) = \bar{o}(f(x))$ if $\lim_{x\to a} \frac{g(x)}{f(x)} = 0$,$g(x) = \grdO(f(x))$ if  $\limsup_{x\to a}\big| \frac{g(x)}{f(x)} \big| < +\infty$ and $f \asymp g$ if $g(x) = \grdO(f(x))$ and $f(x) = \grdO(g(x))$.} $\beta \sum_{z \in \mathcal{R}_n} \omega_z - h |\mathcal{R}_n| \sim  - h |\mathcal{R}_n|$ for $\PP$-almost all~$\omega$, see~\cite{berger2020one}. 
Thus, disorder does not sufficiently impact the behavior of the polymer on a first approximation, which is seen in Theorem~\ref{thm0} below.
We introduce the following notation: fix $\omega$ and let $\xi_n^\omega$ be $E$-valued random variables, with $(E,d)$ a metric space. Consider $\xi^\omega \in E$, we write
\[ \xi_n^\omega \xrightarrow[n \to \infty]{\mathbf{P}_{n,h}^{\omega, \beta}} \xi^\omega \quad \Longleftrightarrow \quad \forall \eps > 0, \lim_{n \to \infty} \mathbf{P}_{n,h}^{\omega, \beta} \left( d(\xi_n^\omega, \xi^\omega) > \eps \right) = 0 \, . \]
We will say that ``$\xi_n^\omega$ converges in $\mathbf{P}_{n,h}^{\omega, \beta}$-probability'' even if $\mathbf{P}_{n,h}^{\omega, \beta}$ depends on $n$. If this holds for $\PP$-almost all $\omega$, we will say that $\xi_n$ converges in $\mathbf{P}_{n,h}^{\omega, \beta}$-probability, $\PP$-almost surely. In our results, we will take $(E,d)$ to be $(\RR, |\cdot|)$, or the closed bounded subsets of $\RR^d$ endowed with the Hausdorff distance.

Let us express the results of \cite{berger2020one} with this notation, which states that $|R_n| \sim c_h n^{1/3}$ for $\PP$-almost all realization of $\omega$.

\begin{theorem}[{\cite[Theorem 1.2-(1.a)]{berger2020one}}]
	\label{thm0}
	For all $h > 0$, define $c_h \defeq (\pi^2 h^{-1})^{1/3}$. Then, for any $h,\beta>0$, $\PP$-almost surely we have the following convergence
	\begin{equation}\label{result-quentin}
		\lim_{n\to \infty} \frac{1}{n^{1/3}} \log Z_{n,h}^{\omega,\beta} = -\frac{3}{2} (\pi h)^{2/3}, 
		\qquad  n^{-1/3}|\mathcal{R}_n|\xrightarrow[n \to \infty]{\mathbf{P}_{n,h}^{\omega, \beta}}  c_h \,.
	\end{equation}
\end{theorem}

The main goal of this paper is to extract further information on the polymer, notably on the location of the segment where the random walk is folded, or on how $|\mathcal{R}_n|$ fluctuates at lower scales than $n^{1/3}$.

To do so, we will prove the following expansion of the partition function: there are random variables $u_*, \mathcal{U},\mathcal{V}$ and processes $X, \mathcal{Y}$ such that
\[ \log Z_{n,h}^{\omega,\beta} = -\frac32 h c_h n^{1/3} + \beta X_{u_*} n^{1/6} + \frac{\beta}{\sqrt{2}} \left( \mathcal{Y}_{\mathcal{U},\mathcal{V}} - \frac{3\pi^2}{\beta c_h^4 \sqrt{2}} \big( \mathcal{U} + \mathcal{V} \big)^2 \right) n^{1/9} (1+\bar{o}(1)) \]
holds $\PP$-a.s.\ with the $\bar{o}(1)$ going to $0$ in $\mathbf P_{n,h}^{\beta,\omega}$-probability. The random variables are given by variational problems in Theorems \ref{th-1/6}, \ref{th-1/9}, which are the main results of this paper. Thanks to this expansion, we also get a precise description of $\mathcal{R}_n$ at scale $n^{2/9}$ under $\mathbf{P}_{n,h}^{\omega, \beta}$, that is
\[ \mathcal{R}_n \approx \llbracket -u_*n^{1/3} + \mathcal{U}n^{2/9} , (c_h-u_*)n^{1/3} + \mathcal{V}n^{2/9} \rrbracket \, . \]

\subsection{About the homogeneous setting}

Since we are working in dimension one, we make use of the fact that the range is entirely determined by the position of its extremal points, meaning that $\mathcal{R}_n$ is exactly the segment $\llbracket M_n^-,M_n^+ \rrbracket$, where $M_n^- \defeq \min_{0\leq k \leq n}S_k$ and $M_n^+ \defeq \max_{0\leq k \leq n} S_k$.

We will also adopt the following notation:
\begin{equation}\label{eq:notation}
	T_n \defeq M_n^+ - M_n^- = |\mathcal{R}_n| - 1 \,, \qquad T_n^* \defeq \left(\frac{n \pi^2}{h} \right)^{1/3} = c_h n^{1/3} \,, \qquad  \Delta_n \defeq T_n - T_n^*.
\end{equation} 
Hence, $T_n$ is the size of the range and $T_n^*$ is the typical size of the range at scale $n^{1/3}$ under $\mathbf{P}_{n,h}^{\omega, \beta}$ that appears in \eqref{result-quentin}.

In the homogeneous setting, that is when $\beta = 0$, it is proven in \cite{bouchot1} that the location of the left-most point is random (on the scale $n^{1/3}$) with a density proportional to $\sin(\pi u/c_h)$.
As far as the size of the range $T_n$ is concerned, it is shown to have Gaussian fluctuations. In fact, \cite{bouchot1} treats the case of a parameter $h=h_n$ that may depend on the length of the polymer: in this case, fluctuations vanish when the penalty strength $h_n$ is too high. We state the full result for the sake of completeness.

\begin{theorem}[{\cite[Theorem 1.1]{bouchot1}}]
	Recall the notations of \eqref{eq:notation} and replace $h$ by $h_n$ in the definition of $T_n^*$. Then for $\beta = 0$ and all $\omega$, we have the following results:
	\begin{itemize}
		\item Assume that $h_n \geq n^{-1/2} (\log n)^{3/2}$ and $\lim\limits_{n\to\infty} n^{-1/4} h_n = 0$. Let $a_n \defeq \frac{1}{\sqrt{3}} \left( \frac{n \pi^2}{h_n^4}\right)^{1/6}$, which is such that $\lim_{n\to\infty} a_n =+\infty$. Then for any $r < s$ and any $0 \leq a < b \leq 1$,
		\[ 
		\lim_{n\to\infty} \mathbf{P}_{n,h_n}^{\omega, 0} \left( r \leq \frac{\Delta_n}{a_n} \leq s \, ; \, a \leq \frac{|M_n^-|}{T_n^*} \leq b \right) =\frac{\sqrt{\pi}}{2\sqrt{2}} \int_r^s e^{- \frac{u^2}{2}} \dd u  \int_a^b \sin(\pi v) \, \dd v \,.
		\]
		\item Assume that $\lim\limits_{n\to\infty} n^{-1/4}  h_n =+\infty$ and $\lim\limits_{n\to\infty} n^{-1} h_n = 0$. Then we have for any Borel set $B \subseteq [0,1]$
		\[
		\lim_{n\to\infty}\mathbf{P}_{n,h_n}^{\omega, 0} \big( T_n - \lfloor T_n^*-2 \rfloor \not\in \mathset{0,1} \big) = 0, 
		\quad \lim_{n \to \infty} \mathbf{P}_{n,h_n}^{\omega,0} \Big( \frac{|M_n^-|}{T_n^*} \in B \Big) = \frac{\pi}{2} \int_B \sin(\pi v) \, \dd v.
		\]
	\end{itemize}
\end{theorem}

We will see that the disordered model displays a very different behavior: the location of the left-most and right-most points are $\mathbf{P}$-deterministic, in the sense that they are completely determined by the disorder field~$\omega$ (at least for the first two orders).

\subsection{First convergence result}

Akin to \cite{berger2020one}, we define the following quantities: for any $j \geq 0$ for which the sum is not empty,
\[
\Sigma^+_j(\omega) \defeq \sum_{z=0}^j \omega_z \, , \qquad \Sigma^-_j(\omega) \defeq \sum_{z=1}^j \omega_{-z}\,. 
\]
Using Skorokhod's embedding theorem (see \cite[Chapter~7.2]{skorokhod1982studies} and Theorem \ref{th-skorokhod} below) we can define on the same probability space a coupling $\hat{\omega} = \hat{\omega}^{(n)}$ of $\omega$ and two independent standard Brownian motions $X^{(1)}$ and $X^{(2)}$ such that for each $n$, $\hat{\omega}^{(n)}$ has the same law as the environment $\omega$ and
\[ \left( \frac{1}{n^{1/6}} \Sigma^-_{un^{1/3}}(\hat{\omega}) \right)_{u \geq 0} \xrightarrow[n \to \infty]{a.s.} \left( X^{(1)}_u(\hat{\omega})  \right)_{u \geq 0} \,, \hspace{0.5cm} \left( \frac{1}{n^{1/6}} \Sigma^+_{vn^{1/3}}(\hat{\omega}) \right)_{v \geq 0} \xrightarrow[n \to \infty]{a.s.} \left( X^{(2)}_v(\hat{\omega})  \right)_{v \geq 0} \]
in the Skorokhod metric on the space of all càdlàg real functions. With an abuse of notation, we will still denote by $\omega$ this coupling, while keeping in mind that the field now depends on $n$.

Our first result improves estimates on the asymptotic behavior of $Z_{n,h}^{\omega,\beta}$ and $(M_n^-,M_n^+)$.

\begin{theorem}\label{th-1/6}
	For any $h,\beta > 0$, we have the following $\PP$-a.s.\ convergence
	\begin{equation}\label{eq:th-1/6-part}
		\lim_{n\to \infty} \frac{1}{\beta n^{1/6}} \left( \log Z_{n,h}^{\omega,\beta} + \frac{3}{2}h c_h n^{1/3} \right) = \sup_{0 \leq u \leq c_h} \left\{ X^{(1)}_u + X^{(2)}_{c_h-u} \right\} \,,
	\end{equation} 
	where $X^{(1)}$ and $X^{(2)}$ are the two independent standard Brownian motions defined above.\\
	Furthermore, $u_* \defeq \argmax_{u \in [0,c_h]} \left\{ X^{(1)}_u + X^{(2)}_{c_h-u} \right\}$ is $\PP$-a.s.\ unique and
	\begin{equation}\label{eq:th-1/6-proba}
		\frac{1}{n^{1/3}} (M_n^-,M_n^+) \xrightarrow[n \to \infty]{\mathbf{P}_{n,h}^{\omega, \beta}} (-u_*,c_h-u_*) \qquad \PP\text{-a.s.}
	\end{equation}
\end{theorem}

\begin{figure}[h]
	\centering
	\includegraphics[width=\textwidth,height=\textheight,keepaspectratio]{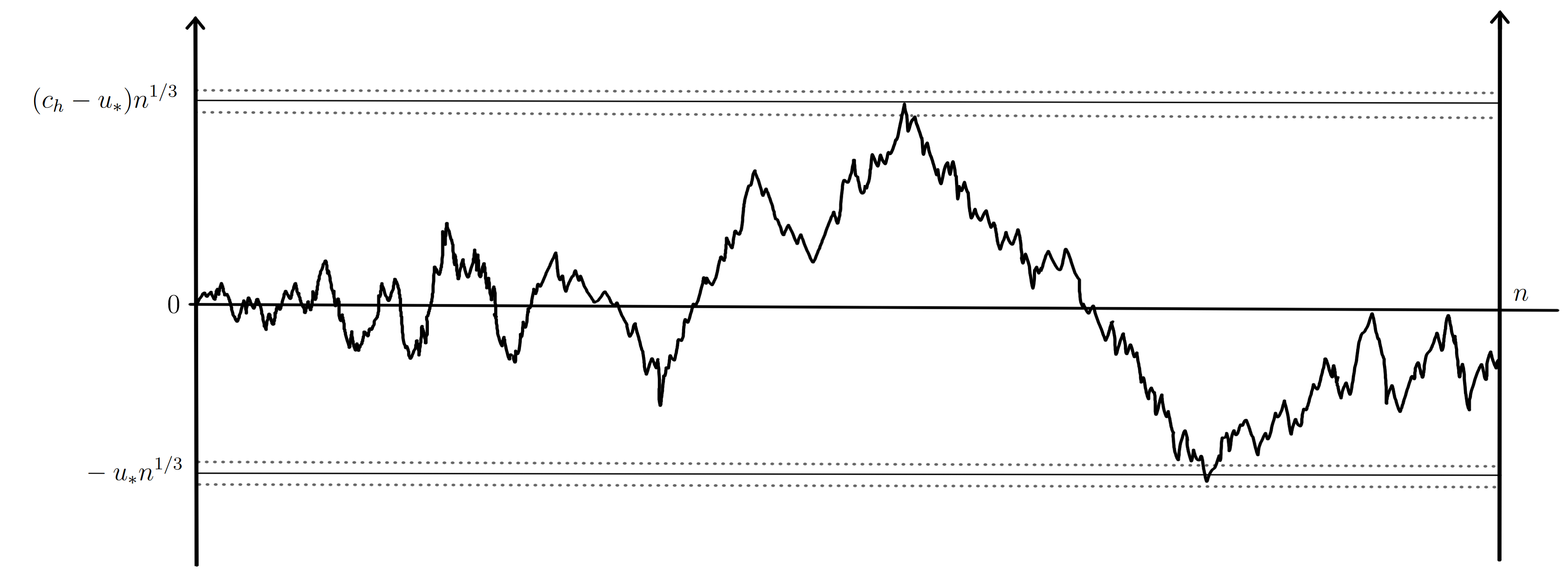}
	\caption{A typical trajectory under the polymer measure for a given $u_*$ and large $n$}
	\label{fig:polymer}
\end{figure}

\begin{remark}
	Theorem \ref{th-1/6} still holds if $(\omega_z)$ are i.i.d and in the domain of attraction of an $\alpha$-stable law with $\alpha \in (1,2)$, only replacing the Brownian motions $X^{(i)}$ by Lévy processes as in \cite{berger2020one} and $n^{1/6}$ by $n^{1/3\alpha}$: we refer to Theorem~\ref{th-1/6-Levy} and its proof in Appendix \ref{appendix-Levy}. As most of the work in this paper requires stronger assumptions on the field $\omega$ we will not dwell further on this possibility and focus on the case where $\esp{\omega_0^2} = 1$.
\end{remark}

\begin{remark}
	We could interpret Theorems \ref{th-1/6} as an almost sure convergence in $\mathbf{P}_{n,h}^{\omega, \beta}$-probability: on the space of closed sets endowed with the Hausdorff distance,
	\[ \mathcal{R}_n n^{-1/3} \xrightarrow[n \to \infty]{\mathbf{P}_{n,h}^{\omega, \beta}} [-u_*, c_h-u_*] \quad \PP\text{-a.s.} \, . \]
\end{remark}

\begin{heuristic}
	Intuitively, the result of Theorem \ref{th-1/6} is a consequence of the following reasoning: if we assume that the optimal size is $T_n^*$ (at a first approximation), the location of the polymer should be around the points $(x_n,y_n) \in \NN^2$ such that $x_n + y_n \approx T_n^*$ and $\Sigma^-_{x_n} + \Sigma^+_{y_n}$ is maximized. Translating in terms of the processes $X^{(1)}, X^{(2)}$, we want to maximize $n^{-1/6} (\Sigma^-_{x_n} + \Sigma^+_{y_n})$, which is ``close'' to $X^{(1)}_{x_n n^{-1/3}} + X^{(2)}_{y_n n^{-1/3}}$. Since $x_n + y_n \sim T_n^*$ we have $y_n n^{-1/3} \sim c_h - x_n n^{-1/3}$ and we want to pick $x_n n^{-1/3}$ to maximize $u \mapsto X^{(1)}_u + X^{(2)}_{c_h - u}$.
\end{heuristic}

\subsection{Second order convergence result}

\par To ease the notation we will denote $X_u \defeq X^{(1)}_{u} + X^{(2)}_{c_h-u}$. Note that $X$ has the distribution of $\sqrt{2} W + X^{(2)}_{c_h}$, where $W$ is a standard Brownian motion.
Hence, the supremum on $[0,c_h]$ of $X_u$ is almost surely finite, attained at a unique~$u_*$ which follows the arcsine law on $[0,c_h]$.
\par In order to extract more information on the typical behavior of the polymer, we need to go deeper into the expansion of $\log Z_{n,h}^{\omega,\beta}$. To do so, we factorize $Z_{n,h}^{\omega,\beta}$ by $e^{\beta n^{1/6} X_{u_*}}$ and we study the behavior of $\log Z_{n,h}^{\omega,\beta} + \frac32 h c_h n^{1/3} - \beta n^{1/6} X_{u_*}$, which is related to the behavior of $X$ near $u_*$.
Studying Wiener processes near their maximum leads to study both the three-dimen\-sio\-nal Bessel process and the Brownian meander (see Appendix \ref{appendix-meandre}).

\begin{proposal}\label{prop:ecriture-X-meandres}
	Conditional on $u_*$, there exist two independent Brownian meanders $(\mathcal{M}^\sigma)_{|\sigma| = 1}$ such that for any $u \in [0,c_h]$,
	\begin{equation}
		X_{u_*} - X_{u} = \sqrt{2 u_*} \mathcal{M}^-_{\frac{u^*-u}{u_*}} \indic{u_* \geq u} + \sqrt{2(c_h - u_*)} \mathcal{M}^+_{\frac{u - u_*}{c_h - u_*}} \indic{u_* < u} \, .
	\end{equation}
\end{proposal}

\begin{proof}
	Recall the fact that $X$ has the law of $\sqrt{2} W + X^{(2)}_{c_h}$ and is maximal when $W$ is maximal. In particular, $X_{u_*} - X_{u}$ has the law of $\sqrt{2}(W_{u_*} - W_{u})$ and $W_{u_*} = \sup_{u \in [0,c_h]} W_u$. Write $\tfrac{1}{\sqrt{2}} (X_{u_*} - X_{u_* + t})$ as $M_X^<(t)$ if $t < 0$, and as $M_X^>(t)$ if $t \geq 0$. By Proposition \ref{prop:max-meandre}, conditional on the value of $u_*$, the processes $M_X^<$ and $M_X^>$ are two independent Brownian meanders, with respective duration $u_*$ and $c_h-u_*$. By the scaling property of the Brownian meander, both $M_X^<$ and $M_X^>$ can be obtained from two independent standard Brownian meanders $\mathcal{M}^\sigma, \sigma \in \mathset{-1,+1}$.
\end{proof}

Some other technical results about the meander are presented in Appendix \ref{appendix-meandre}. We also define the following process, which we call \emph{two-sided three-dimensional Bessel} (BES$_3$) \emph{process}.

\begin{definition}
	We call \textit{two-sided} three-dimensional Bessel process $\mathbf{B}$ the concatenate of two independent three-dimensional Bessel processes $B^-$ and $B^+$. Namely, for all $s\in \mathbb R$, $\mathbf{B}_s = B^-_{-s}\mathbbm{1}_{\RR^-}(s) + B^+_s\mathbbm{1}_{\RR^+}(s)$.
\end{definition}

Additionally, we will use the following coupling between $(X^{(1)}_u+X^{(2)}_{c_h-u}, X_u^{(1)}-X_{c_h-u}^{(2)})$ seen from $u_*$ and a two-sided BES$_3$ process and a Brownian motion. This will allow us to obtain $\PP$-almost sure results instead of convergences in distribution; in particular we obtain trajectorial results that depend on the realization of the environment. The proof is postponed to Appendix \ref{appendix-Construction} and relies on the path decomposition of usual Brownian-related processes.

\begin{proposal}\label{prop-couplage}
	Let 
	\[ X_u = X^{(1)}_u + X^{(2)}_{c_h-u} \,, \qquad Y_u \defeq X^{(1)}_u - X^{(2)}_{c_h-u} \,. \]
	Then, conditionally on $u_*$, one can construct a coupling of $(X^{(1)},X^{(2)})$ and $\mathbf{B}$ a two-sided BES$_3$, $\mathbf{Y}$ a two-sided standard Brownian motion such that: almost surely, there is a $\delta_0 = \delta_0(\omega) >0$ for which on a $\delta_0$-neighborhood of $0$,
	\[ \frac{1}{\sqrt{2}} \big( X_{u_*} - X_{u_*+u} \big) = \chi\mathbf{B}_u \,, \qquad  \frac{1}{\sqrt{2}} \big( Y_{u_*+u} - Y_{u_*} \big) = \mathbf{Y}_u \,,\]
	where we have set $\chi = \chi (u,\omega) \defeq \left(\sqrt{c_h - u_*} \indic{u \geq 0} + \sqrt{u_*} \indic{u < 0} \right)^{-1}$. 
\end{proposal}

\begin{remark}
	It should be noted that $\chi$ actually only depends on the sign of $u$, which means that the process $\chi \mathbf{B}$ has the Brownian scaling invariance property. This will be used in Section \ref{sec:construction-couplage} to get a suitable coupling.
\end{remark}

\begin{theorem}\label{th-1/9}
	Suppose $\esp{|\omega_0|^{3 + \eta}} < \infty$ for some $\eta>0$. 
	With the coupling of Proposition~\ref{prop-couplage}, we have the $\PP$-a.s. convergence
	\begin{equation}\label{eq:th-1/9-part}
		\lim_{n\to \infty} \frac{\sqrt{2}}{\beta n^{1/9}} \left( \log Z_{n,h}^{\omega,\beta} + \frac{3}{2} h c_h n^{1/3} - \beta n^{1/6} X_{u_*} \right) = \sup_{u,v} \left\{ \mathcal{Y}_{u,v} - \frac{3\pi^2}{\beta c_h^4 \sqrt{2}} \big( u + v \big)^2 \right\} \,,
	\end{equation}
	where $\mathcal{Y}_{u,v} \defeq \mathbf{Y}_u - \mathbf{Y}_{-v} - \chi \big[ \mathbf{B}_u +\mathbf{B}_v \big] $.
	
	\noindent
	Moreover, $(\mathcal{U},\mathcal{V}) \defeq \argmax_{u,v} \{ \mathcal{Y}_{u,v} - \frac{3\pi^2}{\beta c_h^4 \sqrt{2}} (u+v)^2\}$ is $\PP$-a.s.\ unique and we have
	\begin{equation}\label{eq:th-1/9-proba}
		\left( \frac{M_n^- + u_*n^{1/3}}{n^{2/9}},\frac{M_n^+ - (c_h-u_*)n^{1/3}}{n^{2/9}} \right) \xrightarrow[n \to \infty]{\mathbf{P}_{n,h}^{\omega, \beta}} (\mathcal{U},\mathcal{V}) \quad \PP\text{-a.s.}
	\end{equation}
	In particular, we have $\displaystyle \frac{T_n - c_h n^{1/3}}{n^{2/9}} \xrightarrow[n \to \infty]{\mathbf{P}_{n,h}^{\omega, \beta}} \mathcal{U}+\mathcal{V}$ $\PP$-a.s.
\end{theorem}

\begin{remark}
	We should be able to obtain a statement assuming only that $\esp{|\omega_0|^{2 + \eta}} < \infty$ for some positive $\eta$. 
	The statement is a bit more involved, as we need to use a different coupling between $\omega$ and $X^{(1)},X^{(2)}$.
	For any $K > 1$, we write
	\[ \bar{\mathcal{Z}}_{n,\omega}^{\leq K} \defeq \bar{Z}_{n,h}^{\omega,\beta} \big( | M_n^-+u_*n^{1/3}| \leq K n^{2/9}, |M_n^+ - (c_h - u_*)n^{1/3}| \leq K n^{2/9} \big) \, . \]
	In Section~\ref{sec:th-1/9-2+}, we are able to the following convergence: write $\mathcal{W}_2$ for the right-hand side of \eqref{eq:th-1/9-part}, then $\PP$-almost surely
	\begin{equation}
		\label{eq:th-1/9-part-weak}
		\lim_{K \to +\infty} \lim_{n\to \infty} \frac{\sqrt{2}}{\beta n^{1/9}} \Bigg( \log \bar{\mathcal{Z}}_{n,\omega}^{\leq K} + \frac{3}{2} h c_h n^{1/3} - \beta \sum_{z = -u_* n^{1/3}}^{(c_h-u_*)n^{1/3}} \omega_z \Bigg) = \mathcal{W}_2 \, .
	\end{equation}
	However, we are not able to get the proof that $\lim_{K \to +\infty} \lim_{n\to \infty} \bar{\mathcal{Z}}_{n,\omega}^{\leq K}/Z_{n,h}^{\omega, \beta} = 1$. We give in Section~\ref{sec:th-1/9-2+} some heuristics for why this second convergence should be true, and why our method fails to prove it.
\end{remark}

\subsection{Comments on the results, outline of the paper}

\paragraph{Expansion of the log-partition function.}

One may think about our results as an expansion of $\log Z_{n,h}^{\beta,\omega}$ up to several orders, gaining each time some information on the location of the endpoints of the range.
A way to formulate such result is, for some real numbers $\alpha_1 > \dots > \alpha_p \geq 0$, to define the following sequence of free energies which we may call \emph{$k$-th order free energy, at scale $\alpha_k$}:
\begin{equation}\label{def:energie}
	\begin{split}
		f_\omega^{(1)}(h,\beta) &= \lim_{n \to \infty} n^{-\alpha_1} \log Z_{n,h}^{\beta,\omega} \\
		f_\omega^{(k+1)}(h,\beta) &= \lim_{n \to \infty} n^{-\alpha_{k+1}} \bigg( \log Z_{n,h}^{\beta,\omega} - \sum_{i=1}^k n^{\alpha_i} f_\omega^{(i)}(h,\beta) \bigg)\,, 
	\end{split}
\end{equation}
when these quantities exist and are in $\RR \setminus \mathset{0}$.

Theorems~\ref{thm0}, \ref{th-1/6} and \ref{th-1/9} can be summarized in the following statement: assuming that $\esp{\omega_0^{3 + \eta}} < \infty$ for some positive $\eta$, then letting $\alpha_k = \frac{1}{3k}$ for $k\in \{1,2,3\}$, we have $\PP$-a.s.
\begin{align*}
	f_\omega^{(1)}(h,\beta) & = \lim_{n\to \infty} \frac{1}{n^{1/3}} \log Z_{n,h}^{\omega,\beta} = -\frac{3}{2} (\pi h)^{2/3} \, , \\
	f_\omega^{(2)}(h,\beta) & = \lim_{n\to \infty} \frac{1}{n^{1/6}} \left( \log Z_{n,h}^{\omega,\beta} + \frac{3}{2}h c_h n^{1/3} \right) = \beta\sup_{0 \leq u \leq c_h} \left\{ X^{(1)}_u + X^{(2)}_{c_h-u} \right\} = \beta X_{u_*} \, , \\
	f_\omega^{(3)}(h,\beta) & = \frac{\beta}{\sqrt{2}} \,\sup_{u,v} \left\{ \mathcal{Y}_{u,v} - \frac{3\pi^2}{\beta c_h^4 \sqrt{2}} \big( u + v \big)^2 \right\} \,.
\end{align*}
Note that the first two orders of $\log Z_{n,h}^{\omega,\beta}$, meaning $f_\omega^{(1)}$ and $f_\omega^{(2)}$, are respectively called the \emph{free energy} and the \emph{surface energy}.

\paragraph{Coupling and almost sure results}

Observe that we combine two different couplings that have different uses to prove our results:
\begin{itemize}
	\item A coupling for a given size $n$ between the environment $\omega$ and two Brownian motions $X^{(1)}$ and $X^{(2)}$. This coupling allows for the almost sure convergence in Theorem \ref{th-1/6}, and the assumption $\esp{\omega_0^{3+\eta}} < +\infty$ is used to have a good enough control on the coupling. This will be detailled in Section \ref{sec:general-1/9}.
	\item A coupling between $(X^{(1)},X^{(2)},u_*)$ and $(\mathbf{B},\mathbf{Y})$ to study the behavior of the Brownian motions $X^{(1)}$ and $X^{(2)}$ near $u_*$. This allows us to get the almost sure convergence of Theorem \ref{th-1/9}. This is the object of Proposition \ref{prop-couplage} and Appendix \ref{appendix-Construction}.
\end{itemize}
We explain how these two couplings combine to get the $\PP$-almost sure results of Theorem \ref{th-1/6},\ref{th-1/9} in Section \ref{sec:construction-couplage}. The main idea is to make the environment $\omega$ and the processes $X^{(1)},X^{(2)}$ depend on $n$ in order to "fix" what we see uniformly in $n$ large enough.

\paragraph{Local limit conjecture}

In Section \ref{sec:simplified} we study a simplified model where the random walk is constrained to be non-negative. By restricting it so, the processes involved are less complex as they depend on only one variable (which represents the higher point of the polymer), which simplifies the calculations.
The idea is to give some insights on what happens when studying $\log Z_{n,h}^{\omega,\beta} - \sum_{i = 1}^3 n^{\alpha_i} f^{(i)}(h,\beta)$, especially on the scale of the $4$-th order free energy.
The environment is taken to be Gaussian in order to get the coupling of $n^{-1/6} \Sigma^\pm_{z n^{1/3}}$ with no coupling error (otherwise the result may not be the same).
We give in Section~\ref{sec:simplified} a detailed justification for the following conjecture, which is a form of local limit theorem.

\begin{conjecture}\label{conj}
	If $\omega_0$ is Gaussian, there is a positive process $\mathcal{W} = \{\mathcal{W}_{a,b}, (a,b) \in \RR^2\}$ and a sequence $(u_n,v_n)_n$ with values in $\left[ 0,1 \right]^2$ such that
	\begin{equation}\label{part-conjecture} \lim_{n \to \infty} \left| Z_{n,h}^{\omega,\beta} \exp \Big( - \sum_{i = 1}^3 n^{\alpha_i} f^{(i)}(h,\beta) \Big) - \sum_{i,j \in \ZZ} e^{- \mathcal{W}_{i+u_n, j+v_n}} \right| = 0 \qquad \PP\text{-a.s.}
	\end{equation}
	In particular, for any couple of integers $(i,j)$, we have $\PP$-a.s.
	\begin{equation}
		\label{proba-conjecture}
		\mathbf{P}_{n,h}^{\omega, \beta} \left( M_n^- - \lfloor - u_*n^{1/3} + \mathcal{U} n^{2/9} \rfloor = i, M_n^+ - \lfloor (c_h - u_*) n^{1/3} + \mathcal{V} n^{2/9} \rfloor = j \right) \sim \frac{e^{- \mathcal{W}_{i+u_n, j+v_n}}}{\theta_\omega (n)}\,,
	\end{equation}
	with $\theta_\omega (n) \defeq \sum_{i,j \in \ZZ} e^{- \mathcal{W}_{i+u_n, j+v_n}}$ a normalizing constant.
\end{conjecture}

The main obstacle to prove this conjecture is that $\mathcal{W}$ is given by zooming in the process $\tilde{\mathcal{Y}}_{u,v} \defeq \mathcal{Y}_{\mathcal{U},\mathcal{V}} - \mathcal{Y}_{u,v}$ as $(u,v)$ gets close to $(\mathcal{U},\mathcal{V})$. With our methods here, we need to know some properties of $\tilde{\mathcal{Y}}$ which is a seemingly complex process due to the already nontrivial nature of $\mathcal{Y}$.

In Section~\ref{sec:simplified}, we study the model with a fixed minimum, that is by replacing the random walk with the random walk conditioned to stay positive. In this case, the process $\mathcal{Y}$ becomes a Brownian motion with parabolic drift, which allows us to conjecture the law of the corresponding $\tilde{\mathcal{Y}}$ as well as a local limit theorem.

\subsection{Related works}

\paragraph{The case of varying parameters $\beta,h$.}

As mentioned above, the present model has previously been studied in \cite{berger2020one}, with the difference that the parameters $\beta,h$ were allowed to depend on $n$, the size of the polymer.
More precisely, the polymer measure considered was given for arbitrary $\hat{h},\hat{\beta} \in \RR$ by
\[ \dd \mathbf{P}_{n,h}^{\omega,\beta}(S) = \frac{1}{Z_{n,h}^{\omega,\beta}} \exp \Big( \sum_{z \in \mathcal{R}_n(S)} \big(\beta_n \omega_z - h_n\big) \Big) \dd \mathbf{P}(S), \quad \text{with } h_n = \hat h n^{-\zeta}, \, \beta_n = \hat \beta n^{-\gamma} \, . \]
The authors in~\cite{berger2020one} obtained $\PP$-almost sure convergences of $n^{-\lambda} \log Z_{n,h_n}^{\omega,\beta_n}$ for some suitable $\lambda \in \RR$, which corresponds to a first order expansion of the log-partition function. Afterwards, asymptotics for $\mathbb E \mathbf{E}_{n,h_n}^{\omega,\beta_n} [|\mathcal{R}_n|]$ as well as scaling limits for $(M_n^-,M_n^+)$ were established and displayed a wide variety of phases. In addition, the authors also investigated the case where $(\omega_z)$ are i.i.d and in the domain of attraction of an $\alpha$-stable law with $\alpha \in (0,1) \cup (1,2)$ to unveil an even richer phase diagram.

Theorems \ref{th-1/6} and \ref{th-1/9} confirm the conjecture of Comment 4 of \cite{berger2020one} that for a typical configuration $\omega$, the fluctuations of the log-partition function and $n^{-1/3} (M_n^-, M_n^+)$ are not $\mathbf{P}$-random for fixed $h,\beta > 0$. With our methods, it should be possible to extend our results to account for size-dependent $h = h_n, \beta = \beta_n$, with similar results for ``reasonable'' $h_n, \beta_n$ (meaning with sufficiently slow growth/decay).

\paragraph{Link with the random walk among Bernoulli obstacles.}

Take a Bernoulli site percolation with parameter $p$, meaning a collection $\mathcal{O} = \mathset{z \in \ZZ^d, \eta_z = 1}$ where $\eta_z$ are i.i.d. Bernoulli variables with parameter $p$, and write $\mathcal{P} = \mathcal{B}(p)^{\otimes \ZZ}$ its law on $\ZZ$. Consider the random walk starting at $0$ and let $\tau$ denote the time it first encounters $\mathcal{O}$ (called the set of obstacles): one is interested in the asymptotic behavior of the survival probability $\mathbf P(\tau>n)$ as $n\to\infty$ and of the behavior of the random walk conditionally on having $\tau>n$, see for example \cite{ding2019poly} and references therein.
The \textit{annealed} survival probability $\mathbb{E}^{\mathcal{P}} \mathbf P(\tau>n)$ is given by
\[
\mathbb{E}^{\mathcal{P}} \otimes \mathbf{E} \left[ \indic{\mathcal{R}_n \cap \mathcal{O} = \varnothing} \right] = \mathbf{E} \left[ \mathcal{P} \left[ \forall z \in \mathcal{R}_n, \eta_z = 0 \right] \right] = \mathbf{E} \left[ (1-p)^{|\mathcal{R}_n|} \right] = \mathbf{E} \left[ e^{|\mathcal{R}_n| \log (1-p)} \right] ,
\]
and we observe that this is exactly $Z_{n,h_p}^{\omega,0}$ with $h_p = - \log (1-p)$. 
Thus, for $\beta = 0$, our model can be seen as an annealed version of the random walk among Bernoulli obstacles with common parameter $p=1-e^{-h}$.
\par If we push the analogy a bit further and assume $\beta \omega_z - h \leq 0$ for all $z \in \ZZ$, we can see $Z_{n,h}^{\omega, \beta}$ as the annealed survival probability of the random walk among obstacles $\mathcal{O}^\omega = \mathset{z \in \ZZ^d, \eta^\omega_z = 1}$ where $\eta^\omega_z$ are i.i.d.\ Bernoulli variables with random parameter $p^\omega_z = 1-e^{\beta \omega_z-h}$.
The averaging is done on the random walk (with law $\mathbf{P}$) and the Bernoulli variables (with law $\mathcal{P}^\omega = \bigotimes_{z \in \ZZ} \mathcal{B}(p^\omega_z)$), while the parameters $p^\omega_z = 1-e^{\beta \omega_z-h}$ (with law $\PP$) are quenched.

\paragraph{Link with the directed polymer model.}
Another famous model  is given by considering a doubly indexed field $(\omega_{i,z})_{(i,z) \in \NN \times \ZZ}$ and the polymer measure
\[ \dd \mathbf{P}_{n,h}^{\omega,\beta}(S) = \frac{1}{Z_{n,h}^{\omega,\beta}} \exp \Big( \sum_{i = 0}^n \big(\beta \omega_{i, S_i} - \lambda(\beta) \big) \Big) \dd \mathbf{P}(S) \, , \quad \lambda(\beta) = \log \esp{e^{\beta \omega}} \,.\]
This is known as the \textit{directed polymer model} (in contrast with our non-directed model) and has been the object of an intense activity over the past decades, see~\cite{Comets} for an overview.
Let us simply mention that the partition function solves (in a weak sense) a discretized version of Stochastic Heat Equation (SHE) with multiplicative space-time noise $\partial_t u = \Delta u + \beta \xi \cdot u$.
Hence, the convergence of the partition function under a proper scaling $\beta = \beta(n)$, dubbed \textit{intermediate disorder} scaling, has raised particular interest in recent years: see \cite{AKQ14,CSZ16} for the case of dimension $1$ and \cite{caravenna-SHF} for the case of dimension $2$, where this approach enabled the authors to give a notion of solution to the SHE; see also \cite{SHE-levy-noise} for the case of a heavy-tailed noise.

The main difference with our model is how the disorder $\omega$ plays into the polymer measure. The directed polymer gets a new reward/penalty $\omega_{i,z}$ at each step it takes, whereas in our model such event only happens when reaching a new site of $\ZZ$, in some sense ``consuming'' $\omega_z$ when landing on $z$ for the first time.

\subsection{Outline of the paper.}
This paper can be split into three parts.
The first part in Section \ref{sec:ordre1} consists in the proof of Theorem \ref{th-1/6}. The second and main part focuses on the proof of Theorem \ref{th-1/9}. This proof is split into Section \ref{sec:gaussian-1/9} where $\omega$ is assumed to be Gaussian and Section \ref{sec:general-1/9} where we explain how to get the general statement thanks to a coupling.
A third part, in Section \ref{sec:simplified}, studies the simplified model where the random walk is constrained to be non-negative. Precise results under some technical assumption help us formulate the conjectures in \eqref{part-conjecture} and \eqref{proba-conjecture}.

Finally, we prove in Appendix \ref{appendix-Levy} the generalization of Theorem \ref{th-1/6} to the case when~$\omega$ does not have a finite second moment, as announced. We also state some useful properties of the Brownian meander that we use in our proofs in Appendix~\ref{appendix-meandre}.
In Appendix \ref{appendix-Construction} we detail a way to couple Brownian meanders with a two-sided three-dimensional Bessel process so that they are equal near $0$ (\textit{i.e.}\ we prove Proposition \ref{prop-couplage}).

\section{Second order expansion and optimal position}\label{sec:ordre1}

We extensively use the following notation: For a given event $\mathcal A$ (which may depend on~$\omega$), we write the partition function restricted to $\mathcal{A}$ as
\[ Z_{n,h}^{\omega,\beta}(\mathcal{A}) \defeq \mathbf{E}\Big[ \exp \Big( \sum_{z \in \mathcal{R}_n(S)} \big(\beta \omega_z - h\big) \Big) \mathbbm{1}_{\mathcal A} \Big], \quad \text{so that} \quad \mathbf{P}_{n,h}^{\omega,\beta}(\mathcal A) = \frac{1}{Z_{n,h}^{\omega,\beta}} Z_{n,h}^{\omega,\beta}(\mathcal{A}). \]

This section consists in the proof of Theorem \ref{th-1/6} and is divided into two steps:
\begin{itemize}
	\item We first make use of a coarse-graining approach with a size $\delta n^{1/3}$ to prove the convergence of the rescaled $\log Z_{n,h}^{\omega,\beta} -\frac32 h T_n^*$. At the same time, we locate the main contribution as coming from trajectories whose left-most point is around $-u_* n^{1/3}$, proving \eqref{eq:th-1/6-part}.
	\item We then prove that $\mathbb P$-a.s., $n^{-1/3} M_n^-$ converges in $\mathbf{P}_{n,h}^{\omega, \beta}$-probability to $-u_*$, using the previously step and the fact that $\mathbf{P}_{n,h}^{\omega,\beta}(\mathcal{A}) = Z_{n,h}^{\omega,\beta} (\mathcal{A}) / Z_{n,h}^{\omega,\beta}$. Since we also have the result of \eqref{result-quentin}, we deduce \eqref{eq:th-1/6-proba} thanks to Slutsky's lemma as $M_n^-, M_n^+$ and $T_n$ are defined on the same probability space.
\end{itemize}

\subsection{A rewriting of the partition function}\label{sec:part-homogene-desordre}

Theorem~\ref{thm0} implies that there is a vanishing sequence $(\eps_n)_{n \geq 0}$ such that $\PP$-almost surely,
\begin{equation}
	\frac{Z_{n,h}^{\omega,\beta}(|\Delta_n| \leq \eps_n T_n^*)}{Z_{n,h}^{\omega,\beta}} \xrightarrow[n \to +\infty]{} 1 \, .
\end{equation}
In particular, we may restrict our study to $Z_{n,h}^{\omega,\beta}(|\Delta_n| \leq \eps_n T_n^*)$. Writing $\Delta_n^{x,y} \defeq x + y - c_h n^{1/3}$, we have
\begin{equation}\label{eq:partition-proba}
	Z_{n,h}^{\omega,\beta}(|\Delta_n| \leq \eps_n T_n^*) = \sum_{\substack{ x,y \geq 0 \\ |\Delta_n^{x,y}| \leq \eps_n T_n^*}} \exp \Big( -h(x+y+1) + \beta \sum_{z = -x}^y \omega_z \Big) \mathbf{P} \big( \mathcal R_n = \llbracket -x,y \rrbracket \big) \,.
\end{equation}
Gambler's ruin formulae derived from \cite[Chap.~XIV]{Feller} can be used to compute sharp asymptotics for $\mathbf{P} \big( \mathcal R_n = \llbracket -x,y \rrbracket \big)$, see \cite[Theorem~1.4]{bouchot1}.
In particular, we get
\begin{equation}
	\lim_{n \to \infty} \sup_{\substack{x,y \in \NN \\ |\Delta_n^{x,y}| \leq \eps_n T_n^*}} \left| \frac{\mathbf{P} \big( \mathcal R_n = \llbracket -x,y \rrbracket \big)}{\Theta_n(x,y)} - 1 \right| = 0 \, ,
\end{equation}
where we defined the function $\Theta_n(x,y)$ for $x+y = T$ as
\[ \Theta_n(x,y) \defeq \frac{4}{\pi}  (e^h-1)\left[ e^h \sin \left( \frac{\pi(x+1)}{T} \right)  -  \sin \left( \frac{\pi x}{T} \right) \right] e^{-g(T)n} \, , \]
with $g(T) = -\log \cos \tfrac{\pi}{T} \sim \tfrac{\pi^2}{2T^2}$. Since for fixed $T = x + y$, we are interested in $|T - T_n^*| = |\Delta_n^{x,y}| \leq \eps_n T_n^*$, we can rewrite with a Taylor expansion
\[ \Theta_n(x,y) = \frac{4}{\pi} 
(e^h-1)^2 \sin \left( \frac{\pi x}{T} \right) \left[ 1 + e^h \frac{\pi}{T_n^*} \cos \left( \frac{\pi x}{T} \right)  \Big( 1 + \grdO((T_n^*)^{-1}) + \grdO(\eps_n) \Big) \right] e^{-g(T)n} \, . \]
Here $\grdO((T_n^*)^{-1})$ is deterministic and uniform in $x,y$ such that $|\Delta_n^{x,y}| \leq \eps_n T_n^*$. Therefore, writing $\psi_h = e^{-h} \frac{4}{\pi} (e^h-1)^2$, we have as $n \to +\infty$
\begin{equation}\label{eq:preuve-homogene-restreint}
	Z_{n,h}^{\omega,\beta}(|\Delta_n| \leq \eps_n T_n^*) = (1 + \bar{o}(1)) \psi_h \sum_{\substack{ x,y \geq 0 \\ |\Delta_n^{x,y}| \leq \eps_n T_n^*}} \exp \Big( -h(x+y) - g(x+y) n + \beta \sum_{z = -x}^y \omega_z \Big) \, ,
\end{equation}
with a deterministic $\bar{o}(1)$. Then, write $\varphi_n(T) := hT + \frac{n\pi^2}{2T^2}$, we have
\begin{equation}\label{eq:preuve-homogene-energie}
	h T + g(T)n = \varphi_n(T) + n \big(g(T) - \frac{\pi^2}{2T^2} \big) = \varphi_n(T) + \frac{n}{12(T_n^*)^4} (1 + \grdO(\eps_n)) \, .
\end{equation}
We also easily check that $T_n^*$ is the minimizer of $\varphi_n$ and that: $\varphi_n(T_n^*) = \tfrac32 h T_n^*$, $\varphi_n''(T) = \frac{3n \pi^2}{T^4}$ and $\varphi_n'''(T) = - \frac{12 n \pi^2}{T^5}$. Thus, with a Taylor expansion, we have for $x,y \geq 0$ that satisfies $|\Delta_n^{x,y}| \leq \eps_n T_n^*$:
\begin{equation}\label{eq:preuve-homogene-DL-phi}
	\bigg|\varphi_n(T) - \varphi_n(T_n^*) - \frac{3n \pi^2}{(T_n^{*})^4} (\Delta_n^{x,y})^2 \bigg| \leq C \frac{|\Delta_n^{x,y}|^3 n}{(T_n^*)^5} \leq \eps_n \frac{n}{(T_n^{*})^4} (\Delta_n^{x,y})^2 \,.
\end{equation}
Assembling \eqref{eq:preuve-homogene-energie} and \eqref{eq:preuve-homogene-DL-phi} with \eqref{eq:preuve-homogene-restreint},
\begin{equation}\label{partition-restrcit-epsnTn}
	\begin{split}
		&Z_{n,h}^{\omega,\beta}(|\Delta_n| \leq \eps_n T_n^*) =\\ &(1+\bar{o}(1)) \psi_h 			e^{-\tfrac32 h T_n^*} \sum_{x,y=0}^{+\infty} \sin \left( \frac{x\pi}{x+y} \right) 			\exp \left( \beta \sum_{z = -x}^y \omega_z - \frac{3 \pi^2 (\Delta_n^{x,y})^2}				{c_h^4 n^{1/3}} (1 + \bar{o}(1)) \right)
	\end{split}
\end{equation}
for deterministic $\bar{o}(1)$ that are uniform in $x,y$ satisfying $|\Delta_n^{x,y}| \leq \eps_n T_n^*$. We also get in particular that $\PP$-almost surely,
\begin{equation}\label{partition}
	Z_{n,h}^{\omega,\beta} = (1+\bar{o}(1)) \psi_h e^{-\tfrac32 h T_n^*} \sum_{x,y=0}^{+\infty} \sin \left( \frac{x\pi}{x+y} \right) \exp \left( \beta \sum_{z = -x}^y \omega_z - \frac{3 \pi^2 (\Delta_n^{x,y})^2}{c_h^4 n^{1/3}} (1 + \bar{o}(1)) \right) \, .
\end{equation}
More generally, for any event $\mathcal{A} \subseteq \mathset{|\Delta_n| \leq \eps_n T_n}$, with the same considerations, we can write $Z_{n,h}^{\omega,\beta}(\mathcal{A})$ as the sum in \eqref{partition-restrcit-epsnTn} restricted to trajectories satisfying $\mathcal{A}$.

\subsection{Convergence of the log partition function}

In order to lighten notation, we always omit integer parts in the following.

\begin{proof}[Proof of Theorem \ref{th-1/6}-\eqref{eq:th-1/6-part}]
	Recall (\ref{partition}), choose some $\delta > 0$ and split the sum over $x,y$ depending on $k_1 \delta n^{1/3} \leq x < (k_1+1) \delta n^{1/3}$ and $k_2 \delta n^{1/3} \leq y < (k_2+1) \delta n^{1/3}$. By \eqref{partition}, we may only consider the pairs $(x,y)$ that satisfy $|\Delta_n^{x,y}| =|x+y- c_h n^{1/3}|\leq \eps_n n^{1/3} < \delta n^{1/3}$ for $n$ sufficiently large; note that this implies that $(k_1+k_2)\delta \in \{ c_h - \delta, c_h\}$. \newline
	We can now rewrite~\eqref{partition} as 
	\begin{equation}
		\label{eq:decoupZnh}
		\log Z_{n,h}^{\omega,\beta} + \frac{3}{2} h c_h n^{1/3} =\bar{o}(1)+ \log \psi_h + \log  \Lambda_{n,h}^{\omega,\beta}(\delta) \,, 
	\end{equation}
	in which we defined
	\begin{equation}\label{def:Lambdanh}
		\Lambda_{n,h}^{\omega,\beta}(\delta)  \defeq \sum_{k_1 = 0}^{c_h/\delta} \sum_{k_2 = \frac{c_h}{\delta} - k_1 - 1}^{\frac{c_h}{\delta} - k_1} Z_{n,h}^{\omega,\beta} (k_1,k_2,\delta)  \,,
	\end{equation}
	with
	\begin{equation}\label{Znw-retreint-1/6}
		Z_{n,h}^{\omega,\beta} (k_1,k_2,\delta) \defeq \!\!\!\sum_{\substack{k_1 \delta n^{1/3} \leq x < (k_1+1) \delta n^{1/3}\\ k_2 \delta n^{1/3} \leq y < (k_2+1) \delta n^{1/3}}} \!\!\! \sin \left( \frac{x\pi}{x+y} \right) \exp \left(\beta \sum_{z = -x}^y \omega_z - \frac{3 \pi^2 (\Delta_n^{x,y})^2}{2 c_h^4 n^{1/3}}(1+\bar{o}(1)) \right) \,.
	\end{equation}

	Then let us define
	\begin{equation}\label{def:W(uvd)}
		\mathcal{W}^{\pm}(u,v,\delta) \defeq  X^{(1)}_u + X^{(2)}_v \pm \sup_{u \leq u' \leq u + \delta} \big| X^{(1)}_{u'}-X^{(1)}_u \big| \pm \sup_{v \leq v' \leq v + \delta} \big| X^{(2)}_{v'}-X^{(2)}_v \big|.
	\end{equation} 
	Theorem \ref{th-1/6}-\eqref{eq:th-1/6-part} essentially derives from the following lemma.
	
	\begin{lemma}\label{lem:cv-1/6}
		For any integers $k_1,k_2$ and any $\delta > 0$, we have $\PP$-almost surely
		\begin{equation*}
			\mathcal{W}^-(k_1\delta,k_2\delta,\delta) \leq \varliminf_{n\to \infty} \frac{\log Z_{n,h}^{\omega,\beta} (k_1,k_2,\delta)}{\beta n^{1/6}} \leq \varlimsup_{n\to \infty} \frac{\log Z_{n,h}^{\omega,\beta} (k_1,k_2,\delta)}{\beta n^{1/6}} \leq \mathcal{W}^+(k_1\delta,k_2\delta,\delta) \, .
		\end{equation*}
	\end{lemma}
	
	Let us use this lemma to conclude the proof of the convergence~\eqref{eq:th-1/6-part}. 
	Since the sum in~\eqref{def:Lambdanh} has $\frac{2c_h}{\delta}$ terms, we easily get that
	\[ 0 \leq \log \Lambda_{n,h}^{\omega,\beta}(\delta) - \max_{\substack{0 \leq k_1,k_2 \leq c_h/\delta\\(k_1+k_2)\delta \in \{ c_h - \delta, c_h\}}} \!\!\! \log Z_{n,h}^{\omega,\beta} (k_1,k_2,\delta) \leq \log \frac{2c_h}{\delta} \, . \]
	Dividing by $\beta n^{1/6}$ and taking the limit $n\to\infty$, Lemma~\ref{lem:cv-1/6} yields
	\[ \varlimsup_{{n \to \infty}} \frac{1}{\beta n^{1/6}} \log \Lambda_{n,h}^{\omega,\beta}(\delta) \leq \max_{\substack{0 \leq k_1,k_2 \leq c_h/\delta\\(k_1+k_2)\delta \in \{ c_h - \delta, c_h\}}} \mathcal{W}^+(k_1\delta,k_2\delta,\delta). \]
	We write $u = k_1 \delta$ and $v = k_2 \delta$, belonging to the finite set $U_\delta$ defined as
	\begin{equation}\label{eq:def-U-delta}
		U_\delta \defeq \left\{ (u,v) \in (\RR_+)^2 \, : \, u \in \big\{ \delta, 2\delta, \dots , \lfloor \frac{c_h}{\delta} \rfloor \delta \big\}, u+v \in \{c_h,c_h-\delta\} \right\} \,,
	\end{equation}
	so 
	\[ \varlimsup_{\delta \to 0} \varlimsup_{n \to \infty} \frac{\log \Lambda_{n,h}^{\omega,\beta}(\delta)}{\beta n^{1/6}} \leq \varlimsup_{\delta \to 0} \max_{\substack{0 \leq u,v \leq c_h\\ u + v \in \{ c_h - \delta, c_h\}}} \mathcal{W}^{+}(u,v,\delta) = \sup_{\substack{0 \leq u,v \leq c_h\\ u + v = c_h}} \left\{ X^{(1)}_u + X^{(2)}_v \right\} \quad \text{$\PP$-a.s.} \,, \]
	where for the last identity, we have used the continuity of $X^{(1)}$ and $X^{(2)}$.
	
	The same goes for $\liminf\limits_{n \to \infty} n^{-1/6} \Lambda_{n,h}^{\omega,\beta}(\delta)$, with the lower bound $\mathcal{W}^{-}(u,v,\delta)$, which concludes the proof.
\end{proof}

\begin{proof}[Proof of Lemma \ref{lem:cv-1/6}]
	The proof is inspired by the proof of Lemma 5.1 in \cite{berger2020one}. Recall the definition \eqref{Znw-retreint-1/6} of $Z_{n,h}^{\omega,\beta} (k_1,k_2,\delta)$ and note that for $k_1 \delta n^{1/3} \leq x < (k_1+1) \delta n^{1/3}$ and $k_2 \delta n^{1/3} \leq y < (k_2+1) \delta n^{1/3}$:
	\begin{equation}
		\left( \Sigma^+_{k_2\delta n^{1/3}} + \Sigma^-_{k_1\delta n^{1/3}} \right) - R_n^\delta(k_1\delta, k_2\delta) \leq \sum_{z = -x}^y \omega_z \leq \left( \Sigma^+_{k_2\delta n^{1/3}} + \Sigma^-_{k_1\delta n^{1/3}} \right) + R_n^\delta(k_1\delta, k_2\delta)
	\end{equation}
	where the error term $R_n^\delta$ is defined for $u,v \geq 0$ by
	\[ R_n^\delta(u,v) \defeq \max_{un^{1/3} + 1 \leq j \leq (u+\delta)n^{1/3}-1} \big| \Sigma^-_j - \Sigma^-_{un^{1/3}} \big| + \max_{vn^{1/3} + 1 \leq j \leq (v+\delta)n^{1/3}-1} \big| \Sigma^+_j - \Sigma^+_{vn^{1/3}} \big| \, .
	\]
	Using the coupling $\hat{\omega}$ and Lemma A.5 of \cite{berger2020one} (for Lévy processes), $\PP$-a.s., for all $\eps> 0$, 
	for all $n$ large enough (how large depends on~$\eps,\delta,\omega$),
	\[
	\frac{1}{n^{1/6}} R_n^\delta(u,v) \leq \eps + \sup_{u \leq u' \leq u + \eps + \delta} \big| X^{(1)}_{u'}-X^{(1)}_u \big| + \sup_{v \leq v' \leq v + \eps + \delta} \big| X^{(2)}_{v'}-X^{(2)}_v \big|
	\]
	and
	\[
	\Big( \Big| \frac{1}{n^{1/6}}\Sigma^+_{vn^{1/3}} - X^{(2)}_v \Big| \vee \Big| \frac{1}{n^{1/6}}\Sigma^-_{un^{1/3}} - X^{(1)}_u \Big| \Big) \leq \eps  \,,
	\]
	uniformly in $u$ and $v$, since $U_\delta$ is a finite set. Thus, letting $n \to \infty$ then $\eps \to 0$ we obtain that $\PP$-almost surely,
	\[ \varlimsup_{n\to \infty} \frac{1}{\beta n^{1/6}} \log Z_{n,h}^{\omega,\beta}(k_1,k_2,\delta) \leq  \sup_{\substack{k_1 \delta n^{1/3} \leq x < (k_1+1) \delta n^{1/3}\\ k_2 \delta n^{1/3} \leq y < (k_2+1) \delta n^{1/3}}} \varlimsup_{n\to \infty} \frac{1}{n^{1/6}}  \sum_{z = -x}^y \omega_z \leq \mathcal{W}^+(u,v,\delta) \, . \]
	in which we recall the definition \eqref{def:W(uvd)} of $\mathcal{W}^{\pm}(u,v,\delta)$.
	
	\medskip
	On the other hand, since $Z_{n,h}^{\omega,\beta} (k_1,k_2,\delta)$ is a sum of non-negative terms, we get a simple lower bound by restricting to configurations with almost no fluctuation around $T_n^*$:
	\[ 
	\begin{split}
		\frac{\log Z_{n,h}^{\omega,\beta}(k_1,k_2,\delta)}{n^{1/6}} & \geq \sup_{|\Delta_n^{x,y}| \leq 1} \left\{ \frac{\beta}{n^{1/6}} \sum_{z = -x}^y \omega_z - \frac{3\pi^2}{2c_h^4} \frac{(\Delta_n^{x,y})^2}{n^{1/2}} \right\} -\bar{o}(1) \\
		& = \frac{\beta}{n^{1/6}} \sup_{|\Delta_n^{x,y}| \leq 1}  \sum_{z = -x}^y \omega_z - \frac{3\pi^2}{2c_h^4\sqrt{n}} -  \bar{o}(1), 
	\end{split}
	\]
	in which the supremum is taken on the $(x,y)$ that satisfy the criteria of $Z_{n,h}^{\omega,\beta} (k_1,k_2,\delta)$, see~\eqref{Znw-retreint-1/6}. 
	In the above, the $\bar{o}(1)$ is deterministic and comes from the contribution of $n^{-1/6}\log \sin(\frac{x\pi}{x+y})$; in the case where $k_1=0$, we restrict the supremum to additionally having $x\neq 0$, so that we always have $\sin(\frac{x\pi}{x+y}) \geq \frac{c}{x+y} \sim \frac{c}{n^{1/3}}$. 
	
	After the exact same calculations as above, we get the lower bound
	\[ \liminf_{n\to \infty} \frac{1}{\beta n^{1/6}} \log Z_{n,h}^{\omega,\beta}(k_1,k_2,\delta) \geq \mathcal{W}^-(u,v,\delta). \qedhere \]
\end{proof}

\begin{lemma}\label{pb-var-1/6}
	The quantity $\sup_{u + v = c_h} \big\{ X^{(1)}_v + X^{(2)}_u \big\} = \sup_{u \in [0,c_h]} X_u$ is almost surely positive and finite, and attained at a unique point $u_*$ of $[0,c_h]$.
\end{lemma}

\begin{proof}
	Recall that $X$ has the same law as $\sqrt{2} W + X^{(2)}_{c_h}$ where $W_u = \frac{1}{\sqrt{2}} (X^{(1)}_u + X^{(2)}_u)$ is a standard Brownian motion independent from $X^{(2)}_{c_h}$. Thus it is a classical result, see for example \cite[Lemma~2.6]{kim1990cube}.
\end{proof}

\subsection{Path properties under the polymer measure}

\begin{proof}[Proof of Theorem \ref{th-1/6}-\eqref{eq:th-1/6-proba}]
	The proof essentially reduces to the following lemma.
	
	\begin{lemma}\label{lem:position}
		For any $h,\beta > 0$, recall
		$u_* \defeq \argmax_{u \in [0,c_h]} \big\{ X^{(1)}_u + X^{(2)}_{c_h-u} \big\}$. Then, $\mathbb P$-a.s.
		\[ \frac{1}{n^{1/3}} M_n^- \xrightarrow[n\to\infty]{ \mathbf{P}_{n,h}^{\omega, \beta} } - u_* \,. \]
	\end{lemma}
	
	By Slutsky's Lemma (for a fixed $\omega$ in the set of $\omega$'s for which both convergences are true), Lemma \ref{lem:position} combined with \eqref{result-quentin} readily implies that $\mathbb P$-a.s. $n^{-1/3} M_n^+$ converges to $c_h - u_*$ in $\mathbf{P}_{n,h}^{\omega, \beta}$-probability. Note that Slutsky's lemma can be used on $M_n^+, M_n^-, T_n$ since they are all defined on the same probability space.
\end{proof}

\begin{proof}[Proof of Lemma \ref{lem:position}]
	The proof is analogous to what is done in \cite{berger2020one}.
	Define the following set
	\[ \mathcal{U}^{\eps,\eps'} \defeq \bigg\{ u \in [0,c_h] \, : \, \sup_{s, |s-u|< \eps} \big\{ X^{(1)}_s + X^{(2)}_{c_h-s} \big\} \geq X_{u_*} - \eps' > 0 \bigg\} \,. \]
	We shall prove that for almost all $\omega$, we have $\mathbf{P}_{n,h}^{\omega, \beta} \left( \frac{1}{n^{1/3}} |M_n^-| \not\in \mathcal{U}^{\eps,\eps'} \right) \to 0$. For this, we denote by $\mathcal{A}_n^{\eps,\eps'}$ the event $\big\{ \frac{1}{n^{1/3}} |M_n^-| \not\in \mathcal{U}^{\eps,\eps'} \big\}$. As 
	\[ \begin{split}
		\log \mathbf{P}_{n,h}^{\omega, \beta} \left( \mathcal{A}_n^{\eps,\eps'} \right) &= \log Z_{n,h}^{\omega,\beta}(\mathcal{A}_n^{\eps,\eps'}) - \log Z_{n,h}^{\omega,\beta}\\
		&= \left(\log Z_{n,h}^{\omega,\beta}(\mathcal{A}_n^{\eps,\eps'}) + \frac{3}{2}hc_hn^{1/3}\right) - \left( \log Z_{n,h}^{\omega,\beta}  + \frac{3}{2}hc_hn^{1/3}\right),
	\end{split} \]
	we only need to prove that $\varlimsup\limits_{n\to\infty} \frac{1}{\beta n^{1/6}} \big[ \log Z_{n,h}^{\omega,\beta}(\mathcal{A}_n^{\eps,\eps'}) + \frac{3}{2}h c_h n^{1/3} \big] < X_{u_*}$. Indeed, using the convergence \eqref{eq:th-1/6-part} in Theorem \ref{th-1/6}, we then get that $\varlimsup\limits_{n\to\infty} \frac{1}{\beta n^{1/6}} \log \mathbf{P}_{n,h}^{\omega, \beta} \left( \mathcal{A}_n^{\eps,\eps'} \right) < 0$.
	\newline We apply the same decomposition we used in the proof of Theorem \ref{th-1/6}-\eqref{eq:th-1/6-part} over indices $k_1$ such that $-[k_1,k_1+1)\delta n^{1/3} \not\subset \mathcal{U}^{\eps,\eps'}$. Thus,
	\[ \varlimsup_{\delta \downarrow 0} \varlimsup_{n \to \infty} \frac{1}{\beta n^{1/6}} \left[ \log Z_{n,h}^{\omega,\beta}(\mathcal{A}_n^{\eps,\eps'}) + \frac{3}{2}hT_n^* \right] \leq \sup_{u \not\in \mathcal{U}^{\eps,\eps'}} \big\{ X^{(1)}_u + X^{(2)}_{c_h-u} \big\}\leq X_{u_*} - \eps' \, , \]
	so we indeed have $\mathbf{P}_{n,h}^{\omega, \beta} (\mathcal{A}_n^{\eps,\eps'}) \to 0$.
	Using that $\bigcap_{\eps'>0} \mathcal{U}^{\eps,\eps'} \subset B_{2\eps}(u_*)$ by unicity of the supremum,
	we have thus proved that, $\mathbb P$-a.s., $n^{-1/3} M_n^- \to -u_*$ in $\mathbf{P}_{n,h}^{\omega, \beta}$-probability.
\end{proof}

\section{Proof of Theorem \ref{th-1/9} for a Gaussian environment}\label{sec:gaussian-1/9}

In this section we prove Theorem \ref{th-1/9} under the  assumption that $\omega_0$ has a Gaussian distribution. We take full advantage of the fact that in this case, the coupling with the Brownian motions $X^{(1)},X^{(2)}$, is just an identity: it will thus not create any coupling error and allows us to work directly on these processes. The proof still requires some heavy calculations as we must first find what are the relevant trajectories in the factorized log-partition function.

Going forward, we take the following setting: random variables $\omega_z$ are i.i.d. with normal distribution $\mathcal{N}(0,1)$ and $X^{(1)},X^{(2)}$ are standard Brownian motions such that 
\begin{equation}\label{MB-omega}
	\frac{1}{n^{1/6}} \sum_{z=1}^x \omega_{-z} = X^{(1)}_{xn^{-1/3}} \; , \hspace{0.8cm} \frac{1}{n^{1/6}} \sum_{z=0}^y \omega_z = X^{(2)}_{yn^{-1/3}} \,.
\end{equation}
We will adapt the following proof to a general environment in Section~\ref{sec:general-1/9} by controlling the error term due to the coupling.

\medskip
We define
\[ \bar{Z}_{n,h}^{\omega, \beta} \defeq Z_{n,h}^{\omega, \beta} \, e^{\frac32 h T_n^* - \beta n^{1/6} X_{u_*}} \; , \quad \bar{Z}_{n,h}^{\omega, \beta}(\mathcal{A}) \defeq Z_{n,h}^{\omega, \beta}(\mathcal{A}) \, e^{\frac32 h T_n^* - \beta n^{1/6} X_{u_*}} \,, \]
so that \eqref{eq:th-1/9-part} can be rewritten as a statement regarding the convergence of $n^{-1/9} \log \bar{Z}_{n,h}^{\omega, \beta}$. \newline
Here are the four steps of the proof:
\begin{itemize}
	\item We first rewrite $\bar{Z}_{n,h}^{\omega, \beta}$ to make $X_{xn^{-1/3}} - X_{u_*}$ appear. Having this negative quantity makes it easier to find the relevant trajectories. Indeed, when $|X_{|M_n^-|n^{-1/3}} - X_{u_*}|$ is too large for a given trajectory, the relative contribution of this trajectory to the partition function goes exponentially to $0$. This means that this configuration has a low $\mathbf{P}_{n,h}^{\omega,\beta}$-probability.
	\item We prove the $\PP$-almost sure convergence of $n^{-1/9} \log \bar{Z}_{n,h}^{\omega,\beta}$ restricted to the event $\mathcal{A}_{n,\omega}^{K,L} = \mathset{|\Delta_n| \leq K n^{2/9},| M_n^-+u_*n^{1/3}| \leq Ln^{2/9}}$ towards a positive value. It consists again of a coarse-graining approach where each component $\bar{Z}_{n,h}^{\omega, \beta}(u,v)$ converges to $\mathcal{Y}_{u,v} - c_{h,\beta} (u+v)^2$. This leads to defining $(\mathcal{U},\mathcal{V})$ via a variational problem.
	\item We prove that $n^{-1/9} \log \bar{Z}_{n,h}^{\omega,\beta}$ restricted to $(\mathcal{A}_{n,\omega}^{K,L})^c$ is almost surely negative as $n \to \infty$ as soon as $K$ or $L$ is sufficiently large. Coupled with the previous convergence towards a positive limit, we prove that all of these trajectories have a negligible contribution.
	\item Afterwards, the convergences in $\mathbf{P}_{n,h}^{\omega,\beta}$-probability are derived in the same way as for Theorem \ref{th-1/6}.
\end{itemize}

\begin{corollary}[of Lemma \ref{lem:position}]\label{cor:u*}
	For any $\eps > 0$, consider the event
	\begin{equation}
		\mathcal{A}_{n,\eps}^{\omega} \defeq \mathset{|\Delta_n| \leq \eps n^{1/3}  , |M_n^- + u_*  n^{1/3}| \leq \eps n^{1/3}} \, .
	\end{equation}
	There exists a vanishing sequence $(\eps_n)_{n \geq 1}$ such that
	\[ \lim_{n \to +\infty} \mathbf{P}_{n,h}^{\omega, \beta}\big(\mathcal{A}_{n,\eps_n}^{\omega} \big) = \lim_{n \to +\infty} \frac{1}{Z_{n,h}^{\omega, \beta}} Z_{n,h}^{\omega, \beta}\big(\mathcal{A}_{n,\eps_n}^{\omega} \big) = 1 \qquad \text{$\PP$-a.s.} \,.\]
\end{corollary}

Going forward, we will work conditionally on $u_*$. Recall that $\frac{1}{\sqrt{2}} (X - X^{(2)}_{c_h})$ has the law of a standard Brownian motion, thus according to Proposition \ref{prop:max-meandre}, the processes $(X_{u_*} - X_{u_*-t}, t \geq 0)$ and $(X_{u_*} - X_{u_*+t}, t \geq 0)$ are two Brownian meanders, respectively on $[0,u_*]$ and $[0, c_h-u_*]$. Recall that since $u_*$ follows the arcsine law on $[0,c_h]$, these intervals are $\PP$-almost surely nonempty.

\subsection{Rewriting the partition function}\label{sec:rewrite}

We define
\begin{equation}\label{fact-1/9}
	\Omega_n^{x,y} \defeq \frac{1}{n^{1/6}}\sum_{z = 1}^{x} \omega_{-z} + \frac{1}{n^{1/6}}\sum_{z = 0}^{y} \omega_{z} -  X_{u_*} =  X^{(1)}_{xn^{-1/3}} + X^{(2)}_{y n^{-1/3}} - X_{u_*} ,
\end{equation} 
where for the last identity we have used the relation~\eqref{MB-omega} between $X$ and $\omega$.
Then, we can rewrite



\begin{equation}
	\bar{Z}_{n,h}^{\omega, \beta}\big( \mathcal{A}_{n,\eps_n}^{\omega} \big) = \psi_h \sum_{x,y \geq 0} \Psi_n^\omega(x,y) \exp\left( \beta n^{1/6} \Omega_{n}^{x,y} - \frac{3 \pi^2 (\Delta_n^{x,y})^2}{2 c_h^4 n^{1/3}} (1+\bar{o}(1)) \right) \, .
\end{equation}
with
\[ \Psi_n^\omega(x,y) = \sin \left( \frac{x \pi}{x+y} \right) \indic{|x-u^*n^{1/3}|\leq \eps_n n^{1/3} , |y-(c_h -u^*)n^{1/3}|\leq \eps_n n^{1/3}} \, . \]
Note that if $x,y$ satisfy $|x-u^*n^{1/3}|\leq \eps_n n^{1/3} , |y-(c_h -u^*)n^{1/3}|\leq \eps_n n^{1/3}$, and if $n$ is large enough to have $\eps_n < \tfrac12$, we have
\[ \left| \sin \left( \frac{x \pi}{x+y} \right) - \sin \left( \frac{u_* \pi}{c_h} \right) \right| \leq 2  \left|\cos \left( \frac{u_* \pi}{c_h} \right) \left( \frac{x \pi}{x+y} - \frac{u_* \pi}{c_h} \right) \right| \leq \frac{2\eps_n(1 + u_*)}{c_h(1-\eps_n)} \leq 4 \eps_n \frac{1+c_h}{c_h} \, . \]
Since $u_* \not\in \mathset{0,c_h}$ $\PP$-a.s., using~\eqref{partition} and Theorem~\ref{th-1/6} we can write:
\begin{equation}\label{partition2}
	\begin{split}
		&\bar{Z}_{n,h}^{\omega, \beta}\big( \mathcal{A}_{n,\eps_n}^{\omega} \big) =\\ &(1 + \bar{o}(1)) \psi_h \sin\left(\frac{u_*\pi}{c_h}\right) \!\!\!\!\!\! \sum_{\substack{|x-u^*n^{1/3}|\leq \eps_n n^{1/3} \\ |y-(c_h -u^*)n^{1/3}|\leq \eps_n n^{1/3}}} \!\!\!\!\!\!
		\exp\left( \beta n^{1/6} \Omega_{n}^{x,y} - \frac{3 \pi^2 (\Delta_n^{x,y})^2}{2 c_h^4 n^{1/3}}(1 + \bar{o}(1)) \right) \, ,
	\end{split}
\end{equation}
where both $\bar{o}(1)$ are deterministic and are a $\grdO(\eps_n)$.

Note that $X_{u_*} - (X^{(1)}_{x n^{-1/3}} + X^{(2)}_{yn^{-1/3}} \big)$ is not necessarily positive since the supremum in \eqref{eq:th-1/6-part} is taken over non negative $u$ and $v$ such that $u+v = c_h$, whereas $x+y \neq c_h n^{1/3}$ in the general case. 
However we can write
\[ \big(X^{(1)}_{xn^{-1/3}} + X^{(2)}_{y n^{-1/3}} \big) = \big(X^{(1)}_{xn^{-1/3}} + X^{(2)}_{c_h- x n^{-1/3}} \big) + \big( X^{(2)}_{yn^{-1/3}} - X^{(2)}_{c_h-x n^{-1/3}} \big) \,,\]
so that $\Omega_n^{x,y}$ can be rewritten as 
\begin{equation}\label{eq:part-rewrit}
	\Omega_{n}^{x,y} \defeq -  \left( X_{u_*} - X_{xn^{-1/3}} \right) +  X^{(2)}_{yn^{-1/3}} - X^{(2)}_{c_h-x n^{-1/3}}\, .
\end{equation}
Note that it is not problematic that $c_h-x n^{-1/3}$ can be negative if $y$ is small enough, since $X^{(2)}$ can be defined on the real line.
Although \eqref{eq:part-rewrit} may seem more complex to study than \eqref{fact-1/9}, having a term that is always non-positive is useful to isolate the main contributions to the partition function.

Recall that $X_{u_*} - X_{xn^{-1/3}}$ can be expressed in terms of Brownian meanders depending on the sign of $u_* - xn^{-1/3}$, see Proposition~\ref{prop:ecriture-X-meandres}. More precisely, there exist $\mathcal{M}^+,\mathcal{M}^-$ two independent standard Brownian meanders such that
\begin{equation}\label{eq:X-meandre}
	X_{u_*} - X_{xn^{-1/3}} = \sqrt{2 u_*} \mathcal{M}^-_{\frac{u^*-xn^{-1/3}}{u_*}} \indic{u_* \geq xn^{-1/3}} + \sqrt{2(c_h - u_*)} \mathcal{M}^-_{\frac{xn^{-1/3} - u_*}{c_h - u_*}} \indic{u_* < xn^{-1/3}} \, .
\end{equation}

\begin{heuristic}
	In~\eqref{partition2} and in view of~\eqref{eq:part-rewrit}, the term inside the exponential can be split into three parts.
	The first part is $- \beta n^{1/6} \left( X_{u_*} - X_{x n^{-1/3}} \right)$, which is negative and of order $n^{1/6} |u_* - x n^{-1/3}|^{1/2}$.
	The second term is $\beta n^{1/6} ( X^{(2)}_{yn^{-1/3}} - X^{(2)}_{c_h-xn^{-1/3}})$, which is of order at most $(\Delta_n)^{1/2}$.
	The last term is $- \tilde c_h (\Delta_n)^2 n^{-1/3}$.
	We thus can easily compare the second term to the last one: dominant terms in \eqref{partition} are all negative when $(\Delta_n)^2 n^{-1/3} \gg (\Delta_n)^{1/2}$ or in other words if $\Delta_n \gg n^{2/9}$.
	Thus we will show that the corresponding trajectories have a negligible contribution to $Z_{n,h}^{\omega, \beta}$, and that we can restrict the partition function to trajectories such that $\Delta_n = \grdO(n^{2/9})$. We can apply the same reasoning to the first term, which must verify $n^{1/6} ( X^{(2)}_{yn^{-1/3}} - X^{(2)}_{c_h-xn^{-1/3}})  = \grdO(n^{1/9})$, from which we will deduce $|M_n^- n^{-1/3}+ u^* |= \grdO(n^{-1/9})$.
\end{heuristic}

\subsection{Coupling and construction}\label{sec:construction-couplage}

Here we explain how these two coupling combine to yield all the desired results.
We start by picking $u_*$ according to the arcsine law on $[0,c_h]$ and by considering a three-dimensional two-sided Bessel process $\mathbf{B}$ as well as an independent two-sided Brownian motion $\mathbf{Y}$, both defined on $\RR$.
Since the process $(X_{u_*}-X_{u_* + u})/\sqrt{2}$ is a two-sided Brownian meander (with left interval $[0,u_*]$ and right interval $[0,c_h-u_*]$), using Proposition \ref{prop-couplage} we can find a $\delta_0(\omega)$ such that if $|u| \leq \delta_0$, we have $X_{u_*}-X_{u_* + u} = \mathbf{B}_u \chi \sqrt{2}$.
\par
We are interested in a coupling that will be such that $n^{1/18} (X_{u_*}-X_{u_* + \frac{u}{n^{1/9}}}) = \mathbf{B}_u \chi \sqrt{2}$ for all $n$ large enough and for any $u n^{-1/9}$ sufficiently close to $0$. To do so, for each $n$ we construct from $\mathbf{B}$ a suitable $X^n$, with the same law as $X$, that satisfies the desired equality.

\par Consider only the pair $(\delta_0,\mathbf{B})$ that was previously defined, and let $n_0$ be such that $\eps_{n_0} \leq \delta_0$. Then, for any $n \geq n_0$, we paste the trajectory of $n^{-1/18} \sqrt{2} \chi \mathbf{B}_{u n^{1/9}}$, which is still a two-sided three-dimensional Bessel process multiplied by $\sqrt{2}$ (note that $\chi$ is scale-invariant), until $|u| = \delta_0$. By construction, we have $X^n_{u_*} - X^n_{u_* + u} = n^{-1/18} \sqrt{2} \chi \mathbf{B}_{u n^{1/9}}$ for $|u| \leq \delta_0$. Next, we consider two independent Brownian meanders $M^{L,n,\delta_0}, M^{R,n,\delta_0}$ of duration $u_*$ (resp. $c_h-u_*$) conditioned on $M^{L,n,\delta_0}_{\delta_0} = n^{-1/18} \sqrt{2} \chi \mathbf{B}_{-\delta_0 n^{1/9}}$ (resp. $M^{R,n,\delta_0}_{\delta_0} = n^{-1/18} \sqrt{2} \chi \mathbf{B}_{\delta_0 n^{1/9}}$), and we plug their trajectory to complete the process $X^n$. The full definition of $X^n$ is thus given by
\[ \frac{1}{\sqrt{2}} (X^n_{u_*} - X^n_u) = n^{-1/18} \chi \mathbf{B}_{u n^{1/9}} \indic{|u-u_*| < \delta_0} + M^{L,n,\delta_0}_{u_* - u} \indic{u \in [0, u_* - \delta_0]} + M^{R,n,\delta_0}_{u - u_*} \indic{u \in [u_* + \delta_0, c_h]} \]
We can similarly define $Y^n_{u_* + u} - Y^n_{u_*} \defeq n^{-1/18} \sqrt{2} \mathbf{Y}_{un^{1/9}}$ where no particular coupling is needed. From $X^n$ and $Y^n$, we can recover our new Brownian motions $X^{(1),n}, X^{(2),n}$.

\par
In the case of a Gaussian environment, we can define the random variables $(\omega_z)_{z \in \ZZ^d}$ using \eqref{MB-omega}.
For the other cases, we construct the environment $\omega = \omega^n$ from the processes $(X^{(1),n},X^{(2),n})$ using Skorokhod's embedding theorem (see Theorems \ref{th-skorokhod},\ref{th-strassen}).

In both cases, all of our processes are defined to have almost sure convergences, and when $n$ is greater than some $n_0(\omega)$ (which only only depends on $\delta_0$) and $|un^{-1/9}| \leq \eps_n \leq \delta_0$, we have
\begin{equation}\label{eq:egalites-couplage}
	n^{1/18} (X^n_{u_*} - X^n_{u_* + \frac{u}{n^{-1/9}}}) = \sqrt{2} \chi \mathbf{B}_u \, , \quad n^{1/18} (Y^n_{u_*} - Y^n_{u_* + \frac{u}{n^{-1/9}}}) = \sqrt{2} \mathbf{Y}_u \, \, .
\end{equation}
Our construction will be used to prove Proposition \ref{th:restrict-1/9} and Proposition \ref{prop-cv-1/9} in order to get proper limits when zooming around $u_*$. This is done by making it so that if $n$ is large enough, the process we study do not depend on $n$, as illustrated by Proposition \ref{prop-couplage}.

\subsection{Restricting the trajectories}

We recall that in what follows, we work conditionally on the value of $u_*$, yet we will still write $\PP$ for the law of $\omega$ under conditional to this value.

Our goal is now to characterize the main contribution to the partition function directly in terms of $M_n^-$ and $M_n^+$ or equivalent quantities, and not in terms of the processes. With this goal in mind, we define, for $K,L\geq 0$:
\[
\bar{Z}_{n,\omega}^{>}(K,L) \defeq \bar{Z}_{n,h}^{\omega,\beta} \big( |\Delta_n| \geq L n^{2/9}, K n^{2/9} \leq |M_n^- +u_*n^{1/3}| \leq \eps_n n^{1/3} \big)
\]
and 
\[
\bar{Z}_{n,\omega}^{\leq}(K,L) \defeq \bar{Z}_{n,h}^{\omega,\beta} \big( |\Delta_n| \leq L n^{2/9}, |M_n^- +u_*n^{1/3}| \leq K n^{2/9} \leq \eps_n n^{1/3} \big) .
\]

In this section, 
In the next section, we will prove, in Proposition \ref{prop-cv-1/9} and Lemma \ref{lem-var2}, that $\PP$-a.s.,
$\varliminf_{n \to \infty} n^{-1/9} \log \bar{Z}_{n,\omega}^{\leq}(K,L) > 0$.

The following proposition and its Corollary \ref{cor:lem:exclusion-1/9} shows that  $\PP$-a.s. for $K$ \textit{or} $L$ large enough, $\varlimsup_{n \to \infty} n^{-1/9} \log \bar{Z}_{n,\omega}^{>}(K,L) < 0$, meaning that trajectories in $\bar{Z}_{n,\omega}^{>}(K,L)$ have a negligible contribution.

\begin{proposal}\label{th:restrict-1/9}
	Uniformly in $n \geq 1$ such that $\eps_n < \frac12$, we have
	\[ \varlimsup_{K \to \infty} \proba{\varlimsup_{n \to \infty} \frac{1}{n^{1/9}} \log \bar{Z}_{n,\omega}^{>}(K,0) \geq -1} = \varlimsup_{L \to \infty} \proba{\varlimsup_{n \to \infty} \frac{1}{n^{1/9}} \log \bar{Z}_{n,\omega}^{>}(0,L) \geq -1} = 0 \, . \]
\end{proposal}

Let us use Proposition \ref{th:restrict-1/9} and conclude on the main result of this section. For any $K > 1$ we define
\begin{equation}\label{eq:def-Z>K} \begin{split}
		\bar{\mathcal{Z}}_{n,\omega}^{> K} \defeq \bar{Z}_{n_0,\omega}^{>}(K,0) + \bar{Z}_{n_0,\omega}^{>}(0,K) = & \bar{Z}_{n,h}^{\omega,\beta} \big(K n^{2/9} < | M_n^-+u_*n^{1/3}| \leq \eps_n n^{1/3} \big) \\
		&+ \bar{Z}_{n,h}^{\omega,\beta} \big( K n^{2/9} < |M_n^+ - (c_h - u_*)n^{1/3}| \leq \eps_n n^{1/3} \big) \, .
\end{split} \end{equation}

Using Proposition \ref{th:restrict-1/9}, we can prove the main result of this section.

\begin{corollary}\label{cor:lem:exclusion-1/9}
	For $\PP$-almost all $\omega$, there is a $K_0 > 1$ such that for all $K \geq K_0$,
	\begin{equation}
		\limsup_{n \to +\infty} \frac{1}{n^{1/9}} \log \bar{\mathcal{Z}}_{n,\omega}^{> K} \leq -1 \, .
	\end{equation}
\end{corollary}

\begin{proof}
	Applying Proposition \ref{th:restrict-1/9}, we proved			
	\begin{equation}\label{eq:Z>K-proba}
		\varlimsup_{K \to \infty} \proba{\varlimsup_{n \to \infty} \frac{1}{n^{1/9}} \log \bar{\mathcal{Z}}_{n,\omega}^{> K} \geq -1} = 0 \, .
	\end{equation}
	Write $p_K$ for the probability in \eqref{eq:Z>K-proba}, in particular we proved that $p_K \to 0$. Thus, we can extract a subsequence $K_k$ such that
	\[ \sum_{k \geq 1} \proba{\limsup_{n \to +\infty} \frac{1}{n^{1/9}} \log \bar{\mathcal{Z}}_{n,\omega}^{> K_k} \geq -1} < +\infty \, . \]
	Using Borel-Cantelli lemma, $\PP$-almost surely there is a $k_0^\omega \geq 1$ such that for all $k \geq k_0^\omega$, $\limsup_{n \to +\infty} \frac{1}{n^{1/9}} \log \bar{\mathcal{Z}}_{n,\omega}^{> K_k} \leq -1$. Since $\bar{\mathcal{Z}}_{n,\omega}^{> K}$ is non-increasing in $K$, we deduce that for any $j \geq K_{k_0^\omega}$, $\limsup_{n \to +\infty} \frac{1}{n^{1/9}} \log \bar{\mathcal{Z}}_{n,\omega}^{> j} \leq -1$, hence the result.
\end{proof}

In the rest of the section, we will prove Proposition \ref{th:restrict-1/9}. This proof boils down to upper bounds on both probabilities, but with a fixed $n$ instead of the limsup. The fact that the bounds are uniform in $n$, and a use of our coupling, will give us the result.

We first explain a small argument that we will use repetitively throughout the paper in order to transfer estimates from standard Brownian meanders to $X_{u_*} - X_{u_* + u}$. Take an interval $I$ and real numbers $\alpha, \lambda > 0$. In the following Lemmas we will need to compute probabilities such as $\PP \Big(\inf_{|u| \in  \lambda I}  \{X_{u_*} - X_{u_* + u}\} \leq \alpha \Big)$.
First note that
\[
\PP \Big(\inf_{|u| \in  \lambda I}  \{X_{u_*} - X_{u_* + u}\} \leq \alpha \Big)
\leq 2 \max_{\sigma \in \mathset{-1,+1}} \PP \Big(\inf_{\sigma u \in  \lambda I}  \{ X_{u_*} - X_{u_* + u} \} \leq \alpha \Big). \]
To get bounds on those probabilities, we use \eqref{eq:X-meandre}, which gives an expression of $X_{u_*} - X_{u_* + u}$ as a concatenate of rescaled Brownian meanders $\mathcal{M}^\sigma, \sigma \in \mathset{-1,+1}$. Taking for example $\sigma = +1$, that is $u \geq 0$, we have 
\begin{equation}\label{eq:changer-les-constantes}
	\PP \Big(\inf_{u \in  \lambda I}  \{ X_{u_*} - X_{u_* + u}\} \leq \alpha \Big) = \PP \Big(\inf_{u \in  \frac{\lambda}{c_h - u_*} I}  \sqrt{c_h - u_*} \mathcal{M}^+_u \leq \alpha \Big).
\end{equation}
In the following lemmas, we work conditionally on $u_*$ and we have no need for precise values of the constants. Thus, we can consider $u_*$ and $c_h - u_*$ as constants, and remove them from our calculations in order to ease the notation. For example, instead of getting a bound on \eqref{eq:changer-les-constantes}, it is sufficient to get a bound on $\PP(\inf_{u \in I}  \mathcal{M}^+_u \leq \alpha )$.

\begin{lemma}\label{lem:excl-traj-L} There is a positive constant $C = C(h,\beta)$, uniform in $L,n\geq 1$ such that
	\begin{equation}\label{eq:cv-proba-delta}
		\proba{\frac{1}{n^{1/9}} \log \bar{Z}_{n,\omega}^{>}(0,L) \geq -1} \leq e^{- C L^3} \xrightarrow[L \to \infty]{} 0 \, .
	\end{equation}
\end{lemma}

\begin{proof}
	We can write $\bar{Z}_{n,\omega}^{>}(0,L) = \sum_{k,l \geq 0} \bar{Z}_{n,\omega}^{>}(0,L)_{k,l}$ with
	\[ 
	\bar{Z}_{n,\omega}^{>}(0,L)_{k,l} = \bar{Z}_{n,h}^{\omega,\beta} \big( |\Delta_n| \in  2^l [1,2)  L n^{2/9} \, , \, ||M_n^-|n^{-1/3} - u_*|  \in [k,k+1) 2^l L n^{-1/9} \big)\,.
	\]
	
	Using~\eqref{partition2}, we have
	\[ \bar{Z}_{n,\omega}^{>}(0,L)_{k,l} \leq C_h (\eps_n n^{1/3})^2 \exp \left( \beta n^{1/6} (\mathcal{X}^{(2)}_{k,l} - \mathcal{M}^n_{k,l}) - \frac{3\pi^2}{2 c_h^4} (2^lLn^{2/9})^2 n^{-1/3} \right) \,,\]
	where we have set 
	\[ \mathcal{M}^n_{k,l} \defeq \inf_{|u-u_*| \in [k,k+1) 2^l L n^{-1/9}} X_{u_*} - X_u \, , \quad \mathcal{X}^{(2)}_{k,l} \defeq \sup_{\substack{|\Delta_n^{x,y}| \in 2^l [1,2) L n^{2/9}\\ |u-u_*| \in [k,k+1) 2^l L n^{-1/9}}} |X^{(2)}_v - X^{(2)}_{c_h-u}| \]
	with $u=xn^{-1/3}$ and $v=y n^{-1/3}$. Thus, a union bound yields
	\[ \begin{split}
		\proba{\bar{Z}_{n,\omega}^{>}(0,L) \geq e^{-n^{1/9}}} & \leq \sum_{k,l = 0}^{+\infty} \proba{C_h (\eps_n n^{1/3})^2 e^{\beta n^{1/6} (\mathcal{X}^{(2)}_{k,l} - \mathcal{M}^n_{k,l})} e^{- \frac{3\pi^2}{2c_h^4} (2^l L)^2 n^{1/9}} \geq \frac{e^{-n^{1/9}}}{2^{l+1}} }\\
		&\leq \sum_{k,l = 0}^{+\infty} \proba{\beta n^{1/6} (\mathcal{X}^{(2)}_{k,l} - \mathcal{M}^n_{k,l}) \geq  c_h' n^{1/9}  2^{2l} L^2} \,,
	\end{split} \]
	where we have used that for $n$ large enough (how large depends only on $h$) 
	\[ n^{1/9}(\frac{3\pi^2}{2c_h^4} 2^{2l} L^2 - 1) - (l+1)\log 2 - 2 \log (\eps_n n^{1/3})  -\log C_h \geq c'_h n^{1/9} 2^{2l} L^2 \]
	for some constant $c_h'$, uniformly in $L\geq c_h^2/\pi$ and $l\geq 0$.
	We now work out an upper bound on $\proba{\mathcal{X}^{(2)}_{k,l} - \mathcal{M}^n_{k,l} \geq  c_h' n^{-1/18}  2^{2l} L^2/\beta}$, we first observe that writing $u_k = u_* + 2^l k L n^{-1/9}$, we have $|c_h-u_k-v| \leq |c_h - u - v| + |u-u_k| \leq 2^{l+1} L n^{-1/9}$ on the intervals where the supremum are taken in $\mathcal{X}^{(2)}_{k,l}$. Thus, we have the upper bound
	\begin{equation}\label{eq:localisation-u-MB}
		\begin{split}
			\mathcal{X}^{(2)}_{k,l} \leq \sup_{|c_h- u_k - v| \leq 2 \frac{2^l L}{n^{1/9}}} |X^{(2)}_v - X^{(2)}_{c_h-u_k}| + \sup_{|u-u_k| \leq \frac{2^l L}{n^{1/9}}} |X^{(2)}_{c_h-u_k} - X^{(2)}_{c_h-u}| \, .
		\end{split} 
	\end{equation}
	Write $\mathcal{X}^{(2),v}_{k,l}$ and $\mathcal{X}^{(2),u}_{k,l}$ for the first and the second term of the right-hand side of \eqref{eq:localisation-u-MB} respectively, as well as $\alpha_l \defeq c_h' n^{-1/18}  2^{2l} L^2/\beta$.
	We first need to control the term $k=0$, in which we know that $\mathcal{M}^n_{0,l} = 0$, then we are left to bound
	\[ \sum_{l = 0}^{+\infty} \proba{\beta n^{1/6} \mathcal{X}^{(2)}_{0,l} \geq  c_h' n^{1/9}  2^{2l} L^2} \leq \sum_{l = 0}^{+\infty} \left[ \proba{\mathcal{X}^{(2),v}_{0,l} \geq \frac{\alpha_{l}}{2}} 
	+ \proba{\mathcal{X}^{(2),u}_{0,l} \geq \frac{\alpha_{l}}{2}} \right] \, . \]
	By the reflection principle for Brownian motion, both of these variables are the modulus of a Gaussian of variance $2 \frac{2^l L}{n^{1/9}}$ and $\frac{2^l L}{n^{1/9}}$ respectively. Thus, we have the upper bound
	\[ \sum_{l = 0}^{+\infty} \proba{\beta n^{1/6} \mathcal{X}^{(2)}_{0,l} \geq  c_h' n^{1/9}  2^{2l} L^2} \leq \sum_{l = 0}^{+\infty} c_0 \left[ e^{- c_1 2^{3l} L^3} 
	+ e^{- c_2 2^{3l} L^3} \right] \leq (cst.) \sum_{l = 0}^{+\infty} e^{- c 2^{3l} L^3} \, . \]
	for some constants $c_0,c_1,c_2,c > 0$.
	\par We now focus on $k \geq 1$ and decompose on whether $\mathcal{M}^n_{k,l}$ is less or greater than $k^{1/8} (2^l L)^{1/2} n^{-1/18}$. On $\mathset{\mathcal{M}^n_{k,l} \leq k^{1/8} (2^l L)^{1/2} n^{-1/18}}$, we use Hölder inequality for $p > 1$:
	\[ \proba{\mathcal{M}^n_{k,l} \leq k^{1/8} \frac{\sqrt{2^l L}}{n^{1/18}}, \mathcal{X}^{(2)}_{k,l} - \mathcal{M}^n_{k,l} \geq \alpha_l} \leq \proba{\mathcal{M}^n_{k,l} \leq k^{1/8} \frac{\sqrt{2^l L}}{n^{1/18}}}^{1/p} \proba{\mathcal{X}^{(2)}_{k,l} \geq \alpha_l}^{1-\frac{1}{p}} . \]
	
	Since $k$ can we taken arbitrarily, in the definition of $\mathcal{M}^n_{k,l}$ we can replace $X_{u_*} - X_{u_*+u}$ by $\mathcal{M}_u$ for $\mathcal{M}$ a standard Brownian meander on $[0,1]$.
	Thus, with the help of Corollary \ref{cor-meander} with $\lambda = 2$ and the previous argument ($L$ and $k$ can be taken up to a positive multiplicative constant), we compute
	\[ \begin{split}
		\proba{\mathcal{M}^n_{k,l} \leq k^{1/8} \sqrt{2^l L} n^{-1/18}} &\leq \frac{16 k^{1/8} }{\sqrt{\pi k}}\left( 1 \wedge \frac{(2 k^{1/8})^2}{2 k} \right) + (cst.)  k^{5/4} \frac{e^{- 2  k^{1/4}}}{1 - e^{-8 k^{1/4} }}\\
		&\leq (cst.) k^{-9/8} + (cst.) k^{1/4} \frac{e^{- 2  k^{5/4}}}{1 - e^{-8 k^{1/4}}} \leq (cst.) k^{-9/8} \, .
	\end{split} \]
	
	And for the other probability, we use
	\[ \proba{\mathcal{X}^{(2)}_{k,l} \geq \alpha_l} = \proba{\mathcal{X}^{(2)}_{0,l} \geq \alpha_l} \leq (cst.) e^{- c 2^{3l} L^3} \, . \]
	Therefore,
	\[ \proba{\mathcal{M}^n_{k,l} \leq k^{1/8} \sqrt{2^l L} n^{-1/18}, \mathcal{X}^{(2)}_{k,l} - \mathcal{M}^n_{k,l} \geq \alpha_l} \leq (cst.) k^{-9/8p} e^{- c(1-\frac{1}{p}) 2^{3l} L^3} \, , \]
	which has a finite sum in $k \geq 1$ and $l \geq 0$ that goes to $0$ when $L \to +\infty$ when $p$ is suffienciently close to $1$.
	\par On the other hand,
	\[ \proba{\mathcal{M}^n_{k,l} \geq k^{1/8} \sqrt{2^l L} n^{-1/18}, \mathcal{X}^{(2)}_{k,l} - \mathcal{M}^n_{k,l} \geq \alpha_l} \leq \proba{\mathcal{X}^{(2)}_{k,l} \geq \alpha_l + k^{1/8} \sqrt{2^l L} n^{-1/18}} \, , \]
	and again by \eqref{eq:localisation-u-MB} and the Brownian reflection principle,
	\[ \proba{\mathcal{X}^{(2)}_{k,l} \geq \alpha_l + k^{1/8} \sqrt{2^l L} n^{-1/18}} \leq C e^{-c \frac{(\alpha_l + k^{1/8} \sqrt{2^l L} n^{-1/18})^2}{2^l L n^{-1/9}}} \leq (cst.) e^{-c_1 2^{3l}L^3} e^{-c_2 k^{1/4}} \, , \]
	which is again summable in $k,l$, with a sum that goes to $0$ when $L \to +\infty$.
	In conclusion, we proved that $\proba{\bar{Z}_{n,\omega}^{>}(0,L) \geq e^{-n^{1/9}}}$ is bounded by $ce^{-CL^3}$, uniformly in $n$ large enough, thus proving the lemma.
\end{proof}

\begin{lemma}\label{lem:excl-traj-K} There is a positive $C$ such that for any $n \geq 1$ such that $\eps_n < \frac12$, any $K \geq 1$
	\begin{equation}\label{eq:cv-proba-u}
		\proba{\frac{1}{n^{1/9}} \log \bar{Z}_{n,\omega}^{>}(K,0) \geq -1} \leq \frac{C}{K^{1/12}} \xrightarrow[K \to \infty]{} 0 \, . 
	\end{equation}
\end{lemma}

\begin{proof}
	We will use the same strategy as for Lemma \ref{lem:excl-traj-L}. In particular we control both $M_n^- + u_* n^{1/3}$ and $\Delta_n$, instead of only $M_n^- + u_* n^{1/3}$. Thus, we consider
	\[ 
	\bar{Z}_{n,\omega}^{>}(K,0)_{k,l} = \bar{Z}_{n,h}^{\omega,\beta} \big( |M_n^- + u_* n^{1/3}| \in 2^k [1,2) K n^{2/9}, |\Delta_n| \in [l,l+1) n^{2/9} \big) \, ,
	\]
	and when summing on $l \geq 0$ we get $\bar{Z}_{n,\omega}^{>}(K,0)_{k} = \bar{Z}_{n,h}^{\omega,\beta} \big( |M_n^- + u_* n^{1/3}| \in 2^k [1,2) K n^{2/9} \big)$ similar to the notation in Lemma \ref{lem:excl-traj-L}.
	Let us introduce 
	\[
	\xi_{k,l} \defeq -\inf_{|u| \in  2^{k} [1,2) K n^{-1/9}}  \{X_{u_*} - X_{u_* + u}\} + \sup_{|\Delta_n| \in [l,l+1) L n^{2/9}} |X^{(2)}_v - X^{(2)}_{c_h-u}| \, .
	\]
	With the same considerations as before, we have by a union bound
	\begin{equation}\label{eq:prop-restrict-1}
		\proba{\bar{Z}_{n,\omega}^{>}(K,0) \geq e^{-n^{1/9}}} \leq \sumtwo{k,l = 0}{2^k K \leq \eps_n n^{1/9}}^{+\infty} \proba{ C_h (\eps_n n^{1/3})^2 e^{\beta n^{1/6} \xi_{k,l}} e^{- \frac{3\pi^2}{2c_h^4} l^2 n^{1/9}} \geq \frac{e^{-n^{1/9}}}{2^{k+l+1}} } \,.
	\end{equation}
	Observe that each probability in the sum is equal to
	\[  \proba{\beta n^{1/6} \xi_{k,l} \geq \Big(\frac{3\pi^2}{2c_h^4} l^2 - 1\Big) n^{1/9} - (k + l + 1) \log 2 - 2 \log (\eps_n n^{1/3}) -\log C_h}. \]
	Since we have the restriction $2^k K n^{-1/9} \leq \eps_n$, assuming $K \geq 1$ we have $k \log 2 \leq \log (\eps_n n^{1/9})$, thus there is a constant $c>0$ such that for $n$ large enough (how large depends only on $h$),
	\[
	\Big(\frac{3\pi^2}{2c_h^4} l^2 - 1\Big) n^{1/9} - (k + l + 1) \log 2 - 2 \log (\eps_n n^{1/3}) \geq  \big(c l^2 - 2\big) n^{1/9} \, ,
	\]
	uniformly in $k, K, l$.
	Therefore, we get that
	\[
	\proba{\bar{Z}_{n,\omega}^{>}(K,0) \geq e^{-n^{1/9}}} \leq \sumtwo{k,l = 0}{2^k K \leq \eps_n n^{1/9}}^{+\infty} \proba{\beta n^{1/18} \xi_{k,l} \geq c l^2 - 2} \,.
	\]
	
	Let us define $C_{n,l} \defeq \sup_{|\Delta_n^{x,y}| \in [l,l+1) n^{2/9}} |X^{(2)}_v - X^{(2)}_{c_h-u}| - \beta^{-1} n^{-1/18} \left( c l^2 - 2 \right)$, with again $u=xn^{-1/3}$ and $v=y n^{-1/3}$. Recalling the definition of $\xi_{k,l}$ above we have
	\[
	\proba{\beta n^{1/18} \xi_{k,l} \geq c l^2 - 2} = \PP \Big( \inf_{|u| \in  2^{k} [1,2) K n^{-1/9}}  \{X_{u_*} - X_{u_* + u}\} \leq C_{n,l} \Big) \, .
	\]
	Let us now decompose over the values of $C_{n,l}$. Since $X_{u_*} - X_u \geq 0$, when $C_{n,l} < 0$ the probability equals $0$, so we can intersect with $C_{n,l} \geq 0$. We have
	\[ 
	\begin{split}
		&\proba{\beta n^{1/18} \xi_{k,l} \geq c l^2 - 2} \\
		& \quad \leq \sum_{j = 1}^{+\infty} \proba{\inf_{|u| \in  2^{k} [1,2) K n^{-1/9}}  X_{u_*} - X_{u_* + u} \leq j n^{-1/18}, C_{n,l} \in \left[j-1, j\right) n^{-1/18}}\\
		& \quad \leq \sum_{j = 1}^{+\infty}  \PP \Big(\inf_{|u| \in  2^{k} [1,2) K n^{-1/9}}  X_{u_*} - X_{u_* + u} \leq j n^{-1/18}\Big)^{1/2} \PP \Big(C_{n,l} \in \left[j-1, j\right) n^{-1/18}\Big)^{1/2} \, ,
	\end{split}
	\]
	where we have used Cauchy-Schwartz inequality. 
	
	First, let us treat the last probability: using the Brownian scaling, we have
	\[ \begin{split}
		\proba{C_{n,l} \in \left[j-1, j\right) n^{-1/18}} &= \PP \Big(\sup_{r \in \left[l,l+1\right)} |X^{(2)}_r| - \beta^{-1} (c l^2 - 2) \in \left[j-1, j\right)\Big)\\
		&\leq \PP \Big(\sup_{r \in \left[l,l+1\right)} |X^{(2)}_r| \geq j -1 +  \beta^{-1}(c l^2 - 2) \Big) \,.
	\end{split}  \]
	We can get a bound on this probability using usual Gaussian bounds and the reflection principle:
	\[ \PP \Big( \sup_{r \in \left[l,l+1\right)} |X^{(2)}_r| \geq \alpha \Big) \leq \PP \Big( \sup_{r \in \left[0,l+1\right)} |X^{(2)}_r| \geq \alpha \Big) \leq 2e^{-\frac{\alpha^2}{2(l+1)}} \,. \] 
	Then, we substitute $\alpha$ with $j -1+ \beta^{-1} (c l^2 - 2)$ to get the upper bound
	\[
	\proba{C_{n,l} \in \left[j-1, j\right) n^{-1/18}} \leq 2 e^{- c \frac{l^4}{l+1}} e^{- \frac{(j-c')^2}{2(l+1)}} e^{- c \frac{l^2}{l+1} (j-c')} \,,
	\]
	for some constants $c,c'$ (that depend only on $h,\beta$).
	
	For the other probability, with the argument explained previously (since $K$ and $j$ can again be taken up to a positive multiplicative constant), we only need to get a bound on
	\begin{equation}\label{eq:reflexion-apply}
		\PP \Big( \inf_{u \in  2^{k} [1,2) K n^{-1/9}} \mathcal{M}^+_u \leq j n^{-1/18} \Big) \,.
	\end{equation}
	For $\sigma = -1$ we can do the same reasoning, thus we only need to get a bound for~\eqref{eq:reflexion-apply}.
	We use Corollary \ref{cor-meander}: for any $\lambda > 0$, we have
	\[ \PP \Big( \inf_{u \in  [s,t]} \mathcal{M}^+_u \leq a \Big) \leq  \frac{8 \lambda a}{\sqrt{\pi s}}\left( 1 \wedge \frac{(\lambda a)^2}{2s} \right) + (cst.) \frac{a \sqrt{t}}{t-s} \frac{e^{-\frac{2}{t-s} a^2 (\lambda-1)^2}}{1 - e^{-\frac{2}{t-s} a^2 \lambda^2}} \,, \]
	which translates for $\lambda = (2^k K)^{1/3}$ to
	\[ \begin{split}
		\PP \Big( \inf_{u \in  2^{k} [1,2) K n^{-1/9}} \mathcal{M}^+_u \leq j n^{-1/18} \Big) &\leq \frac{8 j (2^k K)^{1/3}}{\sqrt{\pi 2^k K}} + \frac{j \sqrt{2^k K}}{2^k K} \frac{(cst.)}{1 - e^{- 2 (2^k K)^{2/3} \frac{j^2}{2^k K }}}\\
		&\leq \frac{16 j}{\sqrt{\pi} (2^k K)^{1/6}} + (cst.) \frac{j}{2^k K} (2^k K)^{5/6} \leq (cst.) \frac{j}{(2^k K)^{1/6}}
	\end{split} \]
	where we used that $\frac{1}{x(1-e^{-\alpha/x})}$ is bounded for $x \geq 1$, uniformly in $\alpha \geq 1$ (note that $j \geq 1$).
	
	Together with the above, this yields the following upper bound for $2^k K n^{-1/9} \leq \eps_n < \frac12$:
	\begin{equation}\label{eq:prop-restrict-2}
		\proba{\beta n^{1/6} \xi_{k,l} \geq n^{1/9} \left( c l^2 - 2 \right)} \leq \frac{(cst.)}{(2^k K)^{1/12}}\, e^{- c l^3} \sum_{j = 1}^{+\infty} j e^{- \frac{(j-c')^2}{2(l+1)}} e^{- c l (j-c')} \, .
	\end{equation} 
	
	The sum on $j \geq 1$ is bounded from above by $c'' (l+1)^3$ (where the constant $c''$ does not depend on $l\geq 0$), so we finally get
	\[ \proba{\bar{Z}_{n,\omega}^{>}(K,0) \geq e^{-n^{1/9}}} \leq \sum_{k,l = 0}^{+\infty} \proba{\beta n^{1/18} \xi_{k,l} \geq c l^2 - 2} \leq \frac{(cst.)}{K^{1/12}} \sum_{k,l = 0}^{+\infty} \frac{1}{2^{k/6}} (l+1)^3 e^{- c l^3} \,. \]
	The lemma follows since the last sum is finite.
\end{proof}

\begin{proof}[Proof of Proposition \ref{th:restrict-1/9}]
	Use Proposition \ref{prop-couplage} and consider some $n \geq n_0(\omega)$ such that $\eps_n < \delta(\omega)$ (recall the definitions in Proposition \ref{prop-couplage}). Then, using the same notation as in the proofs of Lemmas \ref{lem:excl-traj-K},\ref{lem:excl-traj-L}, we have
	\[ \begin{split}
		-\inf_{|u| \in  2^{k} [1,2) K} n^{1/18} \{X_{u_*} - X_{u_* + u n^{-1/9}}\} = -\inf_{|u| \in  2^{k} [1,2) K} \mathbf{B}_u \, .
	\end{split} \]
	On the other hand,
	\[ X^{(2)}_v - X^{(2)}_{c_h - u} = \frac12 \left( Y_{c_h-v} - X_{c_h-v} + Y_u - X_u \right) \]
	since $c_h - v = \Delta_n^{x,y}n^{-1/3} + u$ and $u = u_* + s$ with $s \leq \eps_n \leq \delta_0(\omega)$, this is therefore equal to $\frac12 \left( \mathbf{Y}_{c_h-v} - \mathbf{X}_{c_h-v} + \mathbf{Y}_u - \mathbf{X}_u \right)$, which means that
	\[ n^{1/18} \sup_{c_h-u-v \in n^{-1/9}} |X^{(2)}_v - X^{(2)}_{c_h - u}| = \frac12 \sup_{r \in [l,l+1) , s \leq \delta} \left( \mathbf{Y}_{r+s} - \mathbf{B}_{r+s} + \mathbf{Y}_s - \mathbf{B}_s \right) \, .  \]
	Thus, we see that the random quantities $\xi_{k,l}$ and $\mathcal{X}^{(2)}_l$ do not depend on $n$ when $n$ is large enough (meaning $n \geq n_0$), thus
	\[ \varlimsup_{n \to \infty} \frac{1}{n^{1/9}} \log \bar{Z}_{n,\omega}^{>}(K,0) = \frac{1}{n_0^{1/9}} \log \bar{Z}_{n_0,\omega}^{>}(K,0) \, , \] 
	the same being true for $\bar{Z}_{n,\omega}^{>}(0,L)$. This proves the proposition by applying Lemmas \ref{lem:excl-traj-L}, \ref{lem:excl-traj-K} to $n = n_0$.
\end{proof}

\subsection{Convergence of the restricted log partition function}

In this section we study the convergence of $n^{-1/9} \log \bar{Z}_{n,\omega}^{\leq}(K,L)$ for fixed $K$ and $L$ (large), in which we recall that $\bar{Z}_{n,\omega}^{\leq}(K,L) \defeq \bar{Z}_{n,h}^{\omega,\beta} \big( |\Delta_n| \leq L n^{2/9}, | M_n^-+u_*n^{1/3}| \leq K n^{2/9} \big)$. It is a bit more convenient to transform the condition $|\Delta_n| \leq L n^{2/9}$ into the condition $|M_n^+ - (c_h - u_*)n^{1/3}| \leq L n^{2/9}$, which restricts to the same trajectories after adjusting the value of $L$. Finally, since we plan to take the limit for $K,L \to \infty$
it is enough to treat the case where $K = L$. Thus, we define
\[ \bar{\mathcal{Z}}_{n,\omega}^{\leq K} \defeq \bar{Z}_{n,h}^{\omega,\beta} \big( | M_n^-+u_*n^{1/3}| \leq K n^{2/9}, |M_n^+ - (c_h - u_*)n^{1/3}| \leq K n^{2/9} \big) \, . \]
Note that for any $K > 1$, if we recall the definition \eqref{eq:def-Z>K} of $\bar{\mathcal{Z}}_{n,\omega}^{> K}$, we have
\[ \bar{Z}_{n,h}^{\omega,\beta} \big( |\Delta_n| \leq \eps_n n^{1/3} , |M_n^- + u_* n^{1/3}| \leq \eps_n n^{1/3} \big) = \bar{\mathcal{Z}}_{n,\omega}^{\leq K} + \bar{\mathcal{Z}}_{n,\omega}^{> K} \, . \]

As explained in the beginning of this section, as $K \to \infty$, $\bar{\mathcal{Z}}_{n,\omega}^{\leq K}$ contains all the relevant trajectories giving the main contribution to the partition function.

\begin{proposal}\label{prop-cv-1/9}
	For any $h,\beta > 0$ and any $K > 1$, $\PP$-almost surely,
	\[  \lim_{n\to \infty} \frac{\sqrt{2}}{\beta n^{1/9}} \log \bar{\mathcal{Z}}_{n,\omega}^{\leq K} = \sup_{-K \leq u,v \leq K} \left\{ \mathcal{Y}_{u,v} - \frac{3\pi^2}{\beta c_h^4 \sqrt{2}} \big( u + v \big)^2 \right\} \eqdef \mathcal{W}_2^K \, , \]
	with $\mathcal{Y}_{u,v} \defeq  \mathbf{Y}_u - \mathbf{Y}_{-v} - \chi (\mathbf{B}_u + \mathbf{B}_v)$ and $(\mathbf{B},\mathbf{Y})$ as in Proposition \ref{prop-couplage}.
\end{proposal}

\begin{proof}[Proof of Proposition \ref{prop-cv-1/9}]
	For any $\delta > 0$, we define the following subsets of $\NN$
	\begin{equation}\label{eq-def-intervalles}
		\mathscr{C}^-_{n,\delta}(k_1) \defeq \mathset{ x \, : \, \left\lfloor \frac{x - u_*n^{1/3}}{\delta n^{2/9}} \right\rfloor = k_1 } \, , \quad \mathscr{C}^+_{n,\delta}(k_2) \defeq \mathset{ y \, : \, \left\lfloor \frac{y - (c_h - u_*)n^{1/3}}{\delta n^{2/9}} \right\rfloor = k_2 }
	\end{equation}
	as well as $\mathscr{C}_{n,\delta}(k_1,k_2) \defeq \mathscr{C}^-_{n,\delta}(k_1) \times \mathscr{C}^+_{n,\delta}(k_2)$. Recall \eqref{partition2} and the notation
	\[ 
	\Omega_{n}(x,y) = \Omega_{n}^{x,y} = - \left( X_{u_*} - X_{xn^{-1/3}} \right) + X^{(2)}_{yn^{-1/3}} - X^{(2)}_{c_h-x n^{-1/3}} \, . 
	\]
	Similarly to the proof of Theorem \ref{th-1/6}, we define
	\[ \bar{\Lambda}_{n,h}^{\omega,\beta}(K,\delta) \defeq \sum_{k_1 = -K/\delta}^{K/\delta} \sum_{k_2 = -K/\delta}^{K/\delta} \bar Z_{n,h}^{\omega,\beta} (k_1,k_2,\delta) \, , \]
	where
	\begin{equation}\label{eq:Z-delta-1/9}
		\bar Z_{n,h}^{\omega,\beta} (k_1,k_2,\delta) \defeq  \sum_{(x,y) \in \mathscr{C}_{n,\delta}(k_1,k_2)} \exp \left( - \beta n^{1/6} \Omega_n^{x,y} -  \hat c_{h} \frac{(\Delta_n^{x,y})^2}{n^{1/3}}(1 + \bar{o}(1)) \right) \, .
	\end{equation}
	with $\hat c_h = \frac{3\pi^2}{2c_h^4}$.
	Then, we can write
	\[ \log \bar{\mathcal{Z}}_{n,\omega}^{\leq K} = \log \big(1 + \bar{o}(1)\big)\psi_h \sin \left( \frac{\pi u_*}{c_h} \right) + \log \bar{\Lambda}_{n,h}^{\omega,\beta}(K,\delta) \, . \]
	Note that both $\bar{o}(1) \to 0$ are the deterministic quantities mentioned in Section \ref{sec:part-homogene-desordre}.
	Again, we only have to get bounds on the maximum of $\bar Z_{n,h}^{\omega,\beta} (k_1,k_2,\delta)$, as
	\begin{equation}\label{eq:Lambda-1/9}
		0 \leq \log \bar{\Lambda}_{n,h}^{\omega,\beta}(K,\delta) - \max_{-\frac{K}{\delta} \leq k_1,k_2 \leq \frac{K}{\delta}} \log \bar Z_{n,h}^{\omega,\beta} (k_1,k_2,\delta) \leq 2 \log \frac{4K}{\delta} \, ,
	\end{equation}
	and $n^{-1/9} 2 \log \frac{4K}{\delta}$ goes to $0$ as $n\to\infty$.
	
	Now, if we factorize $\bar{Z}_{n,h}^{\omega,\beta}(k_1,k_2,\delta)$ by the contribution of $x = \hat x_{k_1} := u_*n^{1/3} + k_1\delta n^{2/9}$ and $y=\hat y_{k_2}:=(c_h-u_*)n^{1/3} + k_2\delta n^{2/9}$ respectively, we have
	\[ \bar{Z}_{n,h}^{\omega,\beta}(k_1,k_2,\delta) = e^{\Xi_n(k_1,k_2,\delta)} \sum_{(x,y) \in \mathscr{C}_{n,\delta}(k_1,k_2)} e^{- \beta n^{1/6} \zeta^{k_1,k_2}_{n,\delta} (\frac{x}{n^{1/3}},\frac{y}{n^{1/3}}) + \frac{\hat c_h}{n^{1/3}} \left( (k_1 + k_2)^2 \delta^2 n^{4/9} - (\Delta_n^{x,y})^2 \right)} \]
	where we have defined
	\begin{equation}\label{eq:1/9-grain}
		\Xi_n(k_1,k_2,\delta) \defeq \beta n^{1/6} \Omega_{n}( \hat x_{k_1},\hat y_{k_2} ) - n^{1/9}  \hat c_h \big( k_1 \delta + k_2 \delta \big)^2
	\end{equation}
	and
	\begin{equation}\label{eq:def-zeta}
		\begin{split}
			\zeta^{k_1,k_2}_{n,\delta} (u,v) &\defeq \Omega_n(un^{1/3}, vn^{1/3}) - \Omega_{n}( \hat x_{k_1},\hat y_{k_2} ) \\
			&=X_u - X_{u_* + \frac{k_1\delta}{n^{1/9}}} - X^{(2)}_{c_h - u} + X^{(2)}_{c_h - u_* - \frac{k_1 \delta}{n^{1/9}}} + X^{(2)}_v - X^{(2)}_{c_h - u_* + \frac{k_2 \delta}{n^{1/9}}} \, .
		\end{split}
	\end{equation}
	Finally, define
	\[ 
	\bar{R}_n(k_1,k_2,\delta) \defeq \sup_{(x,y) \in \mathscr{C}_{n,\delta}(k_1,k_2)} \mathset{ \beta n^{1/6} \left| \zeta^{k_1,k_2}_{n,\delta} (\frac{x}{n^{1/3}},\frac{y}{n^{1/3}}) \right|
		+ \frac{\hat c_{h}}{n^{1/3}} \left| (k_1 + k_2)^2 \delta^2 n^{4/9} - (\Delta_n^{x,y})^2 \right|} \, , 
	\]
	then we have
	\begin{equation}\label{eq:reste-1/9}
		\left| \log \bar{Z}_{n,h}^{\omega,\beta}(k_1,k_2,\delta) - \Xi_n(k_1,k_2,\delta) \right| \leq \bar{R}_n(k_1,k_2,\delta) + \log |\mathscr{C}_{n,\delta}(k_1,k_2)|.
	\end{equation}
	Since $n^{-1/9} \log |\mathscr{C}_{n,\delta}(k_1,k_2)| \to 0$ as $n \to \infty$, in the rest of the proof we have to control $n^{-1/9} \bar{R}_n(k_1,k_2,\delta)$ and then prove the convergence of $n^{-1/9} \Xi_n(k_1,k_2,\delta)$. Afterwards we will plug those convergences in \eqref{eq:Lambda-1/9} and \eqref{eq:reste-1/9} to prove Proposition \ref{prop-cv-1/9}.

	\medskip\noindent
	\textit{Control of $\bar{R}_n(k_1,k_2,\delta)$. \ }
	We now seek a bound on $\bar{R}_n(k_1,k_2,\delta)$. First we have 
	\[ \begin{split}
		\left| n^{1/9} \big( k_1 \delta + k_2 \delta \big)^2 - \frac{(\Delta_n^{x,y})^2}{n^{1/3}} \right| &\leq 2 \frac{k_1\delta + k_2 \delta}{n^{1/
				9}} \left| x + y - c_h n^{1/3} - k_1\delta n^{2/9} - k_2\delta n^{2/9} \right|\\
		& \leq 4 |k_1 + k_2| \delta^2 n^{1/9}  \leq 4K \delta n^{1/9}\, .
	\end{split} \]
	
	To control the random part $\zeta^{k_1,k_2}_{n,\delta} (x n^{-1/3},y n^{-1/3})$, we use the following proposition, that we prove afterwards.
	
	\begin{proposal}\label{prop-meandre-intervalle}
		Let $\delta_j = 2^{-j}, j \in \NN$, then, $\PP$-almost surely, there exists a positive $C_\omega$ such that for any $n$ and $j$ large enough, any $k_1,k_2 \in \llbracket - \frac{K}{\delta_j}, \frac{K}{\delta_j} \rrbracket$, we have
		\[ n^{1/6} \sup_{(x,y) \in \mathscr{C}_{n,\delta_j}(k_1,k_2)} \left| \zeta^{k_1,k_2}_{n,\delta} (x n^{-1/3},y n^{-1/3}) \right| \leq C_\omega \delta_j^{1/4} n^{1/9} \, . \]
	\end{proposal}
	
	We will still denote this parameter by $\delta$ while keeping in mind that $\delta \to 0$ along a specific sequence.
	Assembling these results, we see that
	\begin{equation}\label{eq:borne-reste-1/9}
		\begin{split}
			\frac{1}{n^{1/9}} \left| \bar{R}_n(k_1,k_2,\delta) \right| &\leq \beta C_\omega \delta^{1/4} + 4 \hat{c}_h K \delta \eqdef \eps(\omega,\delta) \,.
		\end{split} 
	\end{equation}
	Thus, $n^{-1/9} \left| \bar{R}_n(k_1,k_2,\delta) \right|$ is bounded by a function $\eps(\omega,\delta)$ that goes to $0$ as $\delta \downarrow 0$ uniformly in $-\frac{K}{\delta} \leq k_1,k_2 \leq \frac{K}{\delta}$, for almost all $\omega$.
	
	\medskip\noindent
	\textit{Convergence of $n^{-1/9} \Xi_n(k_1,k_2,\delta)$. \ } 
	As in the proof of Theorem \ref{th-1/6} we write $u = k_1 \delta$ and $v = k_2 \delta$ in \eqref{eq:1/9-grain}:
	recalling the definition~\eqref{fact-1/9} of $\Omega$, this leads to
	\begin{equation}\label{eq-th-1/9-pre-cv}
		\frac{\Xi_n(k_1,k_2,\delta)}{\beta n^{1/9}} =  - n^{1/18} \big( X_{u_*} - X_{u_*+\frac{u}{n^{1/9}}} \big) - c_{h,\beta} \big( u + v \big)^2 + n^{1/18} \big( X^{(2)}_{c_h - u_* + \frac{v}{n^{1/9}}} - X^{(2)}_{c_h - u_* - \frac{u }{n^{1/9}}} \big),
	\end{equation}
	with $c_{h,\beta} =\beta^{-1} \hat c_h$.

	\par 
	Recall Proposition \ref{prop-couplage} and its notation.
	Set
	\[ X_u = X^{(1)}_u + X^{(2)}_{c_h-u} \,, \qquad Y_u = X^{(1)}_u - X^{(2)}_{c_h-u}, \]
	and denote by $\mathbf{B}$ a two-sided three-dimensional Bessel process and by $\mathbf{Y}$ a standard Brownian motion, independent from $\mathbf{B}$. 
	Then for any $n \geq n_0(\omega)$,
	\[ n^{1/18} \big( X_{u_*} - X_{u_*+\frac{u}{n^{1/9}}} \big) =  \sqrt{2} \chi \mathbf{B}_u \,,\]
	with $\chi=\chi(\omega, u) = \left(\sqrt{c_h - u_*} \indic{u \geq 0} + \sqrt{u_*} \indic{u < 0} \right)^{-1}$,
	and
	\[ \begin{split}
		n^{1/18} \big( X^{(2)}_{c_h - u_* + \frac{v}{n^{1/9}}} - X^{(2)}_{c_h - u_* - \frac{u}{n^{1/9}}} \big) &= \frac{n^{1/18}}{2} \left( X_{u_* - \frac{v}{n^{1/9}}} - Y_{u_* - \frac{v}{n^{1/9}}} - X_{u_* + \frac{u}{n^{1/9}}} + Y_{u_* + \frac{u}{n^{1/9}}} \right)\\ &= \frac{1}{\sqrt{2}} \Big( \chi(\mathbf{B}_u - \mathbf{B}_v) + \mathbf{Y}_u - \mathbf{Y}_{-v} \Big) \, .
	\end{split} \]
	Assembling those results with \eqref{eq-th-1/9-pre-cv}, we established the following convergence (which is an identity for $n\geq n_0(\omega)$):
	\begin{equation}\label{eq:Xi-cv}
		\frac{\Xi_n(k_1,k_2,\delta)}{\beta n^{1/9}} \xrightarrow[n \to \infty]{\PP-a.s.} \frac{1}{\sqrt{2}} \left( -\chi(\mathbf{B}_u +\mathbf{B}_v) + \mathbf{Y}_u - \mathbf{Y}_{-v} \right) - c_{h,\beta} \big( u + v \big)^2 \,.
	\end{equation}
	
	\medskip\noindent
	\textit{Conclusion of the proof. \ }
	If we define $\mathcal{Y}_{u,v} \defeq \mathbf{Y}_u - \mathbf{Y}_{-v} - \chi(\mathbf{B}_u + \mathbf{B}_v)$, combining \eqref{eq:borne-reste-1/9} and \eqref{eq:Xi-cv} with \eqref{eq:reste-1/9} proves
	\begin{equation}\label{eq:cv-z-bar-1/9}
		\varlimsup_{n\to \infty} \left| \frac{1}{\beta n^{1/9}} \log \bar{Z}_{n,h}^{\omega,\beta}(k_1,k_2,\delta) - \mathcal{Y}_{k_1 \delta,k_2\delta} + c_{h,\beta} (k_1 \delta + k_2 \delta)^2 \right| \leq \eps(\omega,\delta).
	\end{equation}
	Then, \eqref{eq:Lambda-1/9} and \eqref{eq:cv-z-bar-1/9} lead to
	\[ \varlimsup_{n\to \infty} \frac{\log \bar{\mathcal{Z}}_{n,\omega}^{\leq K}}{\beta n^{1/9}} \leq \sup_{-K \leq u,v \leq K} \left\{ \frac{1}{\sqrt{2}} \mathcal{Y}_{u,v} - \frac{3\pi^2}{2 \beta c_h^4} \big( u + v \big)^2 + \eps(\omega,\delta) \right\} \quad \text{$\PP$-a.s.} \, , \]
	and, using the uniform continuity of $\mathcal{Y}_{u,v}$ and of $(u+v)^2$ on $[-K,K]^2$, we have
	\[ \varliminf_{n\to \infty} \frac{\log \bar{\mathcal{Z}}_{n,\omega}^{\leq K}}{\beta n^{1/9}} \geq \sup_{-K \leq u,v \leq K} \left\{ \frac{1}{\sqrt{2}} \mathcal{Y}_{u,v} - \frac{3\pi^2}{2 \beta c_h^4} \big( u + v \big)^2 - \eps(\omega,\delta) \right\} \quad \text{$\PP$-a.s.} \]
	Finally, letting $\delta$ go to $0$ proves the convergence of $n^{-1/9} \log \bar{\mathcal{Z}}_{n,\omega}^{\leq K}$.
\end{proof}

\begin{proof}[Proof of Proposition \ref{prop-meandre-intervalle}]
	Recall the definition \eqref{eq-def-intervalles} of $\mathscr{C}^-_{n,\delta}(k_1)$ and $\mathscr{C}^+_{n,\delta}(k_2)$ as well as the notation $\mathscr{C}_{n,\delta}(k_1,k_2) = \mathscr{C}^-_{n,\delta}(k_1) \times \mathscr{C}^+_{n,\delta}(k_2)$
	The proof essentially boils down to the following lemma and a use of Borel-Cantelli lemma.
	
	\begin{lemma}\label{lem:proba-A}
		There exists some positive constants $\lambda,\mu$ such that for any $\delta \in (0,1)$ and any $C\geq 1$,
		\[ \sup_{n \geq 1} \sup_{-\frac{K}{\delta} \leq k_1,k_2 \leq \frac{K}{\delta}} \proba{\sup_{(x,y) \in \mathscr{C}_{n,\delta}(k_1,k_2)}  \left| \zeta^{k_1,k_2}_{n,\delta} (x n^{-1/3},y n^{-1/3}) \right| \geq C \frac{\delta^{1/4}}{n^{1/18}}} \leq \mu e^{- \lambda \frac{C}{\delta^{1/4}}} . \]
	\end{lemma}
	
	Using Lemma \ref{lem:proba-A} and a union bound immediately yields
	\[ \sup_{n \geq 1} \proba{\sup_{-\frac{K}{\delta} \leq k_1,k_2 \leq \frac{K}{\delta}} \sup_{(x,y) \in \mathscr{C}_{n,\delta}(k_1,k_2)} \Big| \zeta^{k_1,k_2}_{n,\delta} (u,v) \Big| \geq C \frac{\delta^{1/4}}{n^{1/18}}} \leq \left(\frac{K}{\delta}\right)^2 \mu e^{- \lambda \frac{C}{\delta^{1/4}}}. \]
	Summing over $\delta_j = 2^{-j}$ gives a bound which is summable in $C$: this allows us to use a Borel-Cantelli lemma. This means that with $\PP$-probability $1$, there is a positive $C_\omega$ such that for all $j \geq 0$, for all $-K 2^j \leq k_1,k_2 \leq K 2^j$, for all $(x,y) \in \mathscr{C}_{n,\delta}(k_1,k_2)$, we have $\zeta^{k_1,k_2}_{n,\delta} \left(\frac{x}{n^{1/3}},\frac{y}{n^{1/3}}\right) \leq C_\omega \delta^{1/4} n^{-1/18}$, thus proving the proposition.
\end{proof}

\begin{proof}[Proof of Lemma \ref{lem:proba-A}]
	Recall the definition \eqref{eq:def-zeta}
	\[ \zeta^{k_1,k_2}_{n,\delta} (u,v) = X_u - X_{u_* + \frac{k_1\delta}{n^{1/9}}} - X^{(2)}_{c_h - u} + X^{(2)}_{c_h - u_* - \frac{k_1 \delta}{n^{1/9}}} + X^{(2)}_v - X^{(2)}_{c_h - u_* + \frac{k_2 \delta}{n^{1/9}}} \]
	and that $2 X^{(2)}_{c_h - t} = X_t - Y_t$. Then we can rewrite $\zeta^{k_1,k_2}_{n,\delta} (u,v)$ as
	\[ \begin{split}
		\zeta^{k_1,k_2}_{n,\delta} (u,v) = X_u - X_{u_* + \frac{k_1\delta}{n^{1/9}}} - \frac{X_u - Y_u}{2}&+ \frac{X_{c_h - v} - Y_{c_h - v}}{2}\\
		& + \frac{X_{u_* + \frac{k_1 \delta}{n^{1/9}}} - Y_{u_* + \frac{k_1 \delta}{n^{1/9}}}}{2} - \frac{X_{u_* - \frac{k_2 \delta}{n^{1/9}}} - Y_{u_* - \frac{k_2 \delta}{n^{1/9}}}}{2} \, ,
	\end{split} \]
	which simplifies to
	\[ \begin{split}
		2 \zeta^{k_1,k_2}_{n,\delta} (u,v) &= X_u + Y_u - (X_{u_* + \frac{k_1 \delta}{n^{1/9}}} + Y_{u_* + \frac{k_1 \delta}{n^{1/9}}}) + X_{c_h - v} - Y_{c_h - v} - (X_{u_* - \frac{k_2 \delta}{n^{1/9}}} - Y_{u_* - \frac{k_2 \delta}{n^{1/9}}})\\
		&= X_u - X_{u_* + \frac{k_1 \delta}{n^{1/9}}} + X_{c_h - v} - X_{u_* - \frac{k_2 \delta}{n^{1/9}}} + Y_u - Y_{u_* + \frac{k_1 \delta}{n^{1/9}}} + Y_{u_* - \frac{k_2 \delta}{n^{1/9}}} - Y_{c_h - v}\, .
	\end{split} \]
	We split $\zeta^{k_1,k_2}_{n,\delta} (u,v)$ into four parts corresponding to the terms in $X$ and those in $Y$:
	\begin{itemize}
		\item $X_u - X_{u_* + \frac{k_1 \delta}{n^{1/9}}}$ and $X_{c_h - v} - X_{u_* - \frac{k_2 \delta}{n^{1/9}}}$ which we call ``meander parts'' because of \eqref{eq:X-meandre}
		\item $|Y_u - Y_{u_* + \frac{k_1 \delta}{n^{1/9}}}|$ and $|Y_{u_* - \frac{k_2 \delta}{n^{1/9}}}- Y_{c_h - v}|$ which we call ``Brownian parts''.
	\end{itemize}
	We use a union bound to separately control the probability for each increment to be greater than $\frac{C \delta^{1/4}}{8 n^{1/18}}$.
	
	\medskip
	\noindent
	\textit{Control of the Brownian parts. \ }
	First, recall that by construction, $X$ and $Y$ are independent. This also means that since $u_*$ is $X$-measurable, $u_*$ and $Y$ are also independent. Thus, the Brownian reflection principle yields
	\[ \sup_{u n^{1/3} \in \mathscr{C}_{n,\delta}^-(k_1)}  \Big| Y_u - Y_{u_* + \frac{k_1 \delta}{n^{1/9}}} \Big| \overset{(d)}{=} \sup_{v n^{1/3} \in \mathscr{C}_{n,\delta}^+(k_2)} \Big| Y_{u_* - \frac{k_2 \delta}{n^{1/9}}} - Y_{c_h - v} \Big| \overset{(d)}{=} |W_\frac{\delta}{n^{1/9}}| \, , \]
	where $W$ is a standard Brownian motion. This leads us to
	\[ \begin{split}
		\proba{\sup_{(u,v) n^{1/3} \in \mathscr{C}_{n,\delta}(k_1,k_2)} \Big| Y_u - Y_{u_* + \frac{k_1 \delta}{n^{1/9}}}\Big| \geq \frac{C \delta^{1/4}}{8 n^{1/18}}}
		\leq \proba{|W_\frac{\delta}{n^{1/9}}| \geq \frac{C \delta^{1/4}}{8 n^{1/18}}} \leq  e^{-\frac{C^2}{128 \sqrt{\delta}}} \, ,
	\end{split} \]
	and similarly for $|Y_{u_* - \frac{k_2 \delta}{n^{1/9}}} - Y_{c_h - v}|$.

	\medskip
	\noindent
	\textit{Control of the meander parts. \ }
	We have to bound the following:
	\begin{equation}\label{eq:proba-meandre-union}
		\PP \Bigg( \sup_{u n^{1/3} \in \mathscr{C}_{n,\delta}^-(k_1)} |X_u - X_{u_* + \frac{k_1 \delta}{n^{1/9}}}| \geq \frac{C \delta^{1/4}}{8 n^{1/18}} \Bigg)
	\end{equation}
	and similarly for $|X_{c_h - v} - X_{u_* - \frac{k_2 \delta}{n^{1/9}}}|$;
	we will focus on bounding~\eqref{eq:proba-meandre-union} since the other bound follows from it\footnote{Observe that if $v n^{1/3} \in \mathscr{C}_{n,\delta}^+(k_2)$, writing $\tilde v \defeq c_h - v$ and assuming $v \neq c_h - u_* + \frac{k_2\delta}{n^{1/9}}$, we have
		$\left\lfloor \frac{\tilde v n^{1/3} - u_*n^{1/3}}{\delta n^{2/9}} \right\rfloor = \left\lfloor - \frac{v n^{1/3} - (c_h- u_*)n^{1/3}}{\delta n^{2/9}} \right\rfloor = - 1 - k_2  , $
		thus $(c_h - v)n^{1/3} \in \mathscr{C}_{n,\delta}^+(-1-k_2)$.}.
	Recall \eqref{eq:X-meandre} to get
	\[ \sup_{u n^{1/3} \in \mathscr{C}_{n,\delta}^-(k_1)} |X_u - X_{u_* + \frac{k_1 \delta}{n^{1/9}}}| = \sup_{u n^{1/3} \in \mathscr{C}^-_{n,\delta}(k_1)} \chi \left| \mathcal{M}_{\frac{u - u_*}{\chi^2}} - \mathcal{M}_{\frac{k_1\delta n^{-1/9}}{\chi^2}} \right| \, . \]
	Observe that $\chi=\chi(u,\omega)$ is a constant that only depends on the sign of $k_1$ and that since $K$ and $C$ are arbitrary chosen, we only need to get a bound on the probability of $\sup_{u n^{1/3} \in \mathscr{C}^-_{n,\delta}(k_1)} | \mathcal{M}_{u - u_*} - \mathcal{M}_{\frac{k_1\delta}{n^{1/9}}}|$ being greater than $\frac{C\delta^{1/4}}{8 n^{1/18}}$.
	
	\medskip
	Without any loss of generality we can suppose that $k_1 \geq 0$: to get the case where $k_1 \leq 0$ we only need to do the same proof with $|k_1 + 1|$ instead.
	Use Lemma~\ref{reflexion} and Markov's inequality to get
	\[ \begin{split}
		\proba{\sup_{u n^{1/3} \in \mathscr{C}^-_{n,\delta}(k_1)} \left| \mathcal{M}_{u - u_*} - \mathcal{M}_{\frac{k_1\delta}{n^{1/9}}} \right| \geq C \frac{\delta^{1/4}}{8 n^{1/18}}} &\leq 4 \proba{\mathcal{M}_{\frac{(k_1+1)\delta}{n^{1/9}}} - \mathcal{M}_{\frac{k_1\delta}{n^{1/9}}} \geq\frac{C \delta^{1/4}}{8 n^{1/18}}}\\
		&\leq 4 \esp{e^{\frac{n^{1/18}}{\sqrt{\delta}} (\mathcal{M}_{\frac{(k_1+1)\delta}{n^{1/9}}} - \mathcal{M}_{\frac{k_1\delta}{n^{1/9}}})}} e^{-\frac{C}{8} \delta^{\frac14}} .
	\end{split} \]
	We show below that there is a constant $c = c(K)>0$ such that, for $n$ large enough
	\begin{equation}
		\label{eq:exp-moment-meander}
		\esp{e^{\frac{n^{1/18}}{\sqrt{\delta}} (\mathcal{M}_{\frac{(k_1+1)\delta}{n^{1/9}}} - \mathcal{M}_{\frac{k_1\delta}{n^{1/9}}})}} \leq c 
	\end{equation}
	uniformly in $k_1 \in \{0,1,\ldots, K/\delta \}$ and $0 < \delta < 1$, thus proving Lemma~\ref{lem:proba-A}.
\end{proof}

\begin{proof}[Proof of~\eqref{eq:exp-moment-meander}]
	We want to get an upper bound on quantities $\esp{e^{\alpha (\mathcal{M}_v - \mathcal{M}_u)}}$ for specific $u < v, \alpha > 0$. In order to do so, we first condition on the value of $\mathcal{M}_u$ and use the transition probabilities of the Brownian meander to get an upper bound, which we then integrate with respect to the law of $\mathcal{M}_u$.
	Let us set $\kappa_\delta^n \defeq k_1 \delta n^{-1/9}$ with $ k_1 \neq 0$ (we treat the case $k_1=0$ at the end) and $\alpha:= n^{1/18}/\sqrt{\delta}$.
	
	\medskip
	\textit{Step 1: Meander increment conditioned on $\mathcal{M}_{\kappa_\delta^n} \eqdef x$. \ } 
	We write $\varphi_t(x) \defeq \frac{1}{\sqrt{2 \pi t}} e^{-\frac{x^2}{2t}}$ and $\Phi_t(y) \defeq \int_0^y \varphi_t(x) dx$. 
	Using \eqref{meandre-transition}, the density of an increment between time $\kappa_\delta^n \neq 0$ and a time $u = (k_1+1)\delta n^{-1/9} = \kappa_{\delta}^n + \alpha^{-2}$, when starting at $M_{\kappa_\delta^n}=x$, is given by
	\[ \left[ \varphi_{u-\kappa_\delta^n}(m) - \varphi_{u-\kappa_\delta^n}(m+2x) \right] \frac{\Phi_{1 - u}(m+x)}{\Phi_{1 - \kappa_\delta^n}(x)} \indic{m \geq -x} \,. \]
	Then we use that (recall that $u-\kappa_{\delta}^n =\alpha^{-2}$)
	\[ \varphi_{u-\kappa_\delta^n}(m) - \varphi_{u-\kappa_\delta^n}(m+2x) = \frac{1}{\sqrt{2\pi (u-\kappa_\delta^n)}} \left( e^{-\frac{m^2}{2(u-\kappa_\delta^n)}} - e^{-\frac{(m+2x)^2}{2(u-\kappa_\delta^n)}} \right) \leq \frac{\alpha}{\sqrt{2\pi}} e^{-\frac{\alpha^2 m^2}{2}} \,, \]
	and since $u$ is taken close to $u_*$ (recall $|u-u_*| \leq \eps_n$),
	\[ \frac{\Phi_{1 - u}(m+x)}{\Phi_{1 - \kappa_\delta^n}(x)} \leq  \frac{x+m}{x} e^{\frac{x^2}{2(1-\kappa_\delta^n)}} \sqrt{\frac{1 - \kappa_\delta^n}{1 - u}} \leq (cst.) \Big( 1+\frac{m}{x}\Big) e^{\frac{x^2}{2(1-\kappa_\delta^n)}} \, . \]
	Thus we have to bound
	\begin{equation*}\label{eq:prop-control-int}
		\esp{\exp \Big( \alpha (\mathcal{M}_{\frac{(k_1+1)\delta}{n^{1/9}}} - \mathcal{M}_{\kappa_\delta^n}) \Big) \, \Big| \, \mathcal{M}_{\kappa_\delta^n} = x} \leq (cst.) \frac{\alpha}{\sqrt{2\pi}} e^{\frac{x^2}{2(1-\kappa_\delta^n)}} \int_{-x}^\infty  \Big(1+\frac{m}{x} \Big) e^{\alpha m} e^{-\frac{\alpha^2m^2}{2}} \, \dd m \,.
	\end{equation*}
	Now, setting $\Psi(m) \defeq e^{- \frac{(\alpha m)^2}{2} + \alpha m}$, after integrating by parts (writing $m\Psi(m) = (m e^{- \frac{(\alpha m)^2}{2}}) \times e^{\alpha m}$) we can rewrite the above as
	\[ (cst.)\frac{\alpha e^{\frac{x^2}{2(1-\kappa_\delta^n)}}}{\sqrt{2\pi}} \left[ \int_{-x}^\infty \Psi(m)  \, \dd m  + \frac{2}{x \alpha^2} \left( \Psi(-x) + \alpha \int_{-x}^\infty \Psi(m)  \, \dd m \right) \right] \, . \]
	Usual bounds for Gaussian integrals (notice that $\Psi(m)= e^{1/2} e^{- \frac12 (\alpha m -1)^2}$) then yield the upper bound conditioned to $x = \mathcal{M}_{\kappa_\delta^n}$
	\begin{equation*}
		\esp{\exp \left( \alpha (\mathcal{M}_{\frac{(k_1+1)\delta}{n^{1/9}}} - x) \right) \Big| \mathcal{M}_{\kappa_\delta^n} = x } \leq \frac{\alpha e^{\frac{x^2}{2(1-\kappa_\delta^n)}}}{\sqrt{2\pi}}  \left[ \left( 1 + \frac{2}{\alpha x} \right) \sqrt{2\pi} \frac{1}{\alpha}  e^{1/2} + \frac{2\Psi(-x)}{x \alpha^2} \right] \,,
	\end{equation*}
	which we simplify (using that $\Psi(m)\leq e^{1/2}$ for all $m$) as
	\begin{equation}\label{eq:esp-increment-sachant-x}
		\esp{\exp \left( \alpha (\mathcal{M}_{\frac{(k_1+1)\delta}{n^{1/9}}} - x) \right) \Big| \mathcal{M}_{\kappa_\delta^n} = x } \leq (cst.) e^{\frac{x^2}{2(1-\kappa_\delta^n)}} \left[ 1 + \frac{1}{x \alpha} \right] \, .
	\end{equation}
	
	\medskip
	\textit{Step 2: Averaging on $x = \mathcal{M}_{k_1 \delta n^{-1/9}}$. \ } 
	In order to take the expectation in \eqref{eq:esp-increment-sachant-x}, we use the following bounds, given in~\eqref{meander-expmoment} (using that $\sqrt{2\pi} \geq 2$): for $0<a<r/2$,
	\begin{equation}\label{eq:laplace-meandre}
		\esp{e^{a \mathcal{M}_r^2}} \leq \left( 1 - 2ar \right)^{-3/2} \, , \qquad
		\esp{(\mathcal{M}_r)^{-1} e^{a \mathcal{M}_r^2}} \leq \frac{\sqrt{2\pi}}{\sqrt{r(1-r)}} \left(1-2ra\right)^{-1} \, .
	\end{equation}
	
	Recalling that $\kappa_{\delta}^n = k_1\delta n^{-1/9} \leq \eps_n \to 0$, we can use the above to get that  for $n$ sufficiently large, there is a $C > 0$ such that
	\[ \begin{split}
		\esp{e^{\frac{1}{2(1-\kappa_\delta^n)} \mathcal{M}_{\kappa_\delta^n}^2} \Big( 1 + \frac{1}{\alpha \mathcal{M}_{\kappa_\delta^n}}\Big)} &\leq  \left( 1 - \frac{\kappa_\delta^n}{1-\kappa_\delta^n} \right)^{-3/2} + \frac{2\sqrt{\pi}}{\alpha \sqrt{\kappa_\delta^n}} \left( 1 - \frac{\kappa_\delta^n}{1-\kappa_\delta^n} \right)^{-1} \leq C \,,
	\end{split} \]
	recalling that $\alpha \sqrt{\kappa_\delta^{n}} =\sqrt{k_1}\geq 1$.
	This proves the bound~\eqref{eq:exp-moment-meander} in the case $k_1>0$.
	
	\medskip
	\textit{Case $k_1=0$.\ }
	When $k_1 = 0$ we have
	\[ \esp{\exp \left( \alpha \big(\mathcal{M}_{\frac{(k_1+1)\delta}{n^{1/9}}} - \mathcal{M}_{\frac{k_1\delta}{n^{1/9}}} \big) \right)} = \esp{\exp \left( \alpha \mathcal{M}_{\frac{\delta}{n^{1/9}}} \right)} \leq 4\left( 1 - \frac{2}{\alpha} \right)^{-3/2} \, , \]
	where we used the previous bound and the fact that $\delta n^{-1/9}=1/\alpha^2$. Since $\sqrt{\delta} n^{-1/18} = \bar{o}(1)$, this gives the bound~\eqref{eq:exp-moment-meander} when $k_1 = 0$.
\end{proof}

\subsection{Convergence of the log partition function, proof of Theorem~\ref{th-1/9}-\eqref{eq:th-1/9-part}}

\begin{lemma}\label{lem-var2}
	$\PP$-almost surely, there exists a unique $(\mathcal{U},\mathcal{V})$ such that
	\begin{equation}\label{eq:sup-1/9}
		\mathcal{W}_2 \defeq \sup_{u,v} \left\{ \mathcal{Y}_{u,v} - \frac{3\pi^2}{\beta c_h^4 \sqrt{2}} \big( u + v \big)^2 \right\} = \mathcal{Y}_{\mathcal{U},\mathcal{V}} - \frac{3\pi^2}{\beta c_h^4 \sqrt{2}} \big( \mathcal{U}+\mathcal{V} \big)^2 \in (0, +\infty) \, .
	\end{equation}
\end{lemma}

\begin{proof}
	Choose $v = 0$ to get $\mathcal{W}_2 \geq \sup_{u} \big\{  \mathbf{Y}_u -\mathbf{B}_u - c_{h,\beta} \sqrt{2} u^2 \big\}$. We get a positive lower bound since almost surely, there are real numbers $(u_k)\downarrow 0$ such that $\mathbf{Y}_{u_k} \geq 2 \sqrt{u_k}$ and $\mathbf{B}_{u_k}\leq \sqrt{u_k}$. This leads to $\mathcal{W}_2 \geq \sup_{k} \big\{ \sqrt{u_k} - c_{h,\beta} \sqrt{2} u_k^2 \big\} > 0$ almost surely.
	
	In order to show that $\mathcal{W}_2$ is almost surely finite, see $\mathbf{B}_u$ as the modulus of a 3-dimensional standard Brownian motion $W^{(3)}_u$ and consider a one dimensional Wiener process $W$. We use the fact that $t^{-1} (|W^{(3)}_t| + W_t) \to 0$ almost surely to get $\mathcal{Y}_{u,v}/|u+v| \to 0$ as $|u + v| \to \infty$.
	Thus, $\mathcal{Y}_{u,v} - c_{h,\beta} \sqrt{2} (u + v)^2 \leq 0$ $\PP$-a.s. when $|u+v|$ is large enough, meaning that the supremum of this continuous process is almost surely taken on a compact set, thus it is finite.
	The existence of $(\mathcal U, \mathcal V)$ is also a consequence of the continuity of $\mathcal{Y}_{u,v} - c_{h,\beta} \sqrt{2} (u + v)^2$ and of the fact that the supremum is $\PP$-a.s.\ taken on a compact set.
	The uniqueness of the maximum follows from standard methods for Brownian motion with parabolic drift (see \cite[Appendix A.3]{berger2020one}).
\end{proof}

\begin{remark}
	We could have taken another form of $Z_{n,h}^{\omega,\beta} (k_1,k_2,\delta)$ given by (\ref{fact-1/9}), without using the process $X$ that was only useful to reject trajectories whose minimum is too far from $-u_* n^{1/3}$. This would have led us to the alternative form
	\[ \mathcal{W}_2' \defeq \sup_{u,v} \left\{ \bar{X}^{(1)}_u + \bar{X}^{(2)}_v - c_{h,\beta} \big( u + v \big)^2 \right\} = \mathcal{W}_2 \, , \]
	where $\bar{X}^{(i)}_u$ are Brownian-related processes provided by a suitable coupling. However, these limit processes are not independent and their distribution may not be known processes, making $\mathcal{W}_2'$ less exploitable.
\end{remark}

\begin{proof}[Proof of Theorem~\ref{th-1/9}-\eqref{eq:th-1/9-part}]
	We first see that $\PP$-almost surely, by Proposition \ref{prop-cv-1/9} and Lemma \ref{lem-var2}, we have
	\begin{equation}
		\lim_{K \to +\infty} \lim_{n \to +\infty} \frac{\sqrt{2}}{\beta n^{1/9}} \log \bar{\mathcal{Z}}_{n,\omega}^{\leq K} = \lim_{K \to +\infty} \mathcal{W}_2^K = \mathcal{W}_2 > 0 \, .
	\end{equation}
	On the other hand, using Corollary \ref{cor:lem:exclusion-1/9},
	\begin{equation}
		\limsup_{K \to +\infty} \lim_{n \to +\infty} \frac{\sqrt{2}}{\beta n^{1/9}} \log \bar{\mathcal{Z}}_{n,\omega}^{> K} \leq - \frac{\sqrt{2}}{\beta} < 0  \quad \text{$\PP$-a.s.}  \, .
	\end{equation}
	Therefore, $\PP$-almost surely we have
	\begin{equation}
		\lim_{n \to +\infty} \frac{\sqrt{2}}{\beta n^{1/9}} \log  \bar{Z}_{n,h}^{\omega,\beta} \big( | M_n^-+u_*n^{1/3}| \leq \eps_n n^{1/3}, |\Delta_n| \leq \eps_n n^{1/3} \big) = \mathcal{W}_2 \, ,
	\end{equation}
	which, according to Corollary \ref{cor:u*}, proves the convergence of $Z_{n,h}^{\beta,\omega}$.
\end{proof}

Combining Proposition \ref{th:restrict-1/9} with the fact that the right-hand side quantity in Proposition \ref{prop-cv-1/9} increases with $K$ and thus converges almost surely as $K \to \infty$ yields Theorem~\ref{th-1/9} in the case of a Gaussian environment. We only need to see that the convergence is towards a non trivial quantity, which is the object of the following lemma.

\subsection{Path properties at second order}\label{sec:traj-1/9}

\begin{proof}[Proof of Theorem \ref{th-1/9}-\eqref{eq:th-1/9-proba}]
	The proof of \eqref{eq:th-1/9-proba} is a repeat of the proof of Lemma \ref{lem:position}, this time writing
	\[ \mathcal{U}_2^{\eps,\eps'} \defeq \left\{ (u,v) \in \RR^2 \, : \, \sup_{(s,t) \in B_\eps(u,v)} \Big\{ \mathcal{Y}_{s,t} - c_{\beta,h} (s+t)^2 \Big\} \geq \mathcal{W}_2 - \eps' > 0 \right\} \, , \]
	with $B_\eps(u,v)$ the Euclidean ball of radius $\eps$ centered at $(u,v)$. Define the event
	\[ \mathcal{A}_{2,n}^{\eps,\eps'} \defeq \mathset{\left( \frac{M_n^- + u_*n^{1/3}}{n^{2/9}},\frac{M_n^+ - (c_h-u_*)n^{1/3}}{n^{2/9}} \right) \not\in \mathcal{U}_2^{\eps,\eps'}} \, , \]
	then
	\[ \log \mathbf{P}_{n,h}^{\omega,\beta} \big( \mathcal{A}_{2,n}^{\eps,\eps'} \big) = \log \mathcal{Z}_{n,h}^{\omega,\beta}\big( \mathcal{A}_{2,n}^{\eps,\eps'} \big) - \log \mathcal{Z}_{n,h}^{\omega,\beta} \, . \]
	Afterwards, using the definition of $\mathcal{U}_2^{\eps,\eps'}$ we prove as above that 
	\[ \limsup_{n \to \infty} \frac{\sqrt{2}}{\beta n^{1/9}} \log \mathcal{Z}_{n,h}^{\omega,\beta}\big( \mathcal{A}_{2,n}^{\eps,\eps'} \big) = \sup_{(u,v) \not\in \mathcal{U}_2^{\eps,\eps'}} \Big\{ \mathcal{Y}_{s,t} - \frac{3\pi^2}{\beta c_h^4 \sqrt{2}} (s+t)^2 \Big\} < \mathcal{W}_2 - \eps' \, , \]
	and thus $\limsup\limits_{n \to \infty} n^{-1/9} \log \mathbf{P}_{n,h}^{\omega,\beta} \big( \mathcal{A}_{2,n}^{\eps,\eps'} \big) < 0$, proving \eqref{eq:th-1/9-proba} since $\bigcap\limits_{\eps' > 0} \mathcal{U}_2^{\eps,\eps'} \subset B_{2\eps}((\mathcal{U},\mathcal{V}))$.
\end{proof}

\section{Generalizing with the Skorokhod embedding}\label{sec:general-1/9}

\subsection{Proof of Theorem \ref{th-1/9}, case of a finite \texorpdfstring{$(3+\eta)$}-th moment}\label{sec:th-1/9-3+}

For now, Theorem \ref{th-1/9} has only been established for a Gaussian environment $\omega$, meaning that the variables $(\omega_z)$ are $i.i.d$ with a normal distribution. In the following, we will explain how we can generalize those results to any random $i.i.d.$ field with sufficient moment conditions.

\medskip
\par
We first expand on the coupling between the random field $\omega$ and the Brownian motions $X^{(i)}, i = 1,2$. Our starting point is the following statement from \cite[Chapter~7.2]{skorokhod1982studies}.

\begin{theorem}[Skorokhod]\label{th-skorokhod}
	Let $\xi_1, \dots, \xi_m$ be i.i.d.\ centered variables with finite second moment. For a Brownian motion $W$, there exists independent positive variables $\tau_1, \dots, \tau_m$ such that
	\[ \big(\xi_1, \dots, \xi_m\big) \overset{d}{=} \left( W(\tau_1), W(\tau_1 + \tau_2)-W(\tau_1), \dots, W(\sum_{i = 1}^m \tau_i)-W(\sum_{i = 1}^{m-1} \tau_i) \right) \, . \]
	Moreover, for all $k \leq m$, we have
	\[ \esp{\tau_k} = \esp{\xi_k^2} \quad \text{and} \quad \forall p > 1, \exists C_p > 0, \esp{(\tau_k)^p} \leq C_p \esp{(\xi_k)^{2p}} \, . \]
\end{theorem}

The following theorem gives us asymptotic estimates for the error of this coupling.

\begin{theorem}[{\cite[Theorem 2.2.4]{csorgo2014strong}}]\label{th-strassen}
	Let $(\theta_i)$ be i.i.d. centered variables, and assume that $\esp{|\theta_1|^p} < \infty$ for a real number $p \in \left(2, 4 \right)$. Then, if the underlying probability space is rich enough, there is a Brownian motion $W$ such that
	\[ \displaystyle \Big| \sum_{i = 1}^m \theta_i - W_m \Big| = \bar{o} \big( m^{1/p} (\log m)^{1/2} \big) \quad a.s. \text{ as } m \to \infty \, . \]
\end{theorem}

We can easily adapt this statement and choose the Wiener processes $X^{(1)}$ and $X^{(2)}$ to be independent Brownian motions such that, as $n\to\infty$,
\[ \Big| \sum_{z = 1}^{u n^{1/3}} \omega_{-z} - n^{1/6} X^{(1)}_{u} \Big| \vee \Big| \sum_{z = 0}^{v n^{1/3}} \omega_z - n^{1/6} X^{(2)}_{v} \Big| = \bar{o} \big( (u \vee v)^{1/p} (n^{1/3})^{1/p} (\log n)^{1/2} \big) \]
as long as $\mathbb{E}[|\omega_z|^p]<+\infty$ for some $p\in (2,4)$.
Since in the partition function we can restrict to trajectories with $x$ and $y$ are taken between $0$ and $(c_h + \eps_n)n^{1/3}$ (recall~\eqref{partition}), we can obtain a uniform bound over every $u,v$ we consider, meaning that
$\mathbb{P}$-a.s.\ there is some constant $C(\omega)$ such that for all $n\geq 1$,
\begin{equation}\label{eq:skorokhod}
	\Big| \sum_{z = 1}^{x} \omega_{-z} - n^{1/6} X^{(1)}_{u} \Big| \vee \Big| \sum_{z = 0}^{y} \omega_z - n^{1/6} X^{(2)}_{v} \Big| \leq C(\omega) \; n^{1/3p} (\log n)^{1/2} \, ,
\end{equation} 
uniformly for $|x|,|y|\leq 2c_h n^{1/3}$.
Let us also recall the notation $\bar{Z}_{n,h}^{\omega, \beta} = Z_{n,h}^{\omega, \beta} \, e^{\frac32 h T_n^* - \beta n^{1/6} X_{u_*}}$.

\begin{proof}[Proof of Theorem \ref{th-1/9} with $\esp{|\omega_0|^{3+\eta}} < \infty$]
	We now repeat the proof for a Gaussian field, but with the introduction of an error term given by Theorem \ref{th-strassen}. Recall \eqref{fact-1/9}: with a Gaussian environment, we had $\Sigma^-_{x} = X^{(1)}_{x n^{-1/3}}$ and $\Sigma^+_{y} = X^{(2)}_{y n^{-1/3}}$. Now, we must introduce an error term $E_n(x,y) \defeq \Sigma^-_{x} - X^{(1)}_{x n^{-1/3}} + \Sigma^+_{y} - X^{(2)}_{y n^{-1/3}}$: the equation \eqref{partition2} becomes
	\begin{equation}\label{eq:part-terme-erreur}
		\bar Z_{n,h}^{\omega,\beta} \sim \psi_h \sin\left(\frac{u_*\pi}{c_h}\right) \!\!\! \sum_{\substack{|x-u^*n^{1/3}|\leq \eps_n n^{1/3} \\ |y-(c_h-u^*| n^{1/3}|\leq \eps_n n^{1/3}}} \!\!\!
		\exp\left( \beta \bar \Omega_{n}^{x,y} + E_n(x,y) - \frac{3 \pi^2 (\Delta_n^{x,y})^2}{2 c_h^4 n^{1/3}} (1+\bar{o}(1)) \right),
	\end{equation}
	with $\bar{o}(1)$ deterministic and uniform in $x,y$, and
	\[ \bar\Omega_n^{x,y} := n^{1/6} \left( X^{(1)}_{xn^{-1/3}} + X^{(2)}_{y n^{-1/3}} - X_{u_*} \right) \, , \qquad E_n(x,y) \defeq \sum_{z = -x}^y \omega_z - \bar\Omega_{n}^{x,y}. \]
	
	Take $p = 3 + \eta$ with $\eta \in (0,1)$, and assume that $\esp{|\omega_0|^p} < +\infty$. Then, using \eqref{eq:skorokhod},  we have $|E_n(x,y)| \leq C(\omega) n^{1/(9+3\eta)} (\log n)^{1/2}$ for all summed $(x,y)$. Therefore, combining with \eqref{eq:part-terme-erreur}, we get that $\PP$-a.s.
	\[ \Bigg| \log \bar Z_{n,h}^{\omega,\beta} - \log \!\!\! \sum_{\substack{|x-u^*n^{1/3}|\leq \eps_n n^{1/3} \\ |y-(c_h-u^*) n^{1/3}|\leq \eps_n n^{1/3}}} \!\!\!
	\exp\left( \beta \bar \Omega_{n}^{x,y} - \frac{3 \pi^2 (\Delta_n^{x,y})^2}{2 c_h^4 n^{1/3}} \right) \Bigg| \leq C(\omega) n^{\frac{1}{9+3\eta}} (\log n)^{1/2} \,.
	\]
	Since $n^{\frac{1}{9+3\eta}} (\log n)^{1/2} = \bar{o}(n^{1/9})$, we can restrict our study to (exactly) the same sum that appeared in the Gaussian case, see~\eqref{partition2}.
	Then, \eqref{eq:th-1/9-proba} follows identically from the same proof as in Section \ref{sec:traj-1/9}.
\end{proof}

\subsection{Adaptation to the case of a finite \texorpdfstring{$(2+\eta)$}-th moment}\label{sec:th-1/9-2+}

We now explain how we can infer \eqref{eq:th-1/9-part-weak}, \textit{i.e.}\ a version of Theorem~\ref{th-1/9} where we only assume that $\esp{|\omega_0|^{2+\eta}} < \infty$ for some positive $\eta$, from adapting the proofs of Section \ref{sec:gaussian-1/9}. We are able to prove that the relevant trajectories converge to the suspected limit for $Z_{n,h}^{\omega,\beta}$, however some technicalities prevent us from getting the full theorem.
\par The key observation is the following: when subtracting $\beta \sum_{u_* n^{1/3}}^{(c_h-u_*)n^{1/3}} \omega_z$ instead of $\beta n^{1/6} X_{u_*}$ from $\log Z_{n,h}^{\omega,\beta} + \frac32 h c_h n^{1/3}$, we precisely cancel out the $(\omega_z)$ present in both $\Sigma^-_{|M_n^-|} + \Sigma^+_{M_n^+}$ and $\Sigma^-_{u_* n^{1/3}} + \Sigma^+_{(c_h-u_*) n^{1/3}}$.
This leaves us with a smaller sample of the variables $(\omega_z)$, with size $|M_n^-+u_*n^{1/3}| + |M_n^+-(c_h-u_*)n^{1/3}|$ which is 
at most $2\eps_n n^{1/3}$ (see~\eqref{partition2}), and of order $n^{2/9}$ when restricting to trajectories giving the main contribution (see~Proposition~\ref{th:restrict-1/9}).

Informally, let $\Omega_{n,h}^{u_*}(x,y)$ be the sum of $\omega_z$ that are between $u_* n^{1/3}$ and $x$, and between $(c_h-u_*) n^{1/3}$ and $y$, with proper signs.
With the same arguments as before, we are led to get a convergence for
\begin{equation}
	\tilde{Z}_{n,h}^{\omega,\beta} \defeq \psi_h \sin\left(\frac{u_*\pi}{c_h}\right) \!\!\!\!\!\! \sum_{\substack{|x-u^*n^{1/3}|\leq \eps_n n^{1/3} \\ |y-(c_h -u^*)n^{1/3}|\leq \eps_n n^{1/3}}} \!\!\!\!\!\!
	\exp\left( \beta \Omega_{n,h}^{u_*}(x,y) - \frac{3 \pi^2 (\Delta_n^{x,y})^2}{2 c_h^4 n^{1/3}} (1+\bar{o}(1)) \right),
\end{equation}

We now want to rewrite $\Omega_{n,h}^{u_*}(x,y)$ as $\Omega_n^{x,y}$ in \eqref{eq:part-rewrit} with an additional error term that is a $\bar{o}(n^{1/9})$.
Since we are only interested in the $(\omega_z)$ that are present in $\Omega_{n,h}^{u_*}(x,y)$, we see that we only need to have a good coupling between the environment and $(X^{(1)},X^{(2)})$ near $(u_*,c_h-u_*)$ instead of a global coupling like the one we used in Section \ref{sec:th-1/9-3+}.

An application of Theorem \ref{th-strassen} allows to consider that the field $\omega$ satisfies:
\[ \Big| \Omega_{n,h}^{u_*}(x,y) - n^{1/6} \Big( X^{(1)}_{x n^{-1/3}} + X^{(2)}_{y n^{-1/3}} - X_{u_*} \Big) \Big| \leq C(\omega)\ (g_n(x,y))^{1/2+\eta} (\log n)^{1/2} \, , \]
where $g_n(x,y) = |u_*n^{1/3} - x| \vee (|(c_h-u_*)n^{1/3} - y|\wedge |c_h n^{1/3} - (x+y)|)$ and $C(\omega)$ is a positive constant at fixed $\omega$.

With this coupling, it is easy to adapt the proof of Proposition \ref{prop-cv-1/9} to the case $\esp{|\omega_0|^{2+\eta}} < +\infty$. This yields the convergence
\[ \lim_{K \to +\infty} \lim_{n\to \infty} \frac{\sqrt{2}}{\beta n^{1/9}} \log \bar{\mathcal{Z}}_{n,\omega}^{\leq K} = \sup_{u,v \in \RR} \left\{ \mathcal{Y}_{u,v} - \frac{3\pi^2}{\beta c_h^4 \sqrt{2}} \big( u + v \big)^2 \right\} \, = \mathcal{W}_2 , \]
where $\bar{\mathcal{Z}}_{n,\omega}^{\leq K}$ is defined the same way as in Section \ref{sec:traj-1/9}, only by substraction the sum of $\omega_z$'s instead of $X$.

What remains to show is that $n^{-1/9} \log \bar{\mathcal{Z}}_{n,\omega}^{> K}$ has a non-positive limsup as $K,n \to +\infty$ in the same spirit as Proposition \ref{th:restrict-1/9}.
However, we are not able to get a sufficient decay for the probabilities appearing in the proofs of Lemmas \ref{lem:excl-traj-L} \& \ref{lem:excl-traj-K}. The union bound thus fails to conclude the proof, although we have no doubt on the result being true.

\section{Simplified model : range with a fixed bottom}\label{sec:simplified}

In this section we shall focus on a simpler model in which one of the range's extremal points is fixed at $0$. The main motivation is that this model is sufficiently close to our original model to give us some insight on finer properties of the original polymer, while being easier to study because of the range being only a single variable: the highest point of the polymer.

We give in this section a conjecture on the simplified model which is supported by previously known results about Brownian motions with drift. This conjecture states that the fourth order expansion of the log-partition function is given by a quantity of order one, with a limited dependence on $n$. It is natural to expect the same behavior for our original model, hence Conjecture \ref{conj}.

Let us focus on this simplified polymer, which is modeled by a non-negative random walk. The polymer measure is given by
\[ \tilde{\mathbf{P}}_{n,h}^{\omega,\beta}(S) \defeq \frac{1}{\tilde Z_{n,h}^{\omega,\beta}} \exp \Big( -h M_n^+ + \beta \sum_{i = 0}^{M_n^+} \omega_i \Big) \indic{\forall k \leq n, S_k \geq 0} \mathbf{P}(S). \]
For now, we will keep studying the case where the field $\omega$ is composed of i.i.d. Gaussian variables. We once again take a Brownian motion $X$ such that $\frac{1}{n^{1/6}} \sum_{z=0}^T \omega_z = X_{Tn^{-1/3}}$. The partition function is given by
\[ \tilde{Z}_{n,h}^{\omega,\beta} = \sum_{T = 1}^{+\infty} e^{-hT + \sum_{i\leq T} \omega_i} \mathbf{P}(\mathcal{R}_n = \llbracket 0, T \rrbracket) = \sum_{T = 1}^{+\infty} \phi_n(T) e^{-hT + \sum_{i\leq T} \omega_i -g(T)n} \]
with $g(T) = \pi^2/2T^2$ (see \cite{bouchot1}, this is analogous to what is done in Section \ref{sec:part-homogene-desordre}).

It is not difficult to see that our results up to Section 3 still hold, first we have
\[ \frac{1}{n^{1/3}} \log \tilde Z_{n,h}^{\omega,\beta} \xrightarrow[n \to \infty]{\PP-a.s.} -\frac{3}{2} h c_h \, . \]
This tells us that the range has size $T_n \sim c_h n^{1/3}$ at first order. Since $M_n^- = 0$, it is natural to expect (and not hard to prove)
\[ \frac{1}{\beta n^{1/6}} \left( \log \tilde Z_{n,h}^{\omega,\beta} + \frac{3}{2} h c_h n^{1/3} \right) \xrightarrow[n \to \infty]{\PP-a.s.} X_{c_h} \, , \]
which is an analogue of Theorem \ref{th-1/6} with the knowledge that $u_* = 0$.
Factorizing by $e^{\beta n^{1/6} X_{c_h}}$ yields the following exponential term
\[ \beta \sum_{z = 0}^T \omega_z - \beta n^{1/6} X_{c_h} - \frac{3 \pi^2 (T - T_n^*)^2}{2 c_h^4 n^{1/3}} = -\beta n^{1/6} \big( X_{c_h} - X_{T n^{-1/3}} \big) - \frac{3 \pi^2 (T - T_n^*)^2}{2 c_h^4 n^{1/3}}.\]

\begin{proposal}
	For any $h,\beta > 0$ there is a standard Brownian motion $W$ such that $\PP$-a.s.,
	\begin{equation}\label{prop:simp-pre-conj}
		\lim_{n\to \infty} \frac{1}{\beta n^{1/9}} \left( \log \tilde Z_{n,h}^{\omega,\beta} + \frac{3}{2} h c_h n^{1/3} - \beta n^{1/6} X_{c_h} \right) = \sup_{s \in \RR} \Big\{ W_s - \frac{3\pi^2}{2 \beta c_h^4} s^2 \Big\}\,.
	\end{equation}
\end{proposal}

\noindent
\textit{Proof scheme.\ }
Since $|X_{c_h} - X_{T n^{-1/3}}| \leq C n^{-1/6} \sqrt{|T_n^* - T|}$ with probability at least $1-e^{-C^2}$, we can repeat the proof of Proposition \ref{th:restrict-1/9} and restrict the trajectories. This leads to studying $\tilde{Z}_{n,h}^{\omega,\beta}(|T_n^* - T| \leq K n^{2/9})$ which contains all the main contributions for $K$ large. Split over $k\delta n^{2/9} \leq T_n^* - T \leq (k+1)\delta n^{2/9}$ and the main contribution will be given by the supremum over $k$ of
\[ -\beta n^{1/6} \big( X_{c_h} - X_{c_h + \frac{k\delta n^{2/9}}{n^{1/3}}} \big) - \frac{3 \pi^2 (k\delta n^{2/9})^2}{2 c_h^4 n^{1/3}} = -\beta n^{1/6} \big( X_{c_h} - X_{c_h + \frac{s}{n^{1/9}}} \big) - \frac{3 \pi^2 s^2}{2 c_h^4} n^{1/9}\, , \]
where we wrote $s = k\delta$. We can conclude similarly to the proof of Theorem \ref{th-1/9} by changing the limit process $\mathcal{Y}_{u,v}$ to $B$ which is the limit of the processes $B^{(n)} = n^{1/18} \big(X_{c_h} - X_{c_h + \frac{u}{n^{1/9}}}\big)_u$ and is a standard Brownian motion. Once again, we can couple the Brownian motion $X = X^{(n)}$ so that the processes $B^{(n)}$ are equal to $B$ when $n$ is large, in the same fashion as Proposition \ref{prop-couplage}. We can prove that the right-hand side of \eqref{prop:simp-pre-conj} is $\PP$-a.s. positive and finite, attained at a unique point $s_*$.\qed

\medskip
\par
To sum up the results of this simplified model, we write the following statement

\begin{theorem}
	Recall the notation of \eqref{def:energie}, this time with $\tilde Z_{n,h}^{\omega,\beta}$. Then $\PP$-almost surely,
	\[ \tilde f_\omega^{(1,\frac13)}(h,\beta) = -\frac{3}{2} (\pi h)^{2/3} \, , \quad \tilde f_\omega^{(2,\frac16)}(h,\beta) = \beta X_{c_h} \, , \quad \frac{1}{\beta} \tilde f_\omega^{(3,\frac19)}(h,\beta) = \sup_{s \in \RR} \Big\{ B_s - \frac{3\pi^2}{2 \beta c_h^4} s^2 \Big\} \,.\]
\end{theorem}

Recall the following notation of \eqref{eq:notation}:
\[ T_n \defeq M_n^+ - M_n^- = |\mathcal{R}_n| - 1 \,, \qquad T_n^* \defeq \left(\frac{n \pi^2}{h} \right)^{1/3} = c_h n^{1/3} \,, \qquad  \Delta_n \defeq T_n - T_n^*. \]

\begin{corollary}
	There is a vanishing sequence $(\eps_n)$ such that
	\[ \limsup_{n \to \infty} \tilde{\mathbf{P}}_{n,h}^{\omega,\beta} \left( |\Delta_n - s_* n^{2/9}| \geq \eps_n n^{2/9} \right) = 0 \, . \]
\end{corollary}

\par Our goal is now to find out whether factorizing the partition function by this quantity leads to a bounded logarithm or not; in other words, we are looking for the $4$th order free energy, in the spirit of Section~\ref{def:energie}. 
We develop here some heuristic to justify that the $4$th order free energy is at scale $\alpha_4=0$.

Going forward we work conditionally to $s_*$. We define
\begin{equation}\label{eq:part-factorized-simp}
	\tilde{\mathcal{Z}}_{n,h}^{\omega,\beta} \defeq \tilde{Z}_{n,h}^{\omega,\beta} \exp \left( \frac{3}{2} h c_h n^{1/3} - \beta n^{1/6} X_{c_h} - \beta n^{1/9} \sup_u \Big\{ W_u - \frac{3\pi^2}{2 \beta c_h^4} u^2 \Big\} \right) \phi(T_n^*)^{-1}  \, .
\end{equation}
We first rewrite the factorized partition function $\tilde{\mathcal{Z}}_{n,h}^{\omega,\beta}$. If we write $T_n = c_h n^{1/3} + \Delta_n$ and we recall that thanks to the coupling, for $u$ in a neighborhood of $0$, we have $W_u = n^{1/18} \left( X_{c_h + \frac{u}{n^{1/9}}} - X_{c_h} \right)$ for sufficiently large $n$, we can rewrite
\[ \beta n^{1/6} \left( X_{T_n n^{-1/3}} - X_{c_h} \right) = \beta n^{1/6} \left( X_{c_h + \Delta_n n^{-1/3}} - X_{c_h} \right) = \beta n^{1/9} B_{\Delta_n n^{-2/9}} \, . \]
Then, we have
\begin{equation}\label{eq:diff-MBdrift}
	\tilde{\mathcal{Z}}_{n,h}^{\omega,\beta} \sim \sum_{|k - s_* n^{2/9}| \leq \eps_n n^{2/9}} \exp \left( \beta n^{1/9} \left[ \left( W_{k n^{-2/9}} - c_{h,\beta} \frac{k^2}{n^{4/9}} \right) - \sup_{s \in \RR} \mathset{W_s - c_{h,\beta} s^2} \right] \right) \, .
\end{equation}
We define the process $Y_s \defeq B_s - c_{h,\beta} s^2$ which is a Brownian motion with quadratic drift, and $s_*$ the point at which it attains its maximum on $\RR$. \eqref{eq:diff-MBdrift} can thus be rewritten as
\begin{equation}\label{eq:diff-MBdrift-Y}
	\tilde{\mathcal{Z}}_{n,h}^{\omega,\beta} \sim \sum_{|k - s_* n^{2/9}| \leq \eps_n n^{2/9}} \exp \left( \beta n^{1/9} \left( Y_{k n^{-2/9}}  - Y_{s_*} \right) \right) \, .
\end{equation}
The exponential term is non-positive, which means that the typical trajectories for the polymer are those that minimize the difference in \eqref{eq:diff-MBdrift}. 

\begin{remark}
	Previous works studied with some extent the laws of $s_*$ and $Y_{s_*}$ (see \cite{janson2013moments}). In particular $s_*$ follows the so-called Chernov distribution, which is symmetric. Writing Ai for the Airy function, \cite[Theorem 1.1]{janson2013moments} states that
	\[ \quad \esp{s_*^2} = \frac{2^{-2/3} c_{h,\beta}^{-4/3}}{6 i \pi} \int_\RR  \frac{y \, dy}{\text{Ai}(iy)} < \infty \quad \text{and} \quad \forall p \in \NN, \; \esp{s_*^p} < +\infty \, . \]
\end{remark}

In all the following, we use the fact that the distribution of $s_*$ is symmetric to reduce to the case $s_* > 0$. We will also work conditionally on the value of $s_*$, meaning on the location of the maximum of $Y$. We write $\alpha = 2 c_{h,\beta} s_*$, then observe that for any $s > 0$,
\[ Y_{s_*} - Y_{s_* \pm s} = B_{s_*} - B_{s_* \pm s} \pm \alpha s - c_{h,\beta} s^2 \leq B_s \pm \alpha s \eqdef R^\pm_s \, . \]
Thus we have $e^{\beta n^{1/9} (Y_{k n^{-2/9}} - Y_{s_*})} \leq e^{\beta n^{1/9} (R_{k n^{-2/9}} - R_{s_*})}$ which means that we can get an upper bound on the contribution of a given trajectory just by studying the processes $R^\pm$ conditioned to be positive, provided the existence of a coupling between these processes and $Y$. Moreover, since we are interested in the setting $s \to 0$, we should have a lower bound that reads $Y_{s_*} - Y_{s_* \pm s} \geq (1+o(1))R^\pm_s$ as $s \to 0$. This motivates our first conjecture, which is an analog of Proposition \ref{prop-couplage}. We will write $R = R^- \mathbbm{1}_{\RR^-} + R^+ \mathbbm{1}_{\RR^+}$.

\begin{conjecture}\label{prop:couplage-simple}
	There exists a coupling of $(R,Y)$ and a two-sided $\mathrm{BES}_3$ $\tilde{\mathbf{B}}$ such that almost surely, there exists $\delta_1 > 0$ and $n_1 \in \NN$ for which $\forall n \geq n_1$, for all $|u| < \delta_1$
	\begin{equation}\label{eq:couplage-simple}
		n^{1/9} R_{u n^{-2/9}} = n^{1/9} (Y_{s_*} - Y_{s_* + \frac{u}{n^{2/9}}}) = \tilde{\mathbf{B}}_u.
	\end{equation}
\end{conjecture}

The fact that the three-dimensional Bessel process appears is mainly due to the following result from Martinez and San Martin \cite{martinez1994quasi}.

\begin{theorem}
	Define $X^\alpha_t \defeq x + W_t - \alpha t$, with $\alpha, x > 0$. Then the process $X^\alpha$ conditioned to stay positive on $[0,T]$ converges in distribution to the Bessel process as $T \to \infty$.
\end{theorem}

Our second conjecture is a description of the simplified model and the idea should follow along the steps of Section \ref{sec:gaussian-1/9}, excluding trajectories and using Conjecture \ref{prop:couplage-simple} to get an almost-sure convergence of $\tilde{Z}_{n,h}^{\omega,\beta}$.

\begin{conjecture}\label{th-ordre 1}
	There exist $\tilde{\mathbf{B}}$ (given by \eqref{eq:couplage-simple}) a two-sided three-dimensional Bessel process such that for $n$ large enough, writing $s_*^n \defeq s_* n^{2/9} - \lfloor s_* n^{2/9} \rfloor$ we have
	\[ \mathbf{P}_{n,h}^{\omega, \beta} \left( M_n^+ = c_h n^{1/3} + \lfloor s_* n^{2/9} \rfloor + k \right) \sim \frac{1}{\theta_\omega (n)} e^{-\beta \mathcal{W}_{s_*^n + k}} \, , \quad \text{with} \quad  \theta_\omega (n) \defeq \sum_{k \in \ZZ} e^{-\beta \mathcal{W}_{s_*^n + k}(\omega)} \, . \]
\end{conjecture}

\begin{heuristic}
	To minimize $n^{1/9} (Y_{s_*} - Y_{s_* + s}) = n^{1/9} R_s$, since $R_s \asymp \sqrt{s} - \alpha s \asymp \sqrt{s}$ with high probability when $s \to 0$ (we are close to $s_*$) we roughly need to have $s = \grdO (n^{-2/9})$. In the definition of $\tilde{\mathcal{Z}}_{n,h}^{\omega,\beta}$, we take $s + s_* = \Delta n^{-2/9}$, thus we should be able to prove that $\tilde{\mathcal{Z}}_{n,h}^{\omega,\beta}(|\Delta_n - s_*n^{2/9}| > K)/\tilde{\mathcal{Z}}_{n,h}^{\omega,\beta} \to 0$ when $K \to \infty$ and $n$ is large enough. On the other hand, for $\tilde{\mathcal{Z}}_{n,h}^{\omega,\beta}(|\Delta_n - s_*n^{2/9}| \leq K)$, when $n$ is large enough, we have $Y_{kn^{-2/9}} - Y_{s_*} = \tilde{\mathbf{B}}_k$ for any $|k| \leq K$. Thus, we should be able to prove that when $n \to +\infty$, we have $|\tilde{\mathcal{Z}}_{n,h}^{\omega,\beta}(|\Delta_n - s_*n^{2/9}| \leq K) - \sum_{|k| \leq K} e^{-\beta \tilde{\mathbf{B}}_k}| \to 0$ in similar fashion to the proof of Proposition~\ref{prop-cv-1/9}.
\end{heuristic}

\begin{remark}
	It should be possible to obtain an analog of Conjecture \ref{th-ordre 1} in the general model for a Gaussian environment $\omega$, which supports Conjecture \ref{conj}. This would require a coupling of $\mathcal{Y}_{\mathcal{U},\mathcal{V}} - \mathcal{Y}_{u,v}$ (recall the definitions in Theorem \ref{th-1/9}) with some suitable process on a small neighborhood of $(\mathcal{U},\mathcal{V})$.
\end{remark}

\begin{appendix}
	
	\section{Disorder in a domain of attraction of a Lévy process}\label{appendix-Levy}
	
	In this section we will extend the Theorem~\ref{th-1/6} to the case where $(\omega_z)_{z\in \mathbb Z}$ is in the domain of attraction of an $\alpha$-stable law, with $\alpha\in (1,2)$; we refer to~\cite{berger2020one} where the case $\alpha<1$ is shown to have a different behavior.
	More precisely, we assume that the field $\omega$ is such that $\esp{\omega_0} = 0$ and that there exists $\alpha \in (1,2)$ such that
	\begin{equation}\label{eq:alpha-stable}
		\proba{\omega_0 > t} \sim pt^{-\alpha} \quad , \quad \proba{\omega_0 < -t} \sim qt^{-\alpha} \quad \text{as $t \to \infty$ with $p+q=1$.}
	\end{equation}
	This ensures that $\frac{1}{k^{1/\alpha}} \sum_{z = 0}^k \omega_z$ converges in law to an $\alpha$-stable Lévy process, $\alpha\in (0,2)$. Note that we treat the case of a pure power tail in~\eqref{eq:alpha-stable}, \textit{i.e.}\ the normal domain of attraction to an $\alpha$-stable law, only for simplicity, to avoid dealing with slowly varying corrections in the tail behavior.
	
	\par As in the case where $\esp{\omega_0^2} = 1$, one can define a coupling $\hat\omega = \hat\omega^{(n)}$ such that
	\[ \left( \frac{1}{n^{1/3\alpha}} \Sigma^-_{un^{1/3}}(\hat{\omega}) \right)_{u \geq 0} \xrightarrow[n \to \infty]{a.s.} \left( X^{(1)}_u(\hat{\omega})  \right)_{u \geq 0} \; , \hspace{0.5cm} \left( \frac{1}{n^{1/3\alpha}} \Sigma^+_{vn^{1/3}}(\hat{\omega}) \right)_{v \geq 0} \xrightarrow[n \to \infty]{a.s.} \left( X^{(2)}_v(\hat{\omega})  \right)_{v \geq 0}, \]
	where $X^{(1)},X^{(2)}$ are two independent $\alpha$-stable Lévy processes, see \cite[\S1.2]{berger2020one}.

	\medskip
	If the range is of size of order $n^{\xi}$,	then we have that $\sum_{z \in \mathcal{R}_n} \omega_z$ is of order $n^{\xi/\alpha}$, which is negligible compared to $n^{\xi}$ since $\alpha > 1$. 
	Hence the disorder should be negligible at first order, and this is what is proven in  \cite[Thm.~1.2]{berger2020one}: we have
	\[ \lim_{n\to \infty} \frac{1}{n^{1/3}} \log Z_{n,h}^{\omega,\beta} = -\frac{3}{2} (\pi h)^{2/3}, \quad \forall \eps > 0, \mathbf{P}_{n,h}^{\omega,\beta}\left( \Big| n^{-1/3}|\mathcal{R}_n| - c_h \Big| > \eps \right) \xrightarrow[n \to \infty]{} 0 \, . \]
	
	Our result here is to obtain the second order asymptotic for the convergence of $\log Z_{n,h}^{\omega,\beta}$; we deduce a result on the position of the range under $\mathbf{P}_{n,h}^{\omega,\beta}$.
	\begin{theorem}\label{th-1/6-Levy}
		Suppose that $(\omega_z)_{z\in \mathbb Z}$ verifies~\eqref{eq:alpha-stable}.
		Then, for any $h,\beta > 0$, we have the following $\PP$-a.s.\ convergence
		\[ \lim_{n\to \infty} \frac{1}{\beta n^{1/3\alpha}} \left( \log Z_{n,h}^{\omega,\beta} + \frac{3}{2}hc_hn^{1/3} \right) = \sup_{0 \leq u \leq c_h} \left\{ X^{(1)}_u + X^{(2)}_{c_h-u} \right\} \, , \]
		where $X^{(1)}$ and $X^{(2)}$ are two independent $\alpha$-stable Lévy processes.
		
		\noindent
		Furthermore, $u_* \defeq \argmax_{u \in [0,c_h]} \left\{ X^{(1)}_u + X^{(2)}_{c_h-u} \right\}$ exists $\PP$-almost surely and
		\[ \forall \eps > 0, \mathbf{P}_{n,h}^{\omega, \beta} \left( \Big|\frac{1}{n^{1/3}} (M_n^-,M_n^+) - (-u_*,c_h-u_*) \Big| > \eps \right) \xrightarrow[n \to \infty]{} 0 \qquad \text{$\PP$-a.s.}  \]
	\end{theorem}

	\begin{proof}
		The proof is essentially the same as the one of Theorem \ref{th-1/6}.
		As in~\eqref{eq:decoupZnh}, we can write
		\[ \log Z_{n,h}^{\omega,\beta} + \frac{3}{2} h c_hn^{1/3} = \log \big(1 + \bar{o}(1)\big)\psi_h + \log \sum_{k_1 = 0}^{c_h/\delta} \sum_{k_2 = \frac{c_h}{\delta} - k_1 - 1}^{\frac{c_h}{\delta} - k_1} Z_{n,h}^{\omega,\beta} (k_1,k_2,\delta) \, , \]
		with $Z_{n,h}^{\omega,\beta} (k_1,k_2,\delta)$ defined as in \eqref{Znw-retreint-1/6}. Once again we have
		\begin{equation}
			\Big| \sum_{z = -x}^y \omega_z - \left( \Sigma^+_{k_2\delta n^{1/3}} + \Sigma^-_{k_1\delta n^{1/3}} \right) \Big| \leq R_n^\delta(k_1\delta, k_2\delta) \, ,
		\end{equation}
		where the error remainer $R_n^\delta$ is defined for $u,v \geq 0$ by
		\[ R_n^\delta(u,v) \defeq \max_{un^{1/3} + 1 \leq j \leq (u+\delta)n^{1/3}-1} \big| \Sigma^-_j - \Sigma^-_{un^{1/3}} \big| + \max_{vn^{1/3} + 1 \leq j \leq (v+\delta)n^{1/3}-1} \big| \Sigma^+_j - \Sigma^+_{vn^{1/3}} \big| \, .
		\]
		Using the coupling $\hat{\omega}$ and Lemma A.5 of \cite{berger2020one}, we have $\PP-a.s.$ $\forall \eps > 0$, $\exists n_0 = n_0(\eps,\delta,\omega)$ such that $\forall n \geq n_0$,
		\[ \frac{1}{n^{1/3\alpha}} R_n^\delta(u,v) \leq \eps + \sup_{u \leq u' \leq u + \eps + \delta} \big| X^{(1)}_{u'}-X^{(1)}_u \big| + \sup_{v \leq v' \leq v + \eps + \delta} \big| X^{(2)}_{v'}-X^{(2)}_v \big| \]
		\[ \left( \left| \frac{1}{n^{1/3\alpha}}\Sigma^+_{vn^{1/3}} - X^{(2)}_v \right| \vee \left| \frac{1}{n^{1/3\alpha}}\Sigma^-_{un^{1/3}} - X^{(1)}_u \right| \right) \leq \eps\, , 
		\]
		uniformly in $u,v \in U_\delta$ as $U_\delta$ is a finite set (recall the definition \eqref{eq:def-U-delta} of $U_\delta$).
		Letting $N \to \infty$ then $\eps \to 0$ we obtain that $\PP$-almost surely,
		\[ 
		\begin{split}
			\limsup_{n\to \infty} \frac{1}{\beta n^{1/6}} \left( \log Z_{n,h}^{\omega,\beta} + \frac{3}{2}hc_hn^{1/3} \right) \leq  \sup_{\substack{u,v \in U_\delta\\ u+v \in \{c_h,c_h-\delta\}}} \mathcal{W}^+(u,v,\delta) \, , \\
			\liminf_{n\to \infty} \frac{1}{\beta n^{1/6}} \left( \log Z_{n,h}^{\omega,\beta} + \frac{3}{2}hc_hn^{1/3} \right) 
			\geq \sup_{\substack{u,v \in U_\delta\\ u+v \in \{c_h,c_h-\delta\}}} \mathcal{W}^-(u,v,\delta)
		\end{split}
		\]
		in which we wrote
		\[ \mathcal{W}^{\pm}(u,v,\delta) =  X^{(1)}_u + X^{(2)}_v \pm \sup_{u \leq u' \leq u + \delta} \big| X^{(1)}_{u'}-X^{(1)}_u \big| \pm \sup_{v \leq v' \leq v + \delta} \big| X^{(2)}_{v'}-X^{(2)}_v \big| \, . \]
		Using the càdlàg structure of Lévy processes $X^{(1)}$ and $X^{(2)}$ we push $\delta$ to $0$ and get the desired convergence.
		
		\par Afterwards, we can use \cite[Theorem 2.1]{Levy-processes} and \cite[Section 3]{growth-RW-Levy} to prove that the variational problem is positive and finite (in the sense that $\sup_{0 \leq u \leq c_h} \big\{ X^{(1)}_u + X^{(2)}_{c_h-u} \big\}$ is almost surely positive and finite), which relies on the same reasoning as Lemma \ref{pb-var-1/6}. Then, \cite[Proposition 3.1]{berger2020one} proves the existence and unicity of the maximizer $u_*$. The proof of the second part of Theorem~\ref{th-1/6-Levy} is exactly the proof of Lemma~\ref{lem:position}.
	\end{proof}

	\section{Technical results for the Brownian meander}\label{appendix-meandre}
	
	Let $W$ be a standard Brownian motion on $[0,1]$ and denote $\tau \defeq \sup \mathset{t \in [0,1] \; : \, W_t = 0}$.
	The Brownian meander on $[0,1]$ is defined as the rescaled trajectory of $W$ between $\tau$ and~$1$. More precisely it is the process $M$ defined on $[0,1]$ by
	\[ M_t \defeq \frac{1}{\sqrt{1-\tau}} |W_{\tau + t(1-\tau)}| \, . \]
	Note that we could define the meander to be on any interval $[0,T]$ by changing how we rescale the trajectory, leading to define a Brownian meander of duration $T$ as the rescaled process $\sqrt{\frac{T}{1-\tau}} |W_{\tau + \frac{t}{T}(1-\tau)}|$ on $[0,T]$.
	
	The Brownian meander also appears when studying a Brownian motion seen from its maximum over an interval. More precisely, we have the following result.
	
	\begin{proposal}[{\cite[Theorem 5]{imhof1984density}}]\label{prop:max-meandre}
		Let $W$ be a Brownian motion on $[0,1]$ and let $\sigma$ be the time at which $W$ reaches its maximum on $[0,1]$. Conditional on the event $\mathset{\sigma = u}$, the processes $(W_{u+s} - W_u, 0 \leq s \leq 1-u)$ and $(W_{u-s} - W_u, 0 \leq s \leq u)$ are independent Brownian meanders of respective duration $1-u$ and $u$.
	\end{proposal}
	
	Recall the notation $\varphi_t(x) \defeq \frac{1}{\sqrt{2 \pi t}} e^{-\frac{x^2}{2t}}$ and $\Phi_t(y) \defeq \int_0^y \varphi_t(x) dx$.
	The Brownian meander on $[0,1]$ is a continuous, non-homogeneous Markov process starting at $0$, with transition kernel given by
	\begin{equation}\label{meandre-transition}
		\proba{M_t \in dy \, | \, M_s = x} = p^+(s,x,t,y) dy = \left[\varphi_{t-s}(x-y) - \varphi_{t-s}(x+y)\right] \frac{\Phi_{1-t}(y)}{\Phi_{1-s}(x)} dy
	\end{equation} 
	and
	\begin{equation}\label{meandre-transition-0}
		\proba{M_t \in dy} =  p^+(0,0,t,y) dy = 2y t^{-3/2} e^{-\frac{y^2}{2t}} \Phi_{1-t}(y) dy \, .
	\end{equation}
	For the proofs of these facts, we refer to \cite{meandre} and its references.				
	
	\noindent
	Using $\proba{M_t \in dy} \leq 2y t^{-3/2} e^{-\frac{y^2}{2t}} dy$ and $\Phi_{1-t}(y) \leq y/\sqrt{2\pi(1-t)}$, we have the following estimates: for any $a<r/2$,
	\begin{equation}
		\label{meander-expmoment}
		\begin{split}
			\esp{e^{a M_r^2}} & \leq \frac{2}{r^{3/2}\sqrt{2\pi }} \int_0^\infty y e^{-\frac{1-2ar}{2r} y^2} \, \dd y =  \left(1-2ra\right)^{-3/2} \, , \\
			\esp{(M_r)^{-1} e^{a \mathcal{M}_r^2}} &\leq \frac{1}{r^{3/2}\pi\sqrt{(1-r)}} \int_0^\infty y e^{-\frac{1-2ar}{2r} y^2} \, \dd y = \frac{\sqrt{2\pi}}{\sqrt{r(1-r)}} \left(1-2ra\right)^{-1} \,,
		\end{split}
	\end{equation}

	\smallskip
	The asymmetry of the meander can be used to prove the following ``reflection principle''.
	
	\begin{lemma}[Reflection principle for the meander]\label{reflexion}
		Let $M$ be a Brownian meander, then for all $b > 0$ and all $0 \leq s < t \leq 1$,
		\[  \PP \Big( \sup_{0 \leq r \leq t} M_r \geq b \Big) \leq 2 \proba{M_t \geq b} \, , \quad \PP \Big(\sup_{s \leq r \leq t} |M_r - M_s| \geq b \Big) \leq 4 \proba{M_t - M_s \geq b} \, . \]
	\end{lemma}

	\begin{proof}
		If we denote by $T_b$ the hitting time of $b$, we have
		\[ \proba{\sup_{0 \leq s \leq t} M_s \geq b} = \int_0^t \proba{T_b \in ds} = \int_0^t \proba{T_b \in ds, M_t < b}  + \int_0^t \proba{T_b \in ds, M_t \geq b} \, .  \]
		Now, write $L_b$ the lime of last visit to $b$ before time $t$, on $[L_b,t]$ the process $M_r - b$ is a Brownian bridge conditioned to be above $-b$. We only need to see that any trajectory of $M$ from $b$ to $(0,b]$ which stays above $0$ can thus be transformed into a trajectory from $b$ to $[b,2b)$ that stays above $0$ by reflecting the trajectory between the last visit $L_b$ to $b$ and~$t$ (see Figure \ref{fig:reflexion-meandre}).
		
		\begin{figure}
			\centering
			\includegraphics[width=9cm]{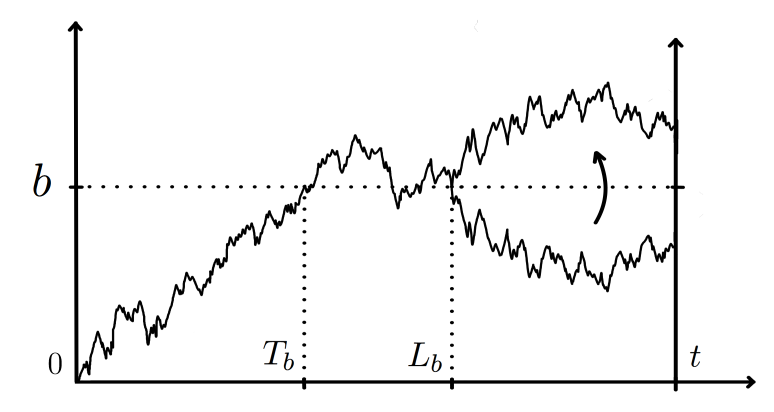}
			\caption{Reflection of the trajectory $b \to (0,b]$ with respect to the horizontal line at $b$}
			\label{fig:reflexion-meandre}
		\end{figure}
		
		Since these two Brownian bridges have the same probability and $[b,2b) \subset [b, +\infty)$ it shows that this operation is injective and thus $\proba{T_b \in ds, M_t < b} \leq \proba{T_b \in ds, M_t \geq b}$ for all $s \leq t$ (note that this is a consequence of the Brownian reflection principle). Therefore, we proved
		\[ \proba{\sup_{0 \leq s \leq t} M_s \geq b} \leq 2 \int_0^t \proba{T_b \in ds, M_t \geq b} = 2 \proba{T_b \leq t, M_t \geq b} = 2 \proba{M_t \geq b} \, . \]
		
		If we study the supremum of an increment $M_r - M_s$, $s \leq r \leq t$ we only need to repeat the proof for a starting point $M_s = x$ and integrate over all the positions $x$. Since the meander is a Markov process, we get 
		$\PP \big(\sup_{s \leq r \leq t} M_r - M_s \geq b \big) \leq 2 \proba{M_t - M_s \geq b}.$
		Afterwards, we only need to see that again using the asymmetry of $M$, we have that 
		\[
		\PP \Big(\sup_{s \leq r \leq t} |M_r - M_s| \geq b \Big) \leq 2 \PP \Big(\sup_{s \leq r \leq t} M_r - M_s \geq b \Big)\,, 
		\]
		hence the result.
	\end{proof}
	
	\begin{corollary}
		\label{cor-meander}
		For any $\lambda > 1, a > 0$ and $0 \leq s < t < \frac12$, we have
		\[ \PP \Big( \inf_{s \leq r \leq t} M_r \leq a \Big) \leq \proba{M_s \leq \lambda a} + \proba{M_t \leq \lambda a} + \frac{4 a \sqrt{2t}}{t-s} \frac{e^{-\frac{2}{t-s} a^2 (\lambda-1)^2}}{1 - e^{-\frac{2}{t-s} a^2 \lambda^2}} \, , \]
		as well as
		$ \proba{M_t \leq a}  \leq \frac{4a}{\sqrt{\pi t}}\left( 1 \wedge \frac{a^2}{2t} \right).$
	\end{corollary}
	
	\begin{proof}
		We decompose the probability on whether $M_s, M_t \leq \lambda a$, meaning we only have to consider $\PP \big( \inf_{s \leq r \leq t} M_r \leq a, M_s > \lambda a, M_t > \lambda a \big)$.
		For this, we first use Brownian bridge estimates: see that for any $z,w, T > 0$, we have
		\[ \begin{split} \probaM{z}{W_T \in dw, \inf_{t \in [0,T]} W_t > 0} &= \frac{1}{\sqrt{2\pi T}} \left( e^{-\frac{1}{2T}(z-w)^2} - e^{-\frac{1}{2T}(z+w)^2} \right) dw \\ \probaM{z}{W_T \in dw} &= \frac{1}{\sqrt{2\pi T}} e^{-\frac{1}{2T}(z-w)^2} dw
		\end{split} \]
		thus we have
		\begin{equation}\label{eq:pont-brownien-pos}
			\proba{\inf_{t \in [0,T]} W_t^{z \to w} > 0} = 1 - e^{\frac{1}{2T}(z-w)^2 - \frac{1}{2T}(z+w)^2} = 1 - e^{-\frac{2}{T}zw} \, .
		\end{equation}
		For any $\alpha > 0$ and $z,w > \alpha$, we define
		\[ P_T^\alpha(z,w) \defeq \proba{\inf_{t \in [0,T]} W_t^{z \to w} \leq \alpha \, | \, \inf_{t \in [0,T]} W_t^{z \to w} > 0} = 1 - \frac{\proba{\inf_{t \in [0,T]} W_t^{z \to w} > \alpha}}{\proba{\inf_{t \in [0,T]} W_t^{z \to w} > 0}} \, . \]
		Then, using \eqref{eq:pont-brownien-pos} with $z,w,z-\alpha, w - \alpha > 0$, we can deduce
		\begin{equation}\label{eq:pont-brownien-fenetre}
			P_T^\alpha(z,w) = 1 - \frac{1 - e^{-\frac{2}{T}(z-\alpha)(w-\alpha)}}{1 - e^{-\frac{2}{T}zw}} = \frac{e^{-\frac{2}{T}(z-\alpha)(w-\alpha)} - e^{-\frac{2}{T}zw}}{1 - e^{-\frac{2}{T}zw}} \, .
		\end{equation}
		Consider the mapping $f_T : (x,y) \mapsto e^{-\frac{2}{T}xy}$. Using the mean value theorem, there is a $c \in [0,1]$ such that
		\begin{equation}\label{eq:TAF-2-var}
			\begin{split}
				f_T(z,w) - f_T(z-\alpha, w-\alpha) &= \nabla f_T \Big( (1-c) \begin{pmatrix} z \\ w \end{pmatrix} + c \begin{pmatrix} z-\alpha \\ w-\alpha \end{pmatrix} \Big) \cdot \Big( \begin{pmatrix} z \\ w \end{pmatrix} - \begin{pmatrix} z-\alpha \\ w-\alpha \end{pmatrix} \Big) \\
				&= - \frac{2 \alpha}{T} (z+w-2c\alpha)e^{-\frac{2}{T}(z-c\alpha)(w-c\alpha)} \, .
			\end{split}
		\end{equation}
		Injecting in \eqref{eq:pont-brownien-fenetre}, this yields
		\begin{equation}\label{eq:pont-brownien-fenetre-final}
			P_T^\alpha(z,w) = \frac{2 \alpha}{T} (z+w-2c\alpha)\frac{e^{-\frac{2}{T}(z-c\alpha)(w-c\alpha)}}{1 - e^{-\frac{2}{T}zw}} \leq \frac{2 \alpha}{T} (z + w) \frac{e^{-\frac{2}{T}(z-c\alpha)(w-c\alpha)}}{1 - e^{-\frac{2}{T}zw}} \, .
		\end{equation}
		In particular, if we assume $z,w \geq \lambda \alpha$ for some $\lambda > 1$, then $f_T(z,w) \leq f_T(\lambda a, \lambda a)$ and we obtain 
		\[  P_T^\alpha(z,w) \leq \frac{2 \alpha}{T} (z + w) \frac{e^{-\frac{2}{T}\alpha^2 (\lambda-c)^2}}{1 - e^{-\frac{2}{T}\alpha^2 \lambda^2}} \, . \]
		Therefore, for any $\lambda > 1$ and $a > 0$,
		\begin{equation}\label{eq:pont-meandre}
			\begin{split}
				\proba{\inf_{s \leq r \leq t} M_r \leq a, M_s > \lambda a, M_t > \lambda a} &= \esp{P_{t-s}^a(M_s,M_t) \indic{M_s,M_t \geq \lambda a}} \\
				&\leq \frac{2 a}{t-s} \frac{e^{-\frac{2}{t-s} a^2 (\lambda-c)^2}}{1 - e^{-\frac{2}{t-s} a^2 \lambda^2}} \esp{M_s + M_t} \, ,
			\end{split}
		\end{equation}
		and we compute $\esp{M_t} \leq \frac{2\sqrt{2}}{\pi} \sqrt{t} \leq \sqrt{2t}$ for $t < 1/2$ to get the desired result.
		\par On the other hand, using~\eqref{meandre-transition-0}, we write for $0<t<\frac12$
		\begin{align*}
			\proba{\mathcal{M}_t \leq a} &= \frac{2}{t^{3/2}} \int_0^a y e^{-\frac{y^2}{2t}} \int_0^y \frac{e^{-\frac{u^2}{2(1-t)}}du}{\sqrt{2\pi(1-t)}} dy \leq \frac{2}{t^{3/2}} \int_0^a y e^{-\frac{y^2}{2t}} \int_0^y \frac{du}{\sqrt{2\pi(1-t)}} dy  \\
			&\leq \frac{2 a t^{-3/2}}{\sqrt{2\pi(1-t)}} \int_0^a y e^{-\frac{y^2}{2t}} dy = \frac{4 a (1 - e^{-\frac{a^2}{2t}})}{\sqrt{2 \pi t(1-t)}} \leq \frac{4a}{\sqrt{\pi t}} \left( 1 \wedge \frac{a^2}{2t} \right) \,.
			\qedhere
		\end{align*}
	\end{proof}

	Let us mention that a process related to the meander is the $3$-dimensional Bessel process $B$. It can be defined as the solution of the SDE $\dd B_t = \dd W_t + B_t^{-1} \dd t$, or as the sum $B_t = |W_t| + L_t$ where $L$ is the local time of $W$ at $0$; it is a homogeneous Markov process that has the Brownian scaling property $(B_{\alpha t})_t \overset{d}{=} (\sqrt{\alpha} B_t)_t$. We refer to \cite{revuz2013continuous} for those results. The link between the Bessel process and the meander is given by the following result.
	
	\begin{proposal}\label{prop-meandre-scaling}
		The law $\PP^{+,T}$ of the Brownian meander on $[0,T]$ has a density with respect to $\PP^B$ the law of the three-dimensional Bessel process: if $X$ is the canonical process, we have
		\[ \PP^{+,T} (A, X_T \in dx) = \frac{1}{x} \sqrt{\frac{\pi T}{2}} \, \PP^B(A,X_T \in dx) \, . \]
		In particular, $\forall \alpha > 0, \forall s \leq T, \PP^{+,\alpha T}(X_{\alpha s} \in dx) = \PP^{+,T}(\sqrt{\alpha} X_s \in dx)$.
	\end{proposal}
	
	\begin{proof}
		The formula for the density can be found in \cite[Section 4]{imhof1984density}. Afterwards, for any positive measurable function $f$ and any $\alpha > 0$, we have
		\[ \EE^{+,\alpha T} \left[ f\Big(\frac{X_{\alpha s}}{\sqrt{\alpha}} \Big) \right] = \EE^B \left[ \frac{1}{X_{\alpha T}} \sqrt{\frac{\pi \alpha T}{2}} f\Big(\frac{X_{\alpha s}}{\sqrt{\alpha}} \Big) \right] = \sqrt{\frac{\pi}{2}} \EE^B \left[ \frac{\sqrt{T}}{X_T} f(X_s) \right] = \EE^{+,T} \left[ f(X_s) \right] \, . \qedhere \]
	\end{proof}

	\section{Coupling of Brownian meander, a three-dimensional Bessel process and a Brownian excursion}\label{appendix-Construction}
	
	In this section we will expand on the way we can construct our different processes to have the almost sure results of Theorems \ref{th-1/6} and \ref{th-1/9}. In particular we want the following result:
	\[ \frac{1}{n^{1/6}} \sum_{-u n^{1/3}}^{vn^{1/3}} \omega_z \xrightarrow[n \to \infty]{a.s.} X^{(1)}_u + X^{(2)}_v \quad \text{and} \quad n^{1/18} (X_{u_* + \frac{u}{n^{1/9}}} - X_{u_*}) \xrightarrow[n \to \infty]{a.s.} \mathbf{B}_u \, . \]
	
	Skorokhod's embedding theorem (Theorem \ref{th-skorokhod}) allows us to sample the Brownian motions $X^{(i)}, i = 1,2$ to get a new environment $\hat\omega^{(n)}$ to obtain the first convergence. Thus we must find how we can couple both processes $X^{(i)}$ to the processes $\mathbf{B},\mathbf{Y}$ in Theorem \ref{th-1/9},
	that is we need to prove Proposition \ref{prop-couplage}.
	This is based on two intermediate results, Lemmas~\ref{lem:meandre-excursion} and~\ref{lem:bessel-excursion} below, which couple a meander, resp.\ a Bessel-$3$ process, to a Brownian excursion.
	
	\begin{lemma}[{\cite[Theorem 2.3]{Bertoin1994PathTC}}]\label{lem:meandre-excursion}
		Let $\mathbf{e}$ be a standard Brownian excursion and $U$ a uniform variable on $[0,1]$. Then, the process $M_t = \mathbf{e}_t \indic{t \leq U} + (\mathbf{e}_U + \mathbf{e}_{1-(t-U)}) \indic{t > U}$ is a Brownian meander on $[0,1]$. In particular, there exists a coupling of the Brownian meander $M$ and the Brownian excursion $\mathbf{e}$ on $[0,1]$ such that $M_t = \mathbf{e}_t$ if $t \leq U$.
	\end{lemma}
	
	\begin{lemma}\label{lem:bessel-excursion}
		For any $T \in [0,1]$, 
		There exists a coupling of the Brownian excursion $\mathbf{e}$ on $[0,1]$ and the three-dimensional Bessel process $\mathbf{B}$ such that there is a positive $\eps(\omega)$ for which we have $\mathbf{B}_t =  \mathbf{e}_{t}$ for any $t \in [0,\eps(\omega)]$.
	\end{lemma}
	
	\begin{proof}
		It is known (see for example \cite[p79]{ito-mckean}) that the Brownian excursion can be decomposed into two Bessel bridges of duration $\frac12$ joining at a point $V$ whose law has density $\frac{16}{\sqrt{2\pi}} v^2 e^{-2v^2}$. Thus we only need to define a coupling between a $3d$-Bessel process $\mathbf{B}$ and a $3d$-Bessel bridge $\mathbf{B}'$ with duration $\frac12$ and endpoint $V$. We use the fact that both processes can be realized by the modulus of a three-dimensional Brownian motion.
		
		\smallskip
		Consider two independent, three-dimensional Brownian bridges $X$ and $Y$ of duration $1/2$, such that $X_0 = x \in \RR^3$ (resp. $Y_0 = y \in \RR^3$) and $X_\frac12 = Y_\frac12 = 0$. Denote by $\tau$ the first time $X$ and $Y$ have the same modulus: $\tau \defeq \inf \mathset{0 \leq t \leq \frac12 \, : \, |X_t| = |Y_t|}$.
		We have the following result.
		
		\begin{lemma}\label{lem:temps-couplage}
			Almost surely, there exists $\eps(\omega) > 0$ such that $\tau \leq \frac12 - \eps(\omega)$.
		\end{lemma}
		
		Using this lemma, we can conclude the construction of the coupling.
		After time $\tau$, we define a coupling by taking the trajectory of $X$ between $\tau$ and $\tfrac12$ and plugging it at $Y_\tau$ after a rotation:
		\[ \text{write } X_t = |X_t| e^{i \theta^X_t}, Y_t = |Y_t| e^{i \theta^Y_t} \text{ and define } \hat Y_t = \begin{dcases} Y_t \quad \text{if} \quad t \leq \tau \,,\\ |X_t| e^{i \theta^X_t + i (\theta^Y_\tau - \theta^X_\tau)} \quad \text{if} \quad \tau < t \leq \tfrac12 \,.\end{dcases} \]
		The new process $\hat{Y}$ is such that for every $t \in [\tau, \frac12]$, we have $|X_t| = |\hat{Y}_t|$.
		Recall that the Brownian bridge is a diffusion process (as the solution to an SDE), thus is Markovian, and $\tau$ is a stopping time for both processes $X$ and $Y$. It follows that $\hat Y$ is a Brownian bridge between $y$ and $0$.
		
		To create the coupling between the two Bessel processes $\mathbf{B}$ and $\mathbf{B}'$, we choose the starting points $x$ and $y$ so that they respectively correspond to $W_\frac12$ (with $W$ a $3d$-Brownian motion) and a uniform variable on the sphere centered at $0$ of radius $V$. Then the processes $\mathbf{B}_t = |X_{\frac12 - t}|$ and $\mathbf{B}'_t = |\hat{Y}_{\frac12 - t}|$ are Bessel processes starting at $0$ that coincide on $[0,\frac12 - \tau]$ and such that $\mathbf{B}'_\frac12 = V$. In particular, the Bessel process $\mathbf{B}$ and the Brownian excursion $\mathbf{e}$ coincide on $[0,\frac12 - \tau]$.
	\end{proof}

	\begin{proof}[Proof of Lemma \ref{lem:temps-couplage}]
		On $[0,\tfrac12]$, consider $\mathbf{B}$ a 3-dimensional Bessel process starting at $0$ and $\mathbf{e}$ the Brownian excursion, which is a Bessel bridge of duration $\tfrac12$ starting at $0$ and ending at $V$. We define $I_{s,t} \defeq \mathset{\forall r \in (s,t), \mathbf{e}_r \neq \mathbf{B}_r}$ the event on which $\mathbf{e}$ and $\mathbf{B}$ never intersect between $0$ and $t$ (with the exception of $0$).
		From \cite[(3.1)]{imhof1984density}, we have $\probaM{x}{A, \mathbf{B}_t \in dz} = \frac{z}{x} \probaM{x}{A, W_t \in dz, H_0 > t}$, where $W$ is a Brownian motion and $H_0$ its first hitting time of $0$. Then for any $\eps > 0$, conditioning on the values of $(\mathbf{e}_\eps,\mathbf{B}_\eps)$ and $(V,\mathbf{B}_t)$, we can write
		\[ \proba{I_{0,t}} \leq \esp{ \proba{I_{\eps,t} \, | \, \mathbf{e}_\eps, \mathbf{B}_\eps} } \leq \esp{ \frac{\mathbf{e}_t \mathbf{B}_t}{\mathbf{e}_\eps \mathbf{B}_\eps} \proba{\mathscr{I}_{\mathbf{e}_\eps \to \mathbf{e}_t}^{\mathbf{B}_\eps \to \mathbf{B}_t}(t-\eps)} } \, , \]
		where we have defined
		\[ \mathscr{I}_{x \to y}^{a \to b}(T) \defeq \mathset{\forall r \in (0,T), W^{x \to y}_{T}(r) > 0, W^{a \to b}_{T}(r) > 0,  W^{x \to y}_{T}(r) \neq W^{a \to b}_{T}(r)} \, , \]
		in which $W^{a \to b}_T$ is a Brownian bridge $a \to b$ of duration $T$ (resp. for $x \to y$). We are interested in taking $t = 1/2$, but this result could be used for any fixed $t > 0$, in the sense that the Bessel process and the Brownian excursion almost surely cross each-other on $]0,t]$ for any fixed $t$.
		Take a positive $C > 0$ to be chosen later (we will choose $C = \eps^{-1/8}$). Then, we first get a bound using Cauchy-Schwartz inequality twice:
		\[ \esp{ \frac{| \mathbf{e}_t \mathbf{B}_t|}{\mathbf{e}_\eps \mathbf{B}_\eps} \proba{\mathscr{I}_{\mathbf{e}_\eps \to V}^{\mathbf{B}_\eps \to \mathbf{B}_t}(t-\eps)} \indic{\mathbf{B}_t \vee V > C}} \leq \esp{\frac{1}{(\mathbf{e}_\eps \mathbf{B}_\eps)^2}}^{\frac12} \esp{(\mathbf{e}_t \mathbf{B}_t)^4}^{\frac14} \proba{\mathbf{B}_t \vee V > C}^{\frac14} \, . \]
		
		Since $\eps < t$ and $\mathbf{e},\mathbf{B}$ are independent, we have $\esp{(\mathbf{e}_t \mathbf{B}_t)^4} \leq c(t)$ and
		\[ \esp{\frac{1}{(\mathbf{e}_\eps \mathbf{B}_\eps)^2}} \leq \frac{2}{\sqrt{\pi}} \Gamma(\frac32) (1-\eps)^{-\frac32} \eps^{-3} \int_{\RR_+^2} e^{-\frac{x^2}{2\eps}} e^{-\frac{y^2}{2\eps(1-\eps)}} dx dy \leq (1-\eps)^{-1} \eps^{-2} \, , \]
		where we used the transition probabilities for the Bessel process \cite[VI §3 Prop. 3.1]{revuz2013continuous}, the Brownian excursion \cite[Section 2.9 (3a)]{ito-mckean} and $\Gamma(\frac32) = \sqrt{\pi}/2$. Finally we compute $\proba{\mathbf{B}_t \vee V > C} \leq e^{-C^2/t^{1/6}}$ to get
		\begin{equation}\label{eq:majore-intersection-inf-C}
			\esp{ \frac{| \mathbf{e}_t \mathbf{B}_t|}{\mathbf{e}_\eps \mathbf{B}_\eps} \proba{\mathscr{I}_{\mathbf{e}_\eps \to V}^{\mathbf{B}_\eps \to \mathbf{B}_t}(t-\eps)} \indic{\mathbf{B}_t \vee \mathbf{e}_t > C}} \leq c_t (1-\eps)^{-\frac12} \eps^{-2} e^{-\frac{C^2}{t^{1/6}}} \, .
		\end{equation}
		On the other hand,
		\begin{equation}\label{eq:separation-intersection-leq}
			\esp{ \frac{| \mathbf{e}_t \mathbf{B}_t|}{\mathbf{e}_\eps \mathbf{B}_\eps} \proba{\mathscr{I}_{\mathbf{e}_\eps \to \mathbf{e}_t}^{\mathbf{B}_\eps \to \mathbf{B}_t}(t-\eps)} \indic{\mathbf{B}_t \vee \mathbf{e}_t \leq C}} \leq \esp{ \frac{C^2}{\mathbf{e}_\eps \mathbf{B}_\eps} \proba{\mathscr{I}_{\mathbf{e}_\eps \to \mathbf{e}_t}^{\mathbf{B}_\eps \to \mathbf{B}_t}(t-\eps)}} \, .
		\end{equation}
		We will use the following lemma to get a bound on $\proba{\mathscr{I}_{\mathbf{e}_\eps \to \mathbf{e}_t}^{\mathbf{B}_\eps \to \mathbf{B}_t}(t-\eps)}$.
		
		\begin{lemma}\label{lem:proba-non-intersection}
			For any $T > 0$, there is a $C_T > 0$ such that for any $x,y,a,b > 0$,
			\begin{equation}
				\proba{\mathscr{I}_{x \to y}^{a \to b}(T)} \leq C_T (x^2 + a^2)^2 (y^2 + b^2)^2 \, .
			\end{equation}
		\end{lemma}
		
		Thus, using Lemma \ref{lem:proba-non-intersection} in \eqref{eq:separation-intersection-leq}, we have the upper bound
		\[ \esp{ \frac{| \mathbf{e}_t \mathbf{B}_t|}{\mathbf{e}_\eps \mathbf{B}_\eps} \proba{\mathscr{I}_{\mathbf{e}_\eps \to \mathbf{e}_t}^{\mathbf{B}_\eps \to \mathbf{B}_t}(t-\eps)} \indic{\mathbf{B}_t \vee \mathbf{e}_t \leq C}} \leq C^6 C_{t,\eps} \esp{ \frac{1}{\mathbf{e}_\eps \mathbf{B}_\eps}\Big( (\mathbf{e}_\eps)^2 + (\mathbf{B}_\eps)^2 \Big)^2} \, , \]
		and we compute
		\[ \begin{split}
			\esp{\frac{\Big( (\mathbf{e}_\eps)^2 + (\mathbf{B}_\eps)^2 \Big)^2}{\mathbf{e}_\eps \mathbf{B}_\eps}} &= \frac{2}{\sqrt{\pi}} \Gamma(\frac32) (1-\eps)^{-\frac32} \eps^{-3} \int_{\RR_+^2} (x^2 + y^2)^2 xy e^{-\frac{x^2}{2\eps}} e^{-\frac{y^2}{2\eps(1-\eps)}} dx dy\\
			&\leq (cst.) \eps^{-3} \int_{\RR_+^2} \eps^2 (u^2 + v^2)^2 uv \eps e^{-\frac{u^2}{2}} e^{-\frac{v^2}{2}} \eps du dv \leq (cst.) \eps \, .
		\end{split} \]
		to get
		\begin{equation}\label{eq:majore-intersection-sup-C}
			\esp{ \frac{| \mathbf{e}_t \mathbf{B}_t|}{\mathbf{e}_\eps \mathbf{B}_\eps} \proba{\mathscr{I}_{\mathbf{e}_\eps \to \mathbf{e}_t}^{\mathbf{B}_\eps \to \mathbf{B}_t}(t-\eps)} \indic{\mathbf{B}_t \vee \mathbf{e}_t \leq C}} \leq C^6 C_t \eps \, .
		\end{equation}
		Thus, assembling \eqref{eq:majore-intersection-inf-C} and \eqref{eq:majore-intersection-sup-C} while taking $C = \eps^{-1/8}$, for any $t > 0$ we then have
		\[ \proba{I_{0,t}} \leq C_t \left( \eps^{-1} e^{-\frac{C^2}{4(t-\eps)^{1/6}}} + C^6 \eps \right) \leq C_t \left( \eps^{-1} \exp \Big(-\frac{\eps^{-1/4}}{2}\Big) + \eps^{1/4} \right) \xrightarrow[\eps \to 0]{} 0 \, . \]
		This means that $\proba{I_{0,t}} = 0$ and in particular, taking $t = \frac12$, one can almost-surely find a positive $\eps$ such that Lemma \ref{lem:temps-couplage} is true.
	\end{proof}

	\begin{proof}[Proof of Lemma \ref{lem:proba-non-intersection}]
		We can assume $0 < x < a$ and $0 < y < b$ (otherwise the probability is zero), then we have
		\begin{equation}\label{eq:probaBB}
			\proba{\mathscr{I}_{x \to y}^{a \to b}(T)}  = \proba{\forall r \in [0,T], 0 < W_T^{x \rightarrow y}(r) < W_T^{a \rightarrow b}(r)}.
		\end{equation}
		Observe that \eqref{eq:probaBB} is exactly the probability for the Brownian bridge $W_T^{x \to y, a \to b} \defeq (W_T^{x \rightarrow y}, W_T^{a \rightarrow b})$ to stay in the cone $\mathscr{C} \defeq \mathset{(x,y) \in \RR^2 \, : \, 0 \leq x \leq y}$ for a time $T$, meaning
		\[ \proba{\mathscr{I}_{x \to y}^{a \to b}(T)} = \proba{\forall t \in [0,T], W_T^{x \to y, a \to b}(t) \in \mathscr{C}}. \]
		The isotropy of Brownian motion allows us to consider instead $\hat {\mathscr{C}} \defeq \mathset{r e^{i\theta}, 0 \leq \theta \leq \frac{\pi}{4}}$.
		
		\begin{lemma}\label{lem:BB-cone}
			Let $W^{z \to z'}$ be a two dimensional Brownian bridge from $z$ to $z'$. Then, there is a positive $C_T$ such that uniformly as $|z| \to 0$ we have
			\[ \proba{\forall t \in [0,T], W^{z \to z'}_t \in \hat{\mathscr{C}}} = (1+\bar{o}(1)) C_T |z|^4 |z'|^4 \sin \left( 4  \arg z \right)\sin \left( 4 \arg z' \right) \, . \]
		\end{lemma}
		
		\begin{proof}
			Recall that we identify $\RR^2$ with $\CC$, by writing $W$ for a standard two-dimensional Brownian motion, we have
			\[ \begin{split}
				\proba{\forall t \in [0,T], W^{z \to z'}_t \in \hat{\mathscr{C}}} &= \lim_{\eta \to 0} \frac{\probaM{x}{\forall t \in [0,T], W_t \in \hat{\mathscr{C}}, W_T \in B(z',\eta)}}{\probaM{z}{W_T \in B(z',\eta)}} \\
				&= \lim_{\eta \to 0} \left(C(T) \eta^2 e^{-|z'|^2/2T} \right)^{-1} \int_{B(z',\eta)} K_T^{\hat{\mathscr{C}}}(z,w) dw \, ,
			\end{split} \]
			where $K_T^{\hat {\mathscr{C}}}(z,w)$ is the heat kernel killed on exiting $\hat{\mathscr{C}}$ and $B(z,r)$ is the ball of radius $r$ centered at $z$.
			
			The key ingredient is the following statement, which is a consequence of \cite[Lemma 18 - (32)]{denisov-waschel}: as $\delta \to 0$, uniformly in $|z| \leq \delta \sqrt{T}, |w| \leq \sqrt{T/\delta}$, we have
			\[ K_T^{\hat{\mathscr{C}}}(z,w) \sim \frac{\chi_0}{T^{5}} e^{-|w|^2/2T} u(w) u(z) \quad \text{for some $\chi_0 > 0$.} \]
			where $u(re^{i\theta}) \defeq r^4 \sin(4\theta)$ (this expression is given in \cite[(3)]{denisov-waschel}). This result is also stated in \cite[Corollary 1]{banuelos_brownian_1997}. In particular, as $|z| \to 0$,
			\[ \begin{split}
				\proba{\forall t \in [0,T], W^{z \to z'}_t \in \hat{\mathscr{C}}} &\sim \lim_{\eta \to 0} \left(C(T) \eta^2 e^{-|z'|^2/2T} \right)^{-1} \int_{B(z',\eta)} \frac{\chi_0}{T^{5}} e^{-|w|^2/2T} u(w) u(z) dw \\
				&= \lim_{\eta \to 0} e^{|z'|^2/2T} \frac{\chi_0}{T^{5}} e^{-|z'|^2/2T}  u(z') u(z) \frac{\text{Vol}(B(z',\eta))}{C(T) \eta^2} (1 + h(\eta)) \, ,
			\end{split} \]
			with $h(\eta) \to 0$ and $\text{Vol}(B(z',\eta)) = \pi \eta^2$, leading us to the formula of Lemma \ref{lem:BB-cone}.
		\end{proof}
		
		\begin{remark}
			We could also use the fact that $\hat{\mathscr{C}}$ is the Weyl chamber $B_2$, thus we can use results from \cite[§5.3]{Grabiner1997BrownianMI} after a time scaling by $\eps$ to have that the probability in \eqref{eq:separation-intersection-leq} is of order $(\eps/t)^2$, which is ultimately what we proved.
		\end{remark}
		
		Thus, we proved Lemma \ref{lem:proba-non-intersection} by injecting $z = (x,a)$ and $z' = (y,b)$.
	\end{proof}

	Assembling Lemmas \ref{lem:meandre-excursion} and \ref{lem:bessel-excursion} yields that one can do a coupling of the Brownian meander $\mathcal{M}$ and the three-dimensional Bessel process $\mathbf{B}$ such that almost surely, there is a positive time $\sigma$ for which $\mathcal{M}_t = \mathbf{B}_t$ on $[0,\sigma]$, thus proving Proposition \ref{prop-couplage} using \eqref{eq:X-meandre}.

	\subsection*{Acknowledgements}
	
	The author would like to thank his PhD advisors Quentin Berger and Julien Poisat for their continual help, as well as Pierre Tarrago for his proof of Lemma \ref{lem:BB-cone}.

\end{appendix}

\nocite{*}
\printbibliography[heading=bibintoc]

\end{document}